\documentclass[12pt,reqno]{amsart}
\usepackage{amsmath}
\usepackage{amsfonts}
\usepackage{amssymb}
\usepackage{amsthm}
\usepackage{amscd}
\usepackage{amsbsy}
\usepackage[usenames,dvipsnames,svgnames,x11names,hyperref]{xcolor}
\usepackage{geometry}
\usepackage{graphicx}
\usepackage{hyperref}
\hypersetup{backref=true,
pagebackref=true,
hyperindex=true,
colorlinks=true,
breaklinks=true,
urlcolor=NavyBlue,
linkcolor=Fuchsia,
bookmarks=true,
bookmarksopen=false,
filecolor=black,
citecolor=ForestGreen,
linkbordercolor=red
}
\newtheorem{theorem}{Theorem}[section]
\theoremstyle{plain}

\newtheorem{lemma}[theorem]{Lemma}
\newtheorem{proposition}[theorem]{Proposition}
\newtheorem{remark}[theorem]{Remark}
\newtheorem{example}[theorem]{Example}
\theoremstyle{remark}
\theoremstyle{plain}
\numberwithin{equation}{section}
\newtheorem{theorema}{Theorem}[section]

\allowdisplaybreaks
\AtBeginDocument{\hypersetup{pdfborder={0 0 0.01}}}

\newcommand{\Ent}{\operatorname{Ent}}
\renewcommand{\Re}{\operatorname{Re}}

\begin{document}
\title[Stability of logarithmic and cubic Sobolev inequalities on cycles]{Sharp uniform stability of the logarithmic Sobolev inequality on finite cycles and stability of the cubic Sobolev inequality}
\author{Lu Chen}
\address{Lu Chen: Key Laboratory of Algebraic Lie Theory and Analysis of Ministry of Education, School of Mathematics and Statistics, Beijing Institute of Technology, Beijing 100081, P. R. China}
\email{chenlu5818804@163.com}
\author{Nguyen Lam}
\address{Nguyen Lam: School of Science and the Environment, Grenfell Campus, Memorial
University of Newfoundland, Corner Brook, NL A2H5G4, Canada}
\email{nlam@mun.ca}
\author{Guozhen Lu}
\address{Guozhen Lu: Department of Mathematics, University of Connecticut, Storrs, CT
06269, USA}
\email{guozhen.lu@uconn.edu}
\date{\today}
\begin{abstract}
We determine the sharp uniform constant in the quartic stability estimate for the logarithmic Sobolev inequality on finite cycles. We equip $C_{n}=\mathbb{Z}/n\mathbb{Z}$ with the uniform probability measure, and we let $\mathcal{E}_{n}$ be the Dirichlet form of the simple random walk on $C_{n}$, normalized so that its spectral gap is one. Frank and Ivanisvili recently proved the sharp logarithmic Sobolev inequality $2\mathcal{E}_{n}(u)\geq \mathrm{Ent}(u^{2})$ for all $n\geq 4$. We prove that $2\mathcal{E}_{n}(u)-\mathrm{Ent}(u^{2})\geq \frac{1}{3}\Vert u-1\Vert _{2}^{4}$ for all $n\geq 4$ and all nonnegative $u$ with $\Vert u\Vert _{2}=1$. This is the sharp form, uniformly in $n$, of the quartic stability estimate of Xie and Zhang, who obtained the constant $\frac{1}{12}$. The same estimate holds on the circle and for the word-length energy on $C_{n}$. Among these models, the four-cycle is the only extremal one: on $C_{4}$ the constant $\frac{1}{3}$ is sharp but not attained, while on $C_{n}$ with $n\geq 5$ and on the circle the sharp constant is at least $\frac{1}{3}+c$, where $c>0$ does not depend on $n$. We also give explicit upper bounds for these sharp constants. Our second main result is a stability estimate for the cubic Sobolev inequality of Frank and Ivanisvili, which is the main step in their proof. It gives the quantitative version of this inequality expected by Frank and Ivanisvili, with the optimal powers of $\Vert u-1\Vert _{2}$: four on $C_{n}$ with $n\geq 5$ and on the circle, and six on $C_{4}$. Two key new ingredients of the proofs are a bound for the first Fourier coefficient of $u$ in terms of its mean, which holds because $u$ is nonnegative, and a sharp stability estimate for Gross's two-point logarithmic Sobolev inequality, with sharp constants in both the fourth and the sixth powers of the distance.
\end{abstract}
\maketitle
\section{Introduction}

Let $n\geq2$ be an integer and let $C_{n}=\mathbb{Z}/n\mathbb{Z}$ be the cycle with $n$ vertices. We equip it with the uniform probability measure. For a real function $f$ on $C_{n}$, we write
\[
\langle f\rangle =\frac{1}{n}\sum_{j=0}^{n-1}f_{j},\qquad \left\Vert f\right\Vert _{2}^{2}=\langle f^{2}\rangle ,
\]
where the indices are always taken modulo $n$. The Dirichlet form of the simple random walk on $C_{n}$ and the logarithmic Sobolev constant are given by
\begin{equation*}
\mathcal{D}_{n}(f)=\frac{1}{2n}\sum_{j=0}^{n-1}\left( f_{j+1}-f_{j}\right)
^{2},\qquad \alpha _{n}=\inf_{\Ent(f^{2})>0}\frac{\mathcal{D}_{n}(f)}{\Ent(f^{2})},
\end{equation*}
where the entropy is defined by
\begin{equation*}
\Ent(g)=\langle g\log g\rangle -\langle g\rangle \log \langle g\rangle ,\qquad g\geq 0,
\end{equation*}
with the convention $0\log 0=0$. By Gross's theorem \cite{Gross75}, a logarithmic Sobolev inequality is equivalent to the hypercontractivity of the associated semigroup, see also the survey \cite{Gro06}. The spectral gap of $\mathcal{D}_{n}$ is $\lambda _{n}=1-\cos \frac{2\pi }{n}$. It is well known that the general comparison $\alpha _{n}\leq \frac{\lambda _{n}}{2}$ always holds, see \cite{DS96}. Diaconis and Saloff-Coste \cite{DS96} also showed that $\alpha _{n}$ is of the same order as $\lambda _{n}$. Their work led to the question whether
\begin{equation}
\alpha _{n}=\frac{\lambda _{n}}{2},\qquad n\geq 4,  \label{LSIconstant}
\end{equation}
is true, see \cite[Remark (b) after Theorem 1]{FI26}. Chen and Sheu proved the identity \eqref{LSIconstant} for the even cycles \cite{CS03}. Chen, Liu and Saloff-Coste also proved it for the five-cycle, and they listed the other odd cases as an open problem \cite[(Q2) on p.~242]{CLS08}. Faust and Fawzi \cite{FF24} proved it for the odd cycles with $7\leq n\leq 21$ by sum-of-squares methods. We also note that \eqref{LSIconstant} fails when $n=3$, see \cite{CLS08,DS96}. Very recently, Frank and Ivanisvili \cite{FI26} solved the problem completely for all $n\geq 4$. Shortly afterwards, Yao \cite{Yao26} and, independently, Xie and Zhang \cite{XZ26} proved the sharp logarithmic Sobolev inequality for the word-length energy, see \eqref{wordenergy} below. This inequality is stronger than \eqref{LSIconstant}, so it gives a different proof of \eqref{LSIconstant}. Throughout the paper, we use the normalization
\begin{equation}
\mathcal{E}_{n}(f)=\frac{2}{\mu _{n}}\mathcal{D}_{n}(f)=\frac{1}{n\mu _{n}}
\sum_{j=0}^{n-1}\left( f_{j+1}-f_{j}\right) ^{2},\qquad \mu _{n}=2\left(
1-\cos \frac{2\pi }{n}\right) ,  \label{energy}
\end{equation}
so that the spectral gap of $\mathcal{E}_{n}$ is equal to one. We call $\mathcal{E}_{n}$ the nearest-neighbor energy. With this normalization, \eqref{LSIconstant} reads $2\mathcal{E}_{n}(u)\geq \Ent(u^{2})$. The key idea in \cite{FI26} is to compare the function $w^{2}\log w^{2}$ with its cubic Taylor polynomial at $w=1$. This reduces \eqref{LSIconstant} to the following stronger inequality.

\begin{theorema}[Frank--Ivanisvili \cite{FI26}]
\label{TA}Let $n\geq 4$. For every $u\geq 0$ on $C_{n}$ with $\left\Vert
u\right\Vert _{2}=1$, we have
\begin{equation}
\mathcal{E}_{n}(u)\geq \frac{1}{3}\left\langle \left( u-1\right) ^{2}\left(
u+2\right) \right\rangle .  \label{cubic}
\end{equation}
As a result, $2\mathcal{E}_{n}(u)\geq \Ent(u^{2})$, that is, \eqref{LSIconstant} holds. The same inequalities hold on the circle $\mathbb{T}=\mathbb{R}/\mathbb{Z}$ with the heat energy $\mathcal{E}_{\infty }(u)=\frac{1}{4\pi ^{2}}\int_{0}^{1}\left\vert u^{\prime }\right\vert ^{2}dx$.
\end{theorema}

For the circle, see \cite[Remark (d) after Theorem 1 and Remark (b) after Theorem 2]{FI26}. The logarithmic Sobolev inequality on the circle goes back to Rothaus \cite{Rot80} and Weissler \cite{Wei80}.

When the sharp constant is known, it is natural to study the \emph{stability} of the inequality. That is, we ask whether the deficit controls the distance to the set of extremal functions. For the Sobolev inequality, this question goes back to Brezis and Lieb \cite{BL85}, and it was answered by Bianchi and Egnell \cite{BE91}. Since then, the stability of functional inequalities has become a very active research field, see for example \cite{BDNN25,Car17,CF13,CFW13,CFMP09,FN19,FZ22,VHN16,VHN19}. Moreover, stability estimates with explicit or sharp constants have attracted much attention recently.   We refer the interested reader to \cite{BDS24,CLT23,CLT25,DEFFL25,FIL16,Frank22, Kon23} for the Sobolev and logarithmic Sobolev inequalities, and to \cite{CLT24,CLT26} for the Hardy--Littlewood--Sobolev inequalities and the Sobolev inequalities of higher and fractional order. Stability results for the Caffarelli--Kohn--Nirenberg inequalities can be found in \cite{CFLL24,DLLN26, LLLS25, LLLS}. For the Poincar\'{e} inequalities on Euclidean balls and with Gaussian or weighted Gaussian measures, and for the logarithmic Sobolev inequalities with weighted Gaussian measures, see \cite{DLLN26,LL26,LLR25,LLR26} and references therein.  

For the logarithmic Sobolev inequality on the cycle, the extremal functions are exactly the constant functions, and the stability problem is degenerate. More precisely, the deficit vanishes to \emph{fourth} order in the direction of the first harmonics. The reason is that the second variation of the deficit is zero on the two-dimensional eigenspace of the spectral gap, and that the third-order term also vanishes there when $n\geq 4$. Therefore, the natural stability estimate is of quartic type. In this direction, Xie and Zhang \cite{XZ26} proved the following result.

\begin{theorema}[Xie--Zhang \cite{XZ26}]
\label{TB}Let $n\geq 4$. For every $u\geq 0$ on $C_{n}$ with $\left\Vert
u\right\Vert _{2}=1$, we have
\begin{equation*}
2\mathcal{E}_{\mathrm{w},n}(u)-\Ent(u^{2})\geq \frac{1}{12}\left\Vert u-1\right\Vert
_{2}^{4},
\end{equation*}
where $\mathcal{E}_{\mathrm{w},n}$ is the word-length energy defined in \eqref{wordenergy} below. The power four is optimal.
\end{theorema}

Xie and Zhang also showed that $\mathcal{E}_{\mathrm{w},n}\leq \mathcal{E}_{n}$, see \cite[(1.4)]{XZ26}. Therefore, the same estimate holds for the nearest-neighbor energy $\mathcal{E}_{n}$. They also obtained the constant $\frac{1}{12k}$ on the product of $k$ copies of $C_{n}$, and they showed that the order $k^{-1}$ is optimal.

Theorem \ref{TB} leaves the open question of the sharp stability constant. That is, it does not give the largest constant $C$ such that $2\mathcal{E}_{n}(u)-\Ent(u^{2})\geq C\left\Vert u-1\right\Vert _{2}^{4}$ for all $u\geq 0$ with $\left\Vert u\right\Vert _{2}=1$ and all $n\geq 4$ at the same time. The main purpose of this paper is to answer this question.

Now we describe our main results. In the main results, $n\in \mathcal{I}=\{4,5,6,\ldots \}\cup \{\infty \}$. Here $n=\infty $ means the circle $\mathbb{T}$ with the Lebesgue probability measure and the heat energy $\mathcal{E}_{\infty }$. We also set $\mathcal{J}=\mathcal{I}\setminus \{4\}$. The admissible classes are
\begin{equation*}
\mathcal{A}_{n}=\left\{ u:C_{n}\rightarrow \lbrack 0,\infty ):\left\Vert
u\right\Vert _{2}=1\right\} ,\qquad \mathcal{A}_{\infty }=\left\{ u\in
H^{1}(\mathbb{T}):u\geq 0\text{ a.e.},\ \left\Vert u\right\Vert
_{2}=1\right\} .
\end{equation*}
For $u\in \mathcal{A}_{n}$, we define the distance to the constant function and the \emph{logarithmic deficit} by
\begin{equation*}
d_{n}(u)=\left\Vert u-1\right\Vert _{2},\qquad \Delta _{n}(u)=2\mathcal{E}
_{n}(u)-\Ent(u^{2}).
\end{equation*}
Since $\left\langle u^{2}\right\rangle =1$, we have $\Ent(u^{2})=\left\langle u^{2}\log u^{2}\right\rangle $.
For finite $n$, we also consider the word-length energy and its deficit
\begin{equation}
\mathcal{E}_{\mathrm{w},n}(u)=\sum_{k=0}^{n-1}\psi _{n}(k)\left\vert
\widehat{u}(k)\right\vert ^{2},\qquad \psi _{n}(k)=\min (k,n-k),\qquad \Delta
_{\mathrm{w},n}(u)=2\mathcal{E}_{\mathrm{w},n}(u)-\Ent(u^{2}),
\label{wordenergy}
\end{equation}
where $\widehat{u}(k)=\frac{1}{n}\sum_{j=0}^{n-1}u(j)e^{-2\pi ijk/n}$ are the Fourier coefficients. The word-length energy is the Dirichlet form of the Poisson semigroup on $C_{n}$, see \cite{JPPP17,XZ26}. Its spectral gap is also equal to one, and $\mathcal{E}_{\mathrm{w},n}\leq \mathcal{E}_{n}$. For $n\geq 4$, the two energies are the same only when $n=4$. The sharp hypercontractivity of this semigroup, which is equivalent to the sharp logarithmic Sobolev inequality $2\mathcal{E}_{\mathrm{w},n}(u)\geq \Ent(u^{2})$, was conjectured by Beckner, Janson and Jerison \cite{BJJ83}, who proved it for $n=4$. Partial results were obtained in \cite{JPPP17} and the references therein. Yao first proved it for $n=2^{k}$ and $n=3\cdot 2^{k}$ \cite{Yao25}. The case of all $n\geq 4$ is the result of Yao \cite{Yao26} and of Xie and Zhang \cite{XZ26} mentioned above. On the circle, the corresponding result for the Poisson semigroup is due to Weissler \cite{Wei80}. We denote by $\mathcal{C}_{n}$, $n\in \mathcal{I}$, and by $\mathcal{C}_{n}^{\mathrm{w}}$, $4\leq n<\infty $, the sharp constants in the quartic stability estimates for $\Delta _{n}$ and $\Delta _{\mathrm{w},n}$, that is,
\begin{equation}
\mathcal{C}_{n}=\inf_{u\in \mathcal{A}_{n}\setminus \{1\}}\frac{\Delta
_{n}(u)}{d_{n}(u)^{4}},\qquad \mathcal{C}_{n}^{\mathrm{w}}=\inf_{u\in \mathcal{A}
_{n}\setminus \{1\}}\frac{\Delta _{\mathrm{w},n}(u)}{d_{n}(u)^{4}}.  \label{Cdef}
\end{equation}
For $n=\infty $, only $\mathcal{C}_{\infty }$ is defined. Our first main result is the following.

\begin{theorem}[Sharp uniform stability of the logarithmic Sobolev inequality]
\label{Tlog}Let $n\in \mathcal{I}$ and $u\in \mathcal{A}_{n}$. Then
\begin{align}
\Delta _{n}(u)& \geq \frac{1}{3}d_{n}(u)^{4}+\frac{1}{45}d_{n}(u)^{6},\qquad
n\in \mathcal{J},  \label{log4} \\
\Delta _{4}(u)& \geq \frac{1}{3}d_{4}(u)^{4}+\frac{1}{20}d_{4}(u)^{6},\qquad
n=4.  \label{log4cycle}
\end{align}
Moreover, for every finite $n\geq 4$, we have
\begin{equation}
\Delta _{\mathrm{w},n}(u)\geq \frac{1}{3}d_{n}(u)^{4}.  \label{logword}
\end{equation}
On $C_{4}$, the inequality \eqref{log4cycle} is strict for every nonconstant $u$. The constant $\frac{1}{3}$ is sharp uniformly in $n$, that is, it is the largest constant that works for all $n\in \mathcal{I}$ at the same time. More precisely,
\begin{equation}
\inf_{n\in \mathcal{I}}\mathcal{C}_{n}=\mathcal{C}_{4}=\frac{1}{3},  \label{sharp13}
\end{equation}
and the infimum in the definition of $\mathcal{C}_{4}$ is not attained. In fact, every minimizing sequence converges to the constant function $1$ in $L^{2}(C_{4})$, and we have
\begin{equation}
d_{4}(u)^{2}\leq 20\left( \frac{\Delta _{4}(u)}{d_{4}(u)^{4}}-\frac{1}{3}
\right) ,\qquad u\in \mathcal{A}_{4}\setminus \{1\}.  \label{minimizingintro}
\end{equation}
The power four is optimal in \eqref{log4}--\eqref{logword}. In addition, for every finite $n\geq 4$, we have
\begin{equation}
\Delta _{n}(u)\geq \frac{1}{3}d_{n}(u)^{4}+2\left( \mathcal{E}_{n}(u)-\mathcal{
E}_{\mathrm{w},n}(u)\right) ,\qquad \mathcal{E}_{n}(u)-\mathcal{E}_{\mathrm{w},n}(u)\geq
0.  \label{wordtonn}
\end{equation}
\end{theorem}

Theorem \ref{Tlog} is the sharp form of Theorem \ref{TB}, uniformly in $n$. The constant $\frac{1}{3}$ is four times the constant $\frac{1}{12}$, and by \eqref{sharp13} this constant cannot be improved. Since $\mathcal{E}_{\mathrm{w},4}=\mathcal{E}_{4}$, this holds for both the nearest-neighbor energy and the word-length energy. Its proof is based on Theorems \ref{Tcubic}, \ref{Ttwopoint}, \ref{Tmajorantintro} and \ref{Tmomentintro} below, which are of independent interest. More precisely, we derive the four-cycle case from the two-point estimate of Theorem \ref{Ttwopoint}. The case of cycles with at least five vertices and the case of the circle follow from Theorems \ref{Tcubic} and \ref{Tmajorantintro}. The word-length case follows from Theorems \ref{Tmajorantintro} and \ref{Tmomentintro} when $n\geq 6$. When $n=4$, it is the same as the four-cycle case. When $n=5$, it follows from Proposition \ref{Pfive} in Appendix A. We note that in \eqref{sharp13} the constant $\frac{1}{3}$ is approached only in the limit, when $u\rightarrow 1$. This is the typical situation for a degenerate stability estimate, compare \cite{Frank22}. Since the deficit is degenerate, the sharp constant is not given by the second variation. It comes from the fourth-order behavior of the deficit near the constant function. On the four-cycle, a minimizing sequence is given by functions $G_{\varepsilon }$ that are products of two copies of a two-point function. Each $G_{\varepsilon }$ is a constant plus a small first harmonic plus a smaller second harmonic, see Section \ref{SSoptimal13}. We can compute the sharp constant exactly on the four-cycle, and it is the sharp uniform constant for the whole family.

The next result adds to our first main result. It gives lower and upper bounds for the sharp constants $\mathcal{C}_{n}$ and $\mathcal{C}_{n}^{\mathrm{w}}$. For $n\in \mathcal{I}$, we set
\begin{equation}
q_{n}=4\cos ^{2}\frac{\pi }{n}-1\quad (n<\infty ),\qquad q_{\infty }=3.
\label{qn}
\end{equation}
The number $1+q_{n}$ is the eigenvalue of $\mathcal{E}_{n}$ for the second harmonic, see \eqref{eigen} and the proof of Lemma \ref{Lspectral} below.

\begin{theorem}[Lower and upper bounds for the sharp constants]
\label{Cseparation}We have $\mathcal{C}_{4}=\mathcal{C}_{4}^{\mathrm{w}}=\frac{1}{3}$, while
\begin{align}
\mathcal{C}_{n}& \geq \frac{79}{225}=\frac{1}{3}+\frac{4}{225},\qquad
n\in \mathcal{J},  \label{sepnn} \\
\mathcal{C}_{n}^{\mathrm{w}}& \geq \frac{385}{1152}=\frac{1}{3}+\frac{1}{1152},\qquad
5\leq n<\infty .  \label{sepw}
\end{align}
Moreover, on the five-cycle we have the better bound $\mathcal{C}_{5}^{\mathrm{w}}\geq \frac{503}{1500}=\frac{1}{3}+\frac{1}{500}$. On the other hand,
\begin{align}
\mathcal{C}_{n}& \leq \frac{3}{4}-\frac{1}{4q_{n}},\qquad n\in \mathcal{J},
\label{uppnn} \\
\mathcal{C}_{n}^{\mathrm{w}}& \leq \frac{1}{2},\qquad 5\leq n<\infty .
\label{uppw}
\end{align}
\end{theorem}

In particular, $\mathcal{C}_{\infty }\leq \frac{2}{3}$, and among all the models in Theorem \ref{Tlog}, the four-cycle is the only one where the sharp logarithmic stability constant is equal to $\frac{1}{3}$. We point out that the lower bounds \eqref{sepnn} and \eqref{sepw} do not depend on $n$. That is, $\mathcal{C}_{n}-\frac{1}{3}\geq \frac{4}{225}$ for $n\in \mathcal{J}$ and $\mathcal{C}_{n}^{\mathrm{w}}-\frac{1}{3}\geq \frac{1}{1152}$ for $5\leq n<\infty $. So the sharp constants stay above $\frac{1}{3}$ by a fixed amount. These lower bounds do not follow from the sixth-order term in \eqref{log4}, since this term is small compared with $d_{n}(u)^{4}$ when $u$ is close to $1$. The upper bounds \eqref{uppnn} and \eqref{uppw} come from functions close to the constant function $1$, which we construct in Section \ref{SSupper}. Like the functions $G_{\varepsilon }$ on $C_{4}$, these functions have a small first harmonic and a smaller second harmonic. Since $q_{n}$ increases with $n$, the bound in \eqref{uppnn} increases from $\frac{3}{4}-\frac{\sqrt{5}-1}{8}\approx 0.595$ for $n=5$ to $\frac{2}{3}$ for the circle. For example, $\mathcal{C}_{6}\leq \frac{5}{8}$. All these upper bounds are far above $\frac{1}{3}$. In this paper, we do not determine the exact values of $\mathcal{C}_{n}$ and $\mathcal{C}_{n}^{\mathrm{w}}$ for $n\geq 5$, or the exact value of $\mathcal{C}_{\infty }$.

Our second main result, Theorem \ref{Tcubic} below, is about the cubic Sobolev inequality \eqref{cubic} itself. For this inequality, Frank and Ivanisvili themselves suggested the stability question. In \cite[Remark (c) after Theorem 2]{FI26}, they noted that their proof of \eqref{cubic} is similar to the proof of degenerate stability of some Sobolev-type inequalities in \cite{Frank22}. They also wrote that, probably with some additional work, one could get a quantitative version of \eqref{cubic}, as in \cite{BDS24,Frank22}. Like the logarithmic deficit, the cubic deficit is degenerate in the direction of the first harmonics. So the first task is to find the correct power of the distance. For $u\in \mathcal{A}_{n}$, we define the \emph{cubic deficit}
\begin{equation*}
\mathcal{F}_{n}(u)=\mathcal{E}_{n}(u)-\frac{1}{3}\left\langle \left(
u-1\right) ^{2}\left( u+2\right) \right\rangle \geq 0.
\end{equation*}

\begin{theorem}[Stability of the cubic Sobolev inequality of Frank and Ivanisvili]
\label{Tcubic}Let $n\in \mathcal{I}$ and $u\in \mathcal{A}_{n}$. Then
\begin{align}
\mathcal{F}_{n}(u)& \geq \frac{1}{12}d_{n}(u)^{4},\qquad n\in \mathcal{J},
\label{cubic4} \\
\mathcal{F}_{4}(u)& \geq \frac{5}{192}d_{4}(u)^{6},\qquad n=4.  \label{cubic6}
\end{align}
The cubic deficit vanishes if and only if $u=1$. Moreover, the powers are optimal. More precisely, for $0<p<4$ if $n\in \mathcal{J}$, and for $0<p<6$ if $n=4$, we have
\begin{equation}
\inf_{u\in \mathcal{A}_{n}\setminus \{1\}}\frac{\mathcal{F}_{n}(u)}{
d_{n}(u)^{p}}=0.  \label{cubicpower}
\end{equation}
\end{theorem}

Theorem \ref{Tcubic} gives the quantitative version of \eqref{cubic} expected by Frank and Ivanisvili in \cite{FI26}. It is a global estimate. That is, it holds on the whole admissible class without any assumption on the distance to the constant function. Moreover, it determines the optimal powers of the distance, and it shows that the four-cycle is again exceptional. For the proof of Theorem \ref{Tcubic}, we follow \cite{FI26} and decompose $u=a+v+z$. Here $a=\langle u\rangle $ is the mean, $v$ is the projection onto the first harmonics, and $z$ is orthogonal to the constants and the first harmonics. The main observation is that the nonnegativity of $u$ gives
\begin{equation*}
\left\Vert v\right\Vert _{2}^{2}=2\left\vert \widehat{u}(1)\right\vert
^{2}\leq 2a^{2},
\end{equation*}
and this makes the fourth-power bound possible on the whole admissible class. We call this inequality the positivity bound. It links the size of the first harmonic to the mean. We use it to control the cubic interactions when the mean of $u$ is small or moderate. The proof splits into three explicit ranges of the mean, and in each range we use an elementary quadratic estimate.

The constant $\frac{1}{12}$ in \eqref{cubic4} does not depend on $n$. However, the four-cycle behaves differently. Indeed, a suitable second harmonic cancels the fourth-order term of the deficit along the first harmonics. So the deficit is only of sixth order. We give an exact formula for $\mathcal{F}_{4}$ as a sum of nonnegative terms in Proposition \ref{PSOS}. This representation explains the phenomenon and gives the optimal test family. We do not study the optimal constants in \eqref{cubic4}--\eqref{cubic6} in this paper. Theorem \ref{Tcubic} is also one of the main tools in the proof of Theorem \ref{Tlog}. For the bound \eqref{sepnn} in Theorem \ref{Cseparation}, it is important that the argument keeps an explicit positive remainder for large means, see \eqref{cubicB} in Proposition \ref{Pcubic}.

The main tool in the four-cycle part of Theorem \ref{Tlog} is the following sharp stability estimate for Gross's two-point logarithmic Sobolev inequality \cite{Gross75}. This estimate is of independent interest, because the two-point inequality is the base case of the tensorization argument that leads to the Gaussian logarithmic Sobolev inequality \cite{Gross75,Gro06}. In the equivalent form of hypercontractivity, the two-point inequality was first proved by Bonami \cite{Bon70}, see also \cite{Bec75}. On the two-point space $\{-1,1\}$ with the uniform measure, we set $\mathcal{E}_{2}(g)=\frac{1}{4}\left( g(1)-g(-1)\right) ^{2}$ and $\delta _{2}(g)=2\mathcal{E}_{2}(g)-\Ent(g^{2})$. Then Gross's inequality reads $\delta _{2}(g)\geq 0$.

\begin{theorem}[Sharp stability of Gross's two-point inequality]
\label{Ttwopoint}For every nonnegative function $g$ on $\{-1,1\}$ with $\left\Vert g\right\Vert _{2}=1$, we have
\begin{equation}
\delta _{2}(g)\geq \frac{2}{3}\left\Vert g-1\right\Vert _{2}^{4}+\frac{1}{5}
\left\Vert g-1\right\Vert _{2}^{6},  \label{twopoint}
\end{equation}
with equality if and only if $g=1$. The constant $\frac{2}{3}$ cannot be increased, and with $\frac{2}{3}$ fixed, the constant $\frac{1}{5}$ cannot be increased. Equivalently, let $g\neq 0$ be a nonnegative function on $\{-1,1\}$, and set $\xi =\left\Vert g\right\Vert _{2}-\langle g\rangle \geq 0$. Then we have
\begin{equation}
\delta _{2}(g)\geq \frac{8}{3}\xi ^{2}+\frac{8}{5}\frac{\xi ^{3}}{\left\Vert g\right\Vert _{2}}.
\label{twopointhom}
\end{equation}
Equality holds in \eqref{twopointhom} if and only if $g$ is constant.
\end{theorem}

Our proof of Theorem \ref{Ttwopoint} is an integral argument. It is based on an exact formula for the derivative of the deficit. The sixth-order term comes from a positive integral. This sixth-order term gives the strictness, the nonattainment and the quantitative control \eqref{minimizingintro} on $C_{4}$. We identify $C_{4}$ with the product $\{-1,1\}^{2}$. Then the energy is the sum of two two-point energies, and the entropy satisfies a chain rule. We split the deficit of an arbitrary function on $C_{4}$ along its slices, and we keep all the remainder terms. This gives the exact decomposition
\begin{equation*}
\Delta _{4}(u)-\frac{1}{3}d_{4}(u)^{4}=\text{(sum of five explicit
nonnegative terms)},
\end{equation*}
see Proposition \ref{Pexact4}. Two of these terms are two-point remainders, and Theorem \ref{Ttwopoint} shows that they are nonnegative. This decomposition explains the value $\frac{1}{3}=\frac{1}{2}\cdot \frac{2}{3}$. Indeed, the two-point constant $\frac{2}{3}$ is split equally between the two coordinate directions. The product functions $g(s)g(t)$ show that the constant is optimal. The same slice argument, combined with induction on the dimension, gives the sharp constant $\frac{2}{3d}$ in the quartic stability estimate for Gross's inequality on the discrete cube $\{-1,1\}^{d}$, see Proposition \ref{Rcube}. For even $d$, the cube is the product of $\frac{d}{2}$ copies of $C_{4}$, and the product estimate of Xie and Zhang \cite{XZ26} gives the constant $\frac{1}{6d}$. So the constant $\frac{2}{3d}$ gives the sharp form of their product estimate on the cube.

Now we explain the two remaining tools, Theorems \ref{Tmajorantintro} and \ref{Tmomentintro}. Together with Theorem \ref{Tcubic}, they give Theorem \ref{Tlog} and the lower bounds in Theorem \ref{Cseparation} for the cycles with at least five vertices. The only exception is the word-length energy on $C_{5}$, see Remark \ref{Rfive}. For these models, there is no product structure. We use the exact identity
\begin{equation*}
\Delta _{n}(u)=2\mathcal{F}_{n}(u)+\left\langle H(u)\right\rangle ,\qquad
H(w)=\frac{2}{3}\left( w-1\right) ^{2}\left( w+2\right) +\left(
w^{2}-1\right) -w^{2}\log w^{2}\geq 0,
\end{equation*}
where $H$ is the remainder of the cubic majorant of \cite{FI26}. Then we estimate the two terms separately. For the scalar term, we prove the following result. It holds on an arbitrary probability space.

\begin{theorem}[Integrated scalar remainder]
\label{Tmajorantintro}Let $u\geq 0$ on a probability space with $\left\langle u^{2}\right\rangle =1$. Then
\begin{equation*}
\left\langle H(u)\right\rangle \geq \frac{1}{6}\left\Vert u-1\right\Vert _{2}^{4}+\frac{1}{45}
\left\Vert u-1\right\Vert _{2}^{6}.
\end{equation*}
The constant $\frac{1}{6}$ is the best constant that holds on all probability spaces.
\end{theorem}

No pointwise inequality of the form $H(w)\geq c\left( w-1\right) ^{4}$ with $c>0$ can hold for all $w\geq 0$. Indeed, $H(w)\sim \frac{2}{3}w^{3}$ grows only cubically as $w\rightarrow \infty $. Xie and Zhang \cite[Lemma 12]{XZ26} proved the pointwise bound $H(w)\geq \frac{\left( w-1\right) ^{4}}{6\max \{1,w^{2}\}}$. They combined this bound with the Cauchy--Schwarz inequality and $\left\langle \max \{1,u^{2}\}\right\rangle \leq 1+\left\langle u^{2}\right\rangle =2$, and they obtained $\left\langle H(u)\right\rangle \geq \frac{1}{12}\left\Vert u-1\right\Vert _{2}^{4}$. This is where the constant $\frac{1}{12}$ in Theorem \ref{TB} comes from. The idea of Theorem \ref{Tmajorantintro} is to integrate the exact Taylor remainder of $H$ along the segment from $1$ to $u$. Then we use the normalization $\left\langle u^{2}\right\rangle =1$ through the identity $\left\langle \left( 1-\tau +\tau u\right) ^{2}\right\rangle =1-\tau \left( 1-\tau \right) \left\Vert u-1\right\Vert _{2}^{2}$. This doubles the constant. Together with $2\mathcal{F}_{n}(u)\geq \frac{1}{6}d_{n}(u)^{4}$ from Theorem \ref{Tcubic}, we obtain \eqref{log4}. So the constant $\frac{1}{3}$ is the sum of two equal contributions.

For the word-length energy, we cannot apply the argument of Theorem \ref{Tcubic}. The reason is that the eigenvalues $\psi _{n}(k)$ grow only linearly. So we cannot control the supremum norm of $z$ uniformly in $n$. We refer to \cite[Section 1.1]{XZ26} for this point. Following Xie and Zhang, we estimate the cubic mean directly through the Fourier resonances. The following theorem improves their moment estimate $\left\vert \left\langle f^{3}\right\rangle \right\vert \leq 3\left\Vert f\right\Vert _{2}^{2}\big( \mathcal{E}_{\mathrm{w},n}(f)-\left\Vert f\right\Vert _{2}^{2}\big) ^{1/2}$.

\begin{theorem}[An improved third-moment estimate]
\label{Tmomentintro}Let $n\geq 6$ be finite and let $f:C_{n}\rightarrow \mathbb{R}$ with $\langle f\rangle =0$. Then
\begin{equation*}
\left\vert \left\langle f^{3}\right\rangle \right\vert ^{2}\leq \frac{383}{64
}\left\Vert f\right\Vert _{2}^{4}\left( \mathcal{E}_{\mathrm{w},n}(f)-\left\Vert
f\right\Vert _{2}^{2}\right) .
\end{equation*}
\end{theorem}

The main point is that the squared coefficient $\frac{383}{64}$ is strictly smaller than six. We will see that, in our method, the value six is exactly the threshold for the constant $\frac{1}{3}$. We will also see that the strict margin $\frac{1}{64}$ gives the lower bound \eqref{sepw} for $n\geq 6$, which does not depend on $n$. For this, we combine the margin with the positivity bound on the first harmonic. The improvement comes from a more careful assignment of the resonant terms before we apply the weighted Cauchy--Schwarz inequality. The six-, eight- and ten-cycles need a separate treatment.

\begin{remark}
\label{Rfive}The word-length energy on the five-cycle needs a separate argument. The reason is that the estimate of Theorem \ref{Tmomentintro} fails on $C_{5}$ because of the resonance $1+2+2=5$, see Example \ref{Exfive}. On $C_{5}$, we use a Taylor expansion of $w^{2}\log w^{2}$ at a point which depends on the mean of $u$. This gives the bound $\Delta _{\mathrm{w},5}(u)\geq \frac{503}{1500}d_{5}(u)^{4}$, see Proposition \ref{Pfive} in Appendix A.
\end{remark}

The paper is organized as follows. In Section 2, we collect the Fourier decomposition, the bound for the first harmonic in Lemma \ref{Lpositivity}, the exact formulas for the deficits, and the spectral and interaction estimates from \cite{FI26}. In Section 3, we prove Theorem \ref{Tcubic}. In Section 4, we prove Theorem \ref{Tmajorantintro}, the estimate \eqref{log4}, the optimality of the power four and the lower bound \eqref{sepnn}. In Section 5, we prove Theorem \ref{Ttwopoint}, the four-cycle part of Theorem \ref{Tlog} including \eqref{sharp13} and \eqref{minimizingintro}, and the sharp estimate on the discrete cube (Proposition \ref{Rcube}). In Section 6, we prove Theorem \ref{Tmomentintro}, the word-length part of Theorem \ref{Tlog}, the lower bound \eqref{sepw} and the upper bounds \eqref{uppnn} and \eqref{uppw}. In Section \ref{Sproofs}, we put all these results together and prove Theorems \ref{Tlog} and \ref{Cseparation}. In Appendix A, we treat the word-length energy on the five-cycle.
\section{Preliminaries}

In this section, we fix the notation. We also collect the facts about the Fourier decomposition that we use throughout the paper. Most of them are contained in \cite{FI26}. We recall them with complete proofs in our normalization, so that the paper is self-contained.

\subsection{Fourier analysis on the cycle}

For finite $n$, the characters $\chi _{k}(j)=e^{2\pi ijk/n}$, $0\leq k<n$, satisfy
\begin{equation*}
\left\langle \chi _{k}\overline{\chi _{l}}\right\rangle =\frac{1}{n}
\sum_{j=0}^{n-1}e^{2\pi ij(k-l)/n}=\left\{
\begin{array}{ll}
1, & k\equiv l\ (\mathrm{mod}\ n), \\
0, & \text{otherwise,}
\end{array}
\right.
\end{equation*}
because, for $k\not\equiv l$, the sum is a geometric sum with ratio $\omega =e^{2\pi i(k-l)/n}\neq 1$ and $\omega ^{n}=1$. So $\sum_{j=0}^{n-1}\omega ^{j}=\frac{\omega ^{n}-1}{\omega -1}=0$. Therefore, $\{\chi _{k}\}_{0\leq k<n}$ is an orthonormal basis of $L^{2}(C_{n})$. Every $f$ can be written as $f=\sum_{k=0}^{n-1}\widehat{f}(k)\chi _{k}$ with $\widehat{f}(k)=\left\langle f\overline{\chi _{k}}\right\rangle $. Parseval's identity is $\left\langle fg\right\rangle =\sum_{k}\widehat{f}(k)\overline{\widehat{g}(k)}$ for real $f,g$. In particular, $\left\Vert f\right\Vert _{2}^{2}=\sum_{k}\left\vert \widehat{f}(k)\right\vert ^{2}$. We will use the following consequence many times: \emph{the mean of a product of characters $\chi _{k_{1}}\chi _{k_{2}}\cdots \chi _{k_{m}}=\chi _{k_{1}+\cdots +k_{m}}$ is one if $k_{1}+\cdots +k_{m}\equiv 0$ modulo $n$, and it is zero otherwise.}

The characters diagonalize the energy \eqref{energy}. To see this, we use the graph Laplacian $\mathcal{L}f_{j}=2f_{j}-f_{j-1}-f_{j+1}$. We use a summation by parts. This is only a change of the summation index, and it is allowed because the indices are taken modulo $n$. We get
\begin{multline*}
\sum_{j=0}^{n-1}\left( f_{j+1}-f_{j}\right) ^{2}=\sum_{j}f_{j+1}^{2}+
\sum_{j}f_{j}^{2}-2\sum_{j}f_{j}f_{j+1} \\
=2\sum_{j}f_{j}^{2}-\sum_{j}f_{j}\left(
f_{j+1}+f_{j-1}\right) =\sum_{j}f_{j}\,\mathcal{L}f_{j}
\end{multline*}
and $\mathcal{L}\chi _{k}(j)=\left( 2-e^{2\pi ik/n}-e^{-2\pi ik/n}\right) \chi _{k}(j)=2\left( 1-\cos \frac{2\pi k}{n}\right) \chi _{k}(j)$. Therefore, by Parseval's identity,
\begin{equation}
\begin{split}
\mathcal{E}_{n}(f)& =\frac{1}{\mu _{n}}\left\langle f\,\mathcal{L}f\right\rangle =\sum_{k=0}^{n-1}\ell _{n}(k)\left\vert \widehat{f}(k)\right\vert ^{2}, \\
\ell _{n}(k)& =\frac{2\left( 1-\cos (2\pi k/n)\right) }{\mu _{n}}=\frac{1-\cos (2\pi k/n)}{1-\cos (2\pi /n)}.
\end{split}
\label{eigen}
\end{equation}
Note that $\ell _{n}(0)=0$ and $\ell _{n}(\pm 1)=1$. This is the reason for the normalization \eqref{energy}. With this normalization, the spectral gap is one. Moreover, $\ell _{n}(k)=\ell _{n}(n-k)$. We also have $\ell _{n}(k)>1$ for $2\leq k\leq n-2$, because $k\mapsto 1-\cos (2\pi k/n)$ is increasing on $0\leq k\leq n/2$. On the circle $\mathbb{T}$, we use $\widehat{f}(k)=\int_{0}^{1}f(x)e^{-2\pi ikx}dx$, $k\in \mathbb{Z}$. Since $\widehat{f^{\prime }}(k)=2\pi ik\widehat{f}(k)$, Parseval's identity gives $\mathcal{E}_{\infty }(f)=\sum_{k\in \mathbb{Z}}k^{2}\left\vert \widehat{f}(k)\right\vert ^{2}$. So $\ell _{\infty }(k)=k^{2}$, and again the spectral gap is one. Since our functions are real, we have $\widehat{f}(-k)=\overline{\widehat{f}(k)}$. In the finite case, the indices are taken modulo $n$.

Let $V_{1,n}$ be the real eigenspace of the eigenvalue one. On a finite cycle, it is the real span of $\cos (2\pi j/n)$ and $\sin (2\pi j/n)$. In the case of the circle, it is the real span of $\cos (2\pi x)$ and $\sin (2\pi x)$ on $\mathbb{T}$. It is two-dimensional. For $u\in \mathcal{A}_{n}$, we write
\begin{equation}
u=a+v+z,\qquad a=\langle u\rangle ,\qquad v=P_{1}u\in V_{1,n},\qquad z\perp
1,V_{1,n},  \label{decomp}
\end{equation}
where $P_{1}$ is the orthogonal projection onto $V_{1,n}$. In terms of Fourier coefficients,
\begin{equation*}
a=\widehat{u}(0),\qquad v=\widehat{u}(1)\chi _{1}+\widehat{u}(-1)\chi
_{-1}=2\Re \left( \widehat{u}(1)\chi _{1}\right) ,\qquad z=\sum_{2\leq
k\leq n-2}\widehat{u}(k)\chi _{k}
\end{equation*}
where the last sum is over $\left\vert k\right\vert \geq 2$ when $n=\infty $. We set
\begin{equation*}
r=\left\Vert v\right\Vert _{2},\qquad t=\left\Vert z\right\Vert _{2},\qquad
s=r^{2}+t^{2},\qquad Q_{n}(z)=\mathcal{E}_{n}(z)-\left\Vert z\right\Vert
_{2}^{2}.
\end{equation*}
By Parseval's identity, $r^{2}=\left\vert \widehat{u}(1)\right\vert ^{2}+\left\vert \widehat{u}(-1)\right\vert ^{2}=2\left\vert \widehat{u}(1)\right\vert ^{2}$, $t^{2}=\sum_{2\leq k\leq n-2}\left\vert \widehat{u}(k)\right\vert ^{2}$, and
\begin{equation*}
Q_{n}(z)=\sum_{2\leq k\leq n-2}\left( \ell _{n}(k)-1\right) \left\vert
\widehat{u}(k)\right\vert ^{2}\geq 0,
\end{equation*}
which we call the \emph{excess energy} of the high frequencies. It measures how much energy $z$ has above the Poincar\'{e} level $\left\Vert z\right\Vert _{2}^{2}$. The functions $a$, $v$ and $z$ are mutually orthogonal, and $\left\Vert u\right\Vert _{2}=1$. So we get
\begin{equation}
s=1-a^{2},\qquad \mathcal{E}_{n}(u)=r^{2}+t^{2}+Q_{n}(z)=s+Q_{n}(z),\qquad
d_{n}(u)^{2}=\left\langle u^{2}\right\rangle -2\langle u\rangle +1=2(1-a).
\label{basic}
\end{equation}
Here we used $\ell _{n}(\pm 1)=1$ for the energy of $v$. By the Cauchy--Schwarz inequality, $a=\left\langle u\cdot 1\right\rangle \leq \left\Vert u\right\Vert _{2}\left\Vert 1\right\Vert _{2}=1$. Moreover, $a>0$ since $u\geq 0$ and $u\neq 0$. So $0<a\leq 1$. By the equality case in the Cauchy--Schwarz inequality, $a=1$ if and only if $u=1$. We will use the formula $d_{n}(u)^{2}=2(1-a)$ many times. It says that \emph{the distance to the constant function is measured by the mean}. We also have
\begin{equation*}
d_{n}(u)^{4}=4\left( 1-a\right) ^{2},\qquad d_{n}(u)^{6}=8\left( 1-a\right)
^{3}.
\end{equation*}

The next lemma is where the nonnegativity of $u$ enters the cubic estimate. The lemma is elementary, but it plays a very important role in this paper.

\begin{lemma}
\label{Lpositivity}Let $u\in \mathcal{A}_{n}$, $n\in \mathcal{I}$. Then
\begin{equation}
r^{2}=2\left\vert \widehat{u}(1)\right\vert ^{2}\leq 2a^{2}.  \label{positivity}
\end{equation}
Moreover, if $n$ is finite, then $a\geq n^{-1/2}$. In particular, $a\geq \frac{1}{2}$ on $C_{4}$ and $a\geq \frac{1}{\sqrt{5}}$ on $C_{5}$.
\end{lemma}

\begin{proof}
Since $u_{j}\geq 0$ for all $j$ and $\left\vert e^{-2\pi ij/n}\right\vert =1$, the triangle inequality gives
\begin{equation*}
\left\vert \widehat{u}(1)\right\vert =\left\vert \frac{1}{n}
\sum_{j=0}^{n-1}u_{j}e^{-2\pi ij/n}\right\vert \leq \frac{1}{n}
\sum_{j=0}^{n-1}\left\vert u_{j}\right\vert \left\vert e^{-2\pi
ij/n}\right\vert =\frac{1}{n}\sum_{j=0}^{n-1}u_{j}=a.
\end{equation*}
Squaring and multiplying by two, we obtain \eqref{positivity}. On $\mathbb{T}$, the same computation holds with $\left\vert \int_{0}^{1}u(x)e^{-2\pi ix}dx\right\vert \leq \int_{0}^{1}u(x)dx=a$. Now let $n$ be finite. Since all $u_{j}\geq 0$, all the cross terms $2u_{i}u_{j}$ in the expansion of $\left( \sum_{j}u_{j}\right) ^{2}$ are nonnegative. So $\left( \sum_{j}u_{j}\right) ^{2}\geq \sum_{j}u_{j}^{2}=n\left\langle u^{2}\right\rangle =n$. Since $\sum_{j}u_{j}=na$, this gives $n^{2}a^{2}\geq n$, that is, $a\geq n^{-1/2}$. This completes the proof.
\end{proof}

\subsection{Spectral estimates for the high frequencies}

For finite $n$, the two estimates of the next lemma are proved in \cite[Lemma 3]{FI26}. We give the proofs in our normalization, and we add the corresponding result for the circle.

\begin{lemma}
\label{Lspectral}Let $n\geq 4$ be finite and let $z$ be real with $z\perp 1,V_{1,n}$. Then
\begin{equation}
Q_{n}(z)\geq q_{n}\left\Vert z\right\Vert _{2}^{2},\qquad \left\Vert
z\right\Vert _{\infty }^{2}\leq \beta _{n}Q_{n}(z),  \label{spectral}
\end{equation}
where $q_{n}$ is given by \eqref{qn} and
\begin{equation*}
\beta _{n}=\sum_{2\leq k\leq n-2}
\frac{1}{\ell _{n}(k)-1}=\frac{3}{2}-\frac{1}{2}\tan ^{2}\frac{\pi }{n}.
\end{equation*}
In particular,
\begin{equation}
q_{4}=1,\qquad q_{5}=\frac{1+\sqrt{5}}{2}>\frac{3}{2},\qquad q_{n}\geq
2\quad (n\geq 6),\qquad \beta _{n}\leq \frac{3}{2}.  \label{qvalues}
\end{equation}
On the circle, for every real $z\in H^{1}(\mathbb{T})$ with $z\perp 1,V_{1,\infty }$, we have
\begin{equation}
Q_{\infty }(z)\geq 3\left\Vert z\right\Vert _{2}^{2},\qquad \left\Vert
z\right\Vert _{\infty }^{2}\leq \frac{3}{2}Q_{\infty }(z).  \label{spectralcircle}
\end{equation}
\end{lemma}

\begin{proof}
\emph{The lower bound.} We have $\ell _{n}(n-k)=\ell _{n}(k)$, and $\ell _{n}(k)$ is increasing in $k$ for $0\leq k\leq n/2$. So the smallest value of $\ell _{n}(k)-1$ over $2\leq k\leq n-2$ is attained at $k=2$ and at $k=n-2$. Therefore,
\begin{equation*}
Q_{n}(z)=\sum_{2\leq k\leq n-2}\left( \ell _{n}(k)-1\right) \left\vert
\widehat{z}(k)\right\vert ^{2}\geq \left( \ell _{n}(2)-1\right) \sum_{2\leq
k\leq n-2}\left\vert \widehat{z}(k)\right\vert ^{2}=\left( \ell
_{n}(2)-1\right) \left\Vert z\right\Vert _{2}^{2}.
\end{equation*}
To compute $\ell _{n}(2)$, we use $1-\cos 2\theta =2\sin ^{2}\theta $ with $\theta =\frac{2\pi }{n}$ in the numerator and $\theta =\frac{\pi }{n}$ in the denominator. Then we use $\sin 2\theta =2\sin \theta \cos \theta $. We obtain
\begin{equation*}
\ell _{n}(2)=\frac{1-\cos \frac{4\pi }{n}}{1-\cos \frac{2\pi }{n}}=\frac{2\sin
^{2}\frac{2\pi }{n}}{2\sin ^{2}\frac{\pi }{n}}=\frac{4\sin ^{2}\frac{\pi }{n}
\cos ^{2}\frac{\pi }{n}}{\sin ^{2}\frac{\pi }{n}}=4\cos ^{2}\frac{\pi }{n}.
\end{equation*}
So $q_{n}=4\cos ^{2}\frac{\pi }{n}-1$. For $n=4$, we have $\cos ^{2}\frac{\pi }{4}=\frac{1}{2}$ and $q_{4}=1$. For $n=5$, we have $\cos \frac{\pi }{5}=\frac{1+\sqrt{5}}{4}$. Therefore, $\cos ^{2}\frac{\pi }{5}=\frac{6+2\sqrt{5}}{16}=\frac{3+\sqrt{5}}{8}$ and $q_{5}=\frac{3+\sqrt{5}}{2}-1=\frac{1+\sqrt{5}}{2}\approx 1.618>\frac{3}{2}$. For $n=6$, we have $\cos ^{2}\frac{\pi }{6}=\frac{3}{4}$ and $q_{6}=2$. Since $\cos ^{2}\frac{\pi }{n}$ increases with $n$, we get $q_{n}\geq 2$ for all $n\geq 6$.

\emph{The supremum bound.} Let $j$ be arbitrary. We use the Fourier expansion of $z$ and the triangle inequality. We also use the Cauchy--Schwarz inequality with the weights $\ell _{n}(k)-1>0$. We have
\begin{align*}
\left\vert z(j)\right\vert & \leq \sum_{2\leq k\leq n-2}\left\vert \widehat{z}
(k)\right\vert =\sum_{2\leq k\leq n-2}\frac{1}{\sqrt{\ell _{n}(k)-1}}\cdot
\sqrt{\ell _{n}(k)-1}\left\vert \widehat{z}(k)\right\vert \\
& \leq \Big( \sum_{2\leq k\leq n-2}\frac{1}{\ell _{n}(k)-1}\Big) ^{1/2}\Big(
\sum_{2\leq k\leq n-2}\left( \ell _{n}(k)-1\right) \left\vert \widehat{z}
(k)\right\vert ^{2}\Big) ^{1/2}=\beta _{n}^{1/2}Q_{n}(z)^{1/2}.
\end{align*}
Squaring, we get $\left\Vert z\right\Vert _{\infty }^{2}\leq \beta _{n}Q_{n}(z)$.

\emph{The closed form of $\beta _{n}$.} We put $\theta =\frac{\pi }{n}$. We use $1-\cos 2\theta =2\sin ^{2}\theta $ and $\cos 2\theta -\cos 2k\theta =2\sin \left( (k-1)\theta \right) \sin \left( (k+1)\theta \right) $. Then \eqref{eigen} gives
\begin{equation*}
\ell _{n}(k)-1=\frac{\cos 2\theta -\cos 2k\theta }{1-\cos 2\theta }=\frac{\sin
\left( (k-1)\theta \right) \sin \left( (k+1)\theta \right) }{\sin ^{2}\theta }.
\end{equation*}
For $2\leq k\leq n-2$, both angles $(k-1)\theta $ and $(k+1)\theta $ lie in $(0,\pi )$, so both sines are positive. Moreover, $\cot x-\cot y=\frac{\sin (y-x)}{\sin x\sin y}$. We take $x=(k-1)\theta $ and $y=(k+1)\theta $, so that $y-x=2\theta $. Since $\frac{\sin ^{2}\theta }{\sin 2\theta }=\frac{\tan \theta }{2}$, we get
\begin{equation*}
\frac{1}{\ell _{n}(k)-1}=\frac{\sin ^{2}\theta }{\sin \left( (k-1)\theta
\right) \sin \left( (k+1)\theta \right) }=\frac{\tan \theta }{2}\Big( \cot
\left( (k-1)\theta \right) -\cot \left( (k+1)\theta \right) \Big) .
\end{equation*}
The sum over $2\leq k\leq n-2$ is telescoping. Only the terms with $\cot \theta $, $\cot 2\theta $, $\cot \left( (n-2)\theta \right) $ and $\cot \left( (n-1)\theta \right) $ remain. Since $\cot \left( (n-j)\theta \right) =\cot (\pi -j\theta )=-\cot (j\theta )$, we obtain
\begin{align*}
\beta _{n}& =\frac{\tan \theta }{2}\Big( \cot \theta +\cot 2\theta -\cot \left(
(n-2)\theta \right) -\cot \left( (n-1)\theta \right) \Big) \\
& =\tan \theta \left( \cot \theta +\cot 2\theta \right) =1+\tan \theta \cot 2\theta .
\end{align*}
Finally, $\cot 2\theta =\frac{1-\tan ^{2}\theta }{2\tan \theta }$. Therefore, $\beta _{n}=1+\frac{1-\tan ^{2}\theta }{2}=\frac{3}{2}-\frac{1}{2}\tan ^{2}\frac{\pi }{n}$. Since $\tan ^{2}\geq 0$, we always have $\beta _{n}\leq \frac{3}{2}$. For example, $\beta _{4}=1$, since the only term is $k=2$ with $\ell _{4}(2)-1=1$. This is the computation of \cite[Lemma 3]{FI26}, written in our normalization. Indeed, the quadratic form used in \cite{FI26} is $D(z)/\lambda _{n}-2\left\Vert z\right\Vert _{2}^{2}$ with $D(z)=\frac{1}{n}\sum_{j}\left( z_{j+1}-z_{j}\right) ^{2}=2\mathcal{D}_{n}(z)$. Since $\mathcal{E}_{n}=D/\mu _{n}=D/(2\lambda _{n})$, this form equals $2Q_{n}(z)$.

\emph{The circle.} We have $\ell _{\infty }(k)-1=k^{2}-1\geq 3$ for $\left\vert k\right\vert \geq 2$. So we get $Q_{\infty }(z)\geq 3\left\Vert z\right\Vert _{2}^{2}$. The same Cauchy--Schwarz argument shows that $\sum_{\left\vert k\right\vert \geq 2}\left\vert \widehat{z}(k)\right\vert $ is finite. So the Fourier series of $z$ converges absolutely and uniformly to a continuous representative of $z$, and we get $\left\Vert z\right\Vert _{\infty }^{2}\leq \left( \sum_{\left\vert k\right\vert \geq 2}\frac{1}{k^{2}-1}\right) Q_{\infty }(z)$. We use the partial fraction decomposition $\frac{1}{k^{2}-1}=\frac{1}{2}\left( \frac{1}{k-1}-\frac{1}{k+1}\right) $. Then the sum is telescoping:
\begin{equation*}
\sum_{k\geq 2}\frac{1}{k^{2}-1}=\frac{1}{2}\sum_{k\geq 2}\left( \frac{1}{k-1}-
\frac{1}{k+1}\right) =\frac{1}{2}\left( 1+\frac{1}{2}\right) =\frac{3}{4},
\end{equation*}
so the sum over $\left\vert k\right\vert \geq 2$, with both signs, equals $\frac{3}{2}$. This completes the proof.
\end{proof}

In the same way, every function in $H^{1}(\mathbb{T})$ has summable Fourier coefficients and a continuous representative, since the frequencies $0$ and $\pm 1$ only add three terms. Therefore, all the pointwise statements below make sense in $H^{1}(\mathbb{T})$. No further regularity is needed.

\subsection{Exact expressions of the deficits}

Let $y=v+z=u-a$. Then $\langle y\rangle =0$ and $\left\langle y^{2}\right\rangle =r^{2}+t^{2}=s$. We expand $\left( u-1\right) ^{2}\left( u+2\right) =\left( u^{2}-2u+1\right) \left( u+2\right) =u^{3}-3u+2$, and then $u^{3}=\left( a+y\right) ^{3}=a^{3}+3a^{2}y+3ay^{2}+y^{3}$. Taking means, we get
\begin{equation*}
\left\langle \left( u-1\right) ^{2}\left( u+2\right) \right\rangle
=\left\langle u^{3}\right\rangle -3a+2=a^{3}+3as+\left\langle
y^{3}\right\rangle -3a+2.
\end{equation*}
Now we insert $s=1-a^{2}$ and obtain $a^{3}+3a\left( 1-a^{2}\right) -3a+2=a^{3}+3a-3a^{3}-3a+2=2-2a^{3}$. Therefore,
\begin{equation*}
\left\langle \left( u-1\right) ^{2}\left( u+2\right) \right\rangle =2\left(
1-a^{3}\right) +\left\langle y^{3}\right\rangle .
\end{equation*}
Subtracting one third of this identity from $\mathcal{E}_{n}(u)=s+Q_{n}(z)$, we get
\begin{equation*}
\mathcal{F}_{n}(u)=Q_{n}(z)+\left[ \left( 1-a^{2}\right) -\frac{2}{3}\left(
1-a^{3}\right) \right] -\frac{1}{3}\left\langle y^{3}\right\rangle ,
\end{equation*}
and the bracket equals $\frac{3-3a^{2}-2+2a^{3}}{3}=\frac{1-3a^{2}+2a^{3}}{3}=\frac{\left( 1-a\right) ^{2}\left( 1+2a\right) }{3}$. The last step holds because $\left( 1-a\right) ^{2}\left( 1+2a\right) =\left( 1-2a+a^{2}\right) \left( 1+2a\right) =1-3a^{2}+2a^{3}$. Therefore, we have proved the exact formulas
\begin{align}
\mathcal{F}_{n}(u)& =Q_{n}(z)+\frac{\left( 1-a\right) ^{2}\left( 1+2a\right)
}{3}-\frac{1}{3}\left\langle \left( v+z\right) ^{3}\right\rangle
\label{Fexact} \\
& =Q_{n}(z)+\frac{s^{2}}{4}+\mathcal{R}(a)-\frac{1}{3}\left\langle \left(
v+z\right) ^{3}\right\rangle ,\qquad \mathcal{R}(a)=\frac{\left( 1-a\right)
^{3}\left( 1+3a\right) }{12},  \label{Fexact2}
\end{align}
where the second formula comes from the identity
\begin{equation}
\frac{\left( 1-a\right) ^{2}\left( 1+2a\right) }{3}-\frac{\left(
1-a^{2}\right) ^{2}}{4}=\frac{\left( 1-a\right) ^{3}\left( 1+3a\right) }{12}.
\label{Ridentity}
\end{equation}
To verify \eqref{Ridentity}, we expand both sides. The left side is $\frac{4\left( 1-3a^{2}+2a^{3}\right) -3\left( 1-2a^{2}+a^{4}\right) }{12}=\frac{1-6a^{2}+8a^{3}-3a^{4}}{12}$. The right side is $\frac{\left( 1-3a+3a^{2}-a^{3}\right) \left( 1+3a\right) }{12}=\frac{1-6a^{2}+8a^{3}-3a^{4}}{12}$. We note that $\mathcal{R}(a)=\frac{1+3a}{96}d_{n}(u)^{6}$ is of sixth order in the distance. Moreover, $\frac{s^{2}}{4}=\frac{\left( 1-a\right) ^{2}\left( 1+a\right) ^{2}}{4}\approx d_{n}(u)^{4}/4$ near the constant.

For the logarithmic deficit, we set
\begin{equation}
H(w)=\frac{2}{3}\left( w-1\right) ^{2}\left( w+2\right) +\left(
w^{2}-1\right) -w^{2}\log w^{2},\qquad w\geq 0,  \label{Hdef}
\end{equation}
with $H(0)=\frac{4}{3}-1=\frac{1}{3}$. Here we recall the convention $0\log 0=0$. Since $\left\langle u^{2}\right\rangle =1$, the average of the term $u^{2}-1$ vanishes. So we have
\begin{equation*}
\left\langle H(u)\right\rangle =\frac{2}{3}\left\langle \left( u-1\right)
^{2}\left( u+2\right) \right\rangle -\left\langle u^{2}\log
u^{2}\right\rangle =\frac{2}{3}\left\langle \left( u-1\right) ^{2}\left(
u+2\right) \right\rangle -\Ent(u^{2}).
\end{equation*}
On the other hand, $2\mathcal{F}_{n}(u)=2\mathcal{E}_{n}(u)-\frac{2}{3}\left\langle \left( u-1\right) ^{2}\left( u+2\right) \right\rangle $. We add the two identities. Then the cubic terms cancel, and we obtain
\begin{equation}
\Delta _{n}(u)=2\mathcal{F}_{n}(u)+\left\langle H(u)\right\rangle .
\label{splitting}
\end{equation}
The function $H$ is nonnegative. This is the cubic majorant of \cite[Lemma 1]{FI26}. We will prove it in Section 4, see formula \eqref{Hintegral}. Formula \eqref{splitting} shows that, to get a lower bound for the logarithmic deficit, it is enough to have lower bounds for the cubic deficit $\mathcal{F}_{n}(u)$ and for the scalar quantity $\left\langle H(u)\right\rangle $. The first one is an \emph{energy} quantity, and it depends on the graph. The second one is a \emph{scalar} quantity, which depends only on the distribution of the values of $u$.

\subsection{Cubic interaction estimates}

We now prove the estimates for the cubic mean $\left\langle \left( v+z\right) ^{3}\right\rangle $. For finite $n$, they come from the proof of \cite[Theorem 2]{FI26}.

\begin{lemma}
\label{Linteraction}Let $n\in \mathcal{J}$, $v\in V_{1,n}$ and let $z$ be real with $z\perp 1,V_{1,n}$. We also assume $z\in H^{1}(\mathbb{T})$ when $n=\infty $. Set $r=\left\Vert v\right\Vert _{2}$ and $t=\left\Vert z\right\Vert _{2}$. Then
\begin{align}
\frac{1}{3}\left\vert \left\langle \left( v+z\right) ^{3}\right\rangle
\right\vert & \leq \frac{r^{2}t}{\sqrt{2}}+\sqrt{2}rt^{2}+\frac{1}{\sqrt{6}}
\sqrt{Q_{n}(z)}\,t^{2},\qquad 6\leq n\leq \infty ,  \label{interaction} \\
\frac{1}{3}\left\vert \left\langle \left( v+z\right) ^{3}\right\rangle
\right\vert & \leq \frac{r^{2}t+rt^{2}}{\sqrt{2}},\qquad n=5.
\label{interaction5}
\end{align}
\end{lemma}

\begin{proof}
\emph{Step 1: the first harmonic.} Every $v\in V_{1,n}$ of norm $r$ can be written as $v(j)=\sqrt{2}r\cos \left( 2\pi j/n+\theta \right) $ for some $\theta \in \mathbb{R}$. Indeed, $v=\alpha \cos (2\pi j/n)+\alpha ^{\prime }\sin (2\pi j/n)=\sqrt{\alpha ^{2}+\alpha ^{\prime 2}}\cos (2\pi j/n+\theta )$ for a suitable $\theta $. Moreover, $\left\langle \cos ^{2}(2\pi j/n+\theta )\right\rangle =\frac{1}{2}$ by the orthogonality of the characters $\chi _{k}$. Indeed, we have $\cos ^{2}(2\pi j/n+\theta )=\frac{1}{2}+\frac{1}{2}\cos (4\pi j/n+2\theta )$ and $\cos (4\pi j/n+2\theta )=\Re \left( e^{2i\theta }\chi _{2}(j)\right) $. For $n\geq 3$, we have $2\not\equiv 0$ modulo $n$, so $\left\langle \chi _{2}\right\rangle =0$ by the orthogonality relation above. Therefore, the mean of $\cos (4\pi j/n+2\theta )$ vanishes. So $r^{2}=\frac{1}{2}\left( \alpha ^{2}+\alpha ^{\prime 2}\right) $. As a consequence, $\left\Vert v\right\Vert _{\infty }\leq \sqrt{2}r$. In addition, by $\cos ^{2}\phi =\frac{1+\cos 2\phi }{2}$, we have
\begin{equation*}
v^{2}=r^{2}+r^{2}\cos \left( 4\pi j/n+2\theta \right) ,\qquad \left\Vert
v^{2}-r^{2}\right\Vert _{2}^{2}=r^{4}\left\langle \cos ^{2}\left( 4\pi
j/n+2\theta \right) \right\rangle =\frac{r^{4}}{2}.
\end{equation*}
Here we used that the mean of $\cos \left( 8\pi j/n+4\theta \right) $ vanishes, since $n\geq 5$. This fails on $C_{4}$.
Moreover, $\left\langle v^{3}\right\rangle =0$. Indeed, the frequencies in $v^{3}$ are $\pm 1\pm 1\pm 1\in \{\pm 1,\pm 3\}$. None of them is zero modulo $n$ when $n\geq 4$. For $n=3$, we have $3\equiv 0$ and $\left\langle v^{3}\right\rangle \neq 0$ in general. This is where the argument breaks down for the three-cycle, and in fact \eqref{LSIconstant} is false for $n=3$. On the circle, we write $v(x)=\sqrt{2}\,r\cos \left( 2\pi x+\theta \right) $. The same computations hold, because $\int_{0}^{1}\cos \left( 2\pi mx+\phi \right) dx=0$ for every integer $m\neq 0$ and every real $\phi $.

\emph{Step 2: expansion of the cube.} We expand $\left( v+z\right) ^{3}=v^{3}+3v^{2}z+3vz^{2}+z^{3}$. Note that $\left\langle v^{3}\right\rangle =0$ and $\langle z\rangle =0$. So $\left\langle v^{2}z\right\rangle =\left\langle \left( v^{2}-r^{2}\right) z\right\rangle $. Therefore, we get
\begin{equation*}
\frac{1}{3}\left\langle \left( v+z\right) ^{3}\right\rangle =\left\langle
\left( v^{2}-r^{2}\right) z\right\rangle +\left\langle vz^{2}\right\rangle +
\frac{1}{3}\left\langle z^{3}\right\rangle .
\end{equation*}
We estimate the three terms separately. By the Cauchy--Schwarz inequality and Step 1,
\begin{equation*}
\left\vert \left\langle \left( v^{2}-r^{2}\right) z\right\rangle \right\vert
\leq \left\Vert v^{2}-r^{2}\right\Vert _{2}\left\Vert z\right\Vert _{2}=
\frac{r^{2}t}{\sqrt{2}}.
\end{equation*}
By the supremum bound on $v$,
\begin{equation*}
\left\vert \left\langle vz^{2}\right\rangle \right\vert \leq \left\Vert
v\right\Vert _{\infty }\left\langle z^{2}\right\rangle \leq \sqrt{2}rt^{2}.
\end{equation*}
By Lemma \ref{Lspectral},
\begin{equation*}
\frac{1}{3}\left\vert \left\langle z^{3}\right\rangle \right\vert \leq \frac{1
}{3}\left\Vert z\right\Vert _{\infty }\left\langle z^{2}\right\rangle \leq
\frac{\sqrt{\beta _{n}}}{3}\sqrt{Q_{n}(z)}\,t^{2}\leq \frac{1}{3}\sqrt{\frac{3
}{2}}\sqrt{Q_{n}(z)}\,t^{2}=\frac{1}{\sqrt{6}}\sqrt{Q_{n}(z)}\,t^{2},
\end{equation*}
where $\frac{1}{3}\sqrt{\frac{3}{2}}=\frac{\sqrt{3}}{3\sqrt{2}}=\frac{1}{\sqrt{3}\sqrt{2}}=\frac{1}{\sqrt{6}}$. We add the three bounds and obtain \eqref{interaction}. On the circle, all the products are integrable because $H^{1}(\mathbb{T})\subset C^{0}(\mathbb{T})$. Moreover, \eqref{spectralcircle} gives the same constant.

\emph{Step 3: the five-cycle.} For $n=5$, the only high frequencies are $\pm 2$, since $3\equiv -2$ and $4\equiv -1$. We write $\chi =\chi _{1}$, $v=c_{1}\chi +\overline{c_{1}}\chi ^{-1}$ and $z=c_{2}\chi ^{2}+\overline{c_{2}}\chi ^{-2}$ with $c_{1},c_{2}\in \mathbb{C}$. Then Parseval's identity gives $r^{2}=2\left\vert c_{1}\right\vert ^{2}$ and $t^{2}=2\left\vert c_{2}\right\vert ^{2}$. To compute $\left\langle \left( v+z\right) ^{3}\right\rangle $, we expand the cube of the four-term sum. We keep only the products of three characters with frequencies that add up to $0$ modulo $5$. The available frequencies are $\pm 1,\pm 2$. The multisets with sum $\equiv 0$ are
\begin{equation*}
\{1,1,-2\},\qquad \{-1,-1,2\},\qquad \{1,2,2\},\qquad \{-1,-2,-2\},
\end{equation*}
because $1+1-2=0$ and $1+2+2=5$. No other combination works. If a multiset contains both $k$ and $-k$, then its sum is the third frequency, which is not $0$ modulo $5$. The remaining multisets have the sums $\pm 3$, $\pm 4$ and $\pm 6$, for example $1+1+1=3$, $1+1+2=4$ and $2+2+2=6$, and none of them is $0$ modulo $5$. Each multiset consists of two equal indices and one different index. So it corresponds to $3$ ordered triples. Therefore,
\begin{equation*}
\left\langle \left( v+z\right) ^{3}\right\rangle =3c_{1}^{2}\overline{c_{2}}+3
\overline{c_{1}}^{2}c_{2}+3c_{1}c_{2}^{2}+3\overline{c_{1}}\,\overline{c_{2}}^{2}=6\Re \left( c_{1}^{2}
\overline{c_{2}}\right) +6\Re \left( c_{1}c_{2}^{2}\right) ,
\end{equation*}
that is,
\begin{equation}
\frac{1}{3}\left\langle \left( v+z\right) ^{3}\right\rangle =2\Re \left(
c_{1}^{2}\overline{c_{2}}\right) +2\Re \left( c_{1}c_{2}^{2}\right) .  \label{fiveexact}
\end{equation}
We take absolute values and use $\left\vert c_{1}\right\vert =\frac{r}{\sqrt{2}}$ and $\left\vert c_{2}\right\vert =\frac{t}{\sqrt{2}}$. Then we get
\begin{equation*}
\frac{1}{3}\left\vert \left\langle \left( v+z\right) ^{3}\right\rangle
\right\vert \leq 2\left\vert c_{1}\right\vert ^{2}\left\vert c_{2}\right\vert
+2\left\vert c_{1}\right\vert \left\vert c_{2}\right\vert ^{2}=2\cdot \frac{r^{2}}{2}
\cdot \frac{t}{\sqrt{2}}+2\cdot \frac{r}{\sqrt{2}}\cdot \frac{t^{2}}{2}=\frac{
r^{2}t}{\sqrt{2}}+\frac{rt^{2}}{\sqrt{2}},
\end{equation*}
which is \eqref{interaction5}. This completes the proof.
\end{proof}

\section{Stability of the cubic Sobolev inequality -- \texorpdfstring{\mbox{Proof of Theorem~\ref{Tcubic}}}{Proof of Theorem \ref{Tcubic}}}

In this section, we prove Theorem \ref{Tcubic}. Its first part also gives the lower bound for the energy term $2\mathcal{F}_{n}(u)$ in the splitting \eqref{splitting}. This lower bound is needed for \eqref{log4} and for the lower bound \eqref{sepnn} in Theorem \ref{Cseparation}. The proof follows the strategy of \cite{FI26}. More precisely, by \eqref{Fexact} and Lemma \ref{Linteraction}, $\mathcal{F}_{n}(u)$ is bounded from below by an explicit expression in the four numbers $a,r,t$ and $Q=Q_{n}(z)$. The new point is that we use the positivity bound \eqref{positivity}. We need it on the part of the admissible class where the mean is small or moderate.

\subsection{A positivity-dependent scalar inequality}

The hypotheses of the next scalar inequality come from the facts proved in Section 2. We have $r^{2}+t^{2}=1-a^{2}$ by \eqref{basic}, and $r^{2}\leq 2a^{2}$ by Lemma \ref{Lpositivity}. We have $Q\geq 2t^{2}$ when $6\leq n\leq \infty $ and $Q\geq \frac{3}{2}t^{2}$ when $n=5$, both by Lemma \ref{Lspectral}. We also have $a\geq \frac{1}{\sqrt{5}}$ when $n=5$ by Lemma \ref{Lpositivity}. In addition, we keep a positive term for large means. We will use this term in the proof of \eqref{sepnn}.

\begin{theorem}
\label{Tscalar}Let $0\leq a\leq 1$ and $r,t,Q\geq 0$ be such that
\begin{equation*}
r^{2}+t^{2}=1-a^{2}=:s,\qquad r^{2}\leq 2a^{2}.
\end{equation*}

\begin{enumerate}
\item[(i)] If $Q\geq 2t^{2}$, then
\begin{equation}
Q+\frac{\left( 1-a\right) ^{2}\left( 1+2a\right) }{3}-\frac{r^{2}t}{\sqrt{2}}
-\sqrt{2}rt^{2}-\frac{\sqrt{Q}\,t^{2}}{\sqrt{6}}\geq \frac{\left( 1-a\right)
^{2}}{3}.  \label{scalar1}
\end{equation}

\item[(ii)] If $a\geq \frac{1}{\sqrt{5}}$ and $Q\geq \frac{3}{2}t^{2}$, then
\begin{equation}
Q+\frac{\left( 1-a\right) ^{2}\left( 1+2a\right) }{3}-\frac{r^{2}t+rt^{2}}{
\sqrt{2}}\geq \frac{\left( 1-a\right) ^{2}}{3}.  \label{scalar2}
\end{equation}
\end{enumerate}
Moreover, if $a\geq \frac{3}{5}$, the right-hand sides of \eqref{scalar1} and \eqref{scalar2} can be replaced by $\frac{\left( 1-a\right) ^{2}}{3}+B(a)$, where
\begin{equation}
B(a)=\frac{\left( 1-a\right) ^{2}\left( 5-3a\right) \left( 5a-3\right) }{96}
\geq 0,\qquad \frac{3}{5}\leq a\leq 1.  \label{Bdef}
\end{equation}
\end{theorem}

\begin{proof}
If $a=1$, then $s=0$. So $r=t=0$, and both inequalities become $Q\geq 0$. Also $B(1)=0$. Therefore, we may assume $a<1$. Note that $t^{2}\leq s\leq 1$, so $t\leq 1$, and $s=\left( 1-a\right) \left( 1+a\right) $.

\textbf{Proof of (i).}

\emph{Step 1: minimization over $Q$.} The left-hand side of \eqref{scalar1} depends on $Q$ only through the function $\phi (Q)=Q-\frac{t^{2}}{\sqrt{6}}\sqrt{Q}$. If $t=0$, then $\phi (Q)=Q$ is increasing. If $t>0$, then for $Q\geq 2t^{2}$ we have $\sqrt{Q}\geq \sqrt{2}t$ and
\begin{equation*}
\phi ^{\prime }(Q)=1-\frac{t^{2}}{2\sqrt{6}\sqrt{Q}}\geq 1-\frac{t^{2}}{2
\sqrt{6}\sqrt{2}t}=1-\frac{t}{2\sqrt{12}}\geq 1-\frac{1}{2\sqrt{12}}>0,
\end{equation*}
since $t\leq 1$. So $\phi $ is increasing on $\left[ 2t^{2},\infty \right) $, and its minimum is at $Q=2t^{2}$. We obtain
\begin{equation*}
Q-\frac{\sqrt{Q}\,t^{2}}{\sqrt{6}}\geq \phi (2t^{2})=2t^{2}-\frac{\sqrt{2}
t\cdot t^{2}}{\sqrt{6}}=2t^{2}-\frac{t^{3}}{\sqrt{3}}.
\end{equation*}
In other words, the worst case is when the excess energy of the high frequencies is as small as Lemma \ref{Lspectral} allows.

\emph{Step 2: collecting the terms.} We subtract the right-hand side of \eqref{scalar1}. First, we have
\begin{equation*}
\frac{\left( 1-a\right) ^{2}\left( 1+2a\right) }{3}-\frac{\left( 1-a\right)
^{2}}{3}=\frac{\left( 1-a\right) ^{2}\cdot 2a}{3}=\frac{2a\left( 1-a\right)
^{2}}{3}.
\end{equation*}
Second, we write $r^{2}=s-t^{2}$ in the term $-\frac{r^{2}t}{\sqrt{2}}$. Then
\begin{equation*}
-\frac{r^{2}t}{\sqrt{2}}=-\frac{\left( s-t^{2}\right) t}{\sqrt{2}}=-\frac{st}{
\sqrt{2}}+\frac{t^{3}}{\sqrt{2}}.
\end{equation*}
By Step 1, the left-hand side of \eqref{scalar1} minus its right-hand side is at least
\begin{align}
& 2t^{2}-\frac{t^{3}}{\sqrt{3}}-\frac{st}{\sqrt{2}}+\frac{t^{3}}{\sqrt{2}}-
\sqrt{2}rt^{2}+\frac{2a\left( 1-a\right) ^{2}}{3}  \notag \\
& =\left( 2-\sqrt{2}r\right) t^{2}-\frac{s}{\sqrt{2}}t+\frac{2a\left(
1-a\right) ^{2}}{3}+\left( \frac{1}{\sqrt{2}}-\frac{1}{\sqrt{3}}\right)
t^{3}  \notag \\
& \geq At^{2}-\frac{s}{\sqrt{2}}t+\frac{2a\left( 1-a\right) ^{2}}{3},\qquad
A=2-\sqrt{2}r,  \label{quadratic}
\end{align}
where we dropped the term $\left( \frac{1}{\sqrt{2}}-\frac{1}{\sqrt{3}}\right) t^{3}\geq 0$. The expression \eqref{quadratic} looks like a quadratic polynomial in $t$ with leading coefficient $A$. However, $A$ depends on $r=\sqrt{s-t^{2}}$, and so on $t$. In each case below, we first replace $A$ by a lower bound that does not depend on $t$. The linear coefficient is negative. The constant term is nonnegative, and it vanishes like $\left( 1-a\right) ^{2}$ as $a\rightarrow 1$. We will prove that this expression is nonnegative. For this, we consider three cases, according to the size of the mean. We need these cases because the discriminant condition $\frac{s^{2}}{2}\leq \frac{8aA\left( 1-a\right) ^{2}}{3}$ for the quadratic holds for $a$ close to one, but fails for small $a$. Therefore, for small $a$ we instead use the fact that $t$ is bounded away from zero.

\emph{Case 1: $0\leq a\leq \frac{1}{3}$.} By the positivity bound $r\leq \sqrt{2}a$, we have $\sqrt{2}r\leq 2a\leq \frac{2}{3}$. So $A=2-\sqrt{2}r\geq 2-2a\geq \frac{4}{3}$. Moreover, we have
\begin{equation*}
t^{2}=s-r^{2}=1-a^{2}-r^{2}\geq 1-a^{2}-2a^{2}=1-3a^{2}\geq 1-\frac{1}{3}=
\frac{2}{3},\qquad \text{hence}\qquad t\geq \sqrt{\tfrac{2}{3}}>\frac{3}{4}.
\end{equation*}
Therefore, $At\geq \frac{4}{3}\cdot \frac{3}{4}=1$. On the other hand, $\frac{s}{\sqrt{2}}\leq \frac{1}{\sqrt{2}}<1$. As a result,
\begin{equation*}
At^{2}-\frac{s}{\sqrt{2}}t=t\left( At-\frac{s}{\sqrt{2}}\right) \geq
t\left( 1-\frac{1}{\sqrt{2}}\right) \geq 0,
\end{equation*}
and the other term $\frac{2a\left( 1-a\right) ^{2}}{3}$ in \eqref{quadratic} is nonnegative. So \eqref{quadratic} is nonnegative in Case 1. In other words, when the mean is small, nonnegativity puts most of the mass into the high frequencies. Then $t$ is large, and the quadratic term $At^{2}$ alone is larger than the interaction term.

\emph{Case 2: $\frac{1}{3}\leq a\leq \frac{3}{5}$.} Now we use the positivity bound in the form $A\geq 2-2a=2\left( 1-a\right) $. We also use $s=\left( 1-a\right) \left( 1+a\right) $. Since $t^{2}\geq 0$, the expression \eqref{quadratic} is at least
\begin{equation*}
2\left( 1-a\right) t^{2}-\frac{\left( 1-a\right) \left( 1+a\right) }{\sqrt{2}}
t+\frac{2a\left( 1-a\right) ^{2}}{3}=\left( 1-a\right) \left( 2t^{2}-\frac{1+a
}{\sqrt{2}}t+\frac{2a\left( 1-a\right) }{3}\right) .
\end{equation*}
We complete the square in the bracket and get $2t^{2}-\frac{1+a}{\sqrt{2}}t=2\left( t-\frac{1+a}{4\sqrt{2}}\right) ^{2}-2\cdot \frac{\left( 1+a\right) ^{2}}{32}=2\left( t-\frac{1+a}{4\sqrt{2}}\right) ^{2}-\frac{\left( 1+a\right) ^{2}}{16}$. So the bracket is equal to
\begin{equation*}
2\left( t-\frac{1+a}{4\sqrt{2}}\right) ^{2}+\frac{2a\left( 1-a\right) }{3}-
\frac{\left( 1+a\right) ^{2}}{16}=2\left( t-\frac{1+a}{4\sqrt{2}}\right)
^{2}+\frac{32a\left( 1-a\right) -3\left( 1+a\right) ^{2}}{48}.
\end{equation*}
By expanding, we have $32a\left( 1-a\right) -3\left( 1+a\right) ^{2}=32a-32a^{2}-3-6a-3a^{2}=-35a^{2}+26a-3$. The roots of $-35a^{2}+26a-3$ are $a=\frac{26\pm \sqrt{26^{2}-4\cdot 35\cdot 3}}{70}=\frac{26\pm 16}{70}$, that is, $a=\frac{1}{7}$ and $a=\frac{3}{5}$. Therefore,
\begin{equation*}
-35a^{2}+26a-3=\left( 7a-1\right) \left( 3-5a\right) ,
\end{equation*}
as one can check by expanding the right-hand side. This expansion gives $21a-35a^{2}-3+5a$. The product is nonnegative exactly for $\frac{1}{7}\leq a\leq \frac{3}{5}$, and this range contains the interval $\left[ \frac{1}{3},\frac{3}{5}\right] $. Therefore, the bracket is nonnegative, and \eqref{quadratic} is nonnegative in Case 2. Note that $\left( 7a-1\right) \left( 3-5a\right) $ vanishes at $a=\frac{3}{5}$. For this reason, we must use a different lower bound for $A$ in the next case.

\emph{Case 3: $\frac{3}{5}\leq a<1$.} Here $s=1-a^{2}\leq 1-\frac{9}{25}=\frac{16}{25}$. We do not use the positivity bound here. We only use the trivial bound $r^{2}\leq s$, and we obtain
\begin{equation*}
\sqrt{2}r\leq \sqrt{2s}\leq \sqrt{\frac{32}{25}}=\frac{4\sqrt{2}}{5}<\frac{6}{
5},\qquad \text{hence}\qquad A=2-\sqrt{2}r>\frac{4}{5}.
\end{equation*}
By completing the square with the leading coefficient $\frac{4}{5}$, we get
\begin{equation*}
\frac{4}{5}t^{2}-\frac{s}{\sqrt{2}}t=\frac{4}{5}\left( t-\frac{5s}{8\sqrt{2}}
\right) ^{2}-\frac{4}{5}\cdot \frac{25s^{2}}{128}=\frac{4}{5}\left( t-\frac{5s
}{8\sqrt{2}}\right) ^{2}-\frac{5s^{2}}{32}.
\end{equation*}
Therefore, by using $A\geq \frac{4}{5}$ and $s^{2}=\left( 1-a\right) ^{2}\left( 1+a\right) ^{2}$, we have
\begin{align*}
At^{2}-\frac{s}{\sqrt{2}}t+\frac{2a\left( 1-a\right) ^{2}}{3}& \geq \frac{4}{5}
\left( t-\frac{5s}{8\sqrt{2}}\right) ^{2}+\frac{2a\left( 1-a\right) ^{2}}{3}-
\frac{5\left( 1-a\right) ^{2}\left( 1+a\right) ^{2}}{32} \\
& =\frac{4}{5}\left( t-\frac{5s}{8\sqrt{2}}\right) ^{2}+\left( 1-a\right)
^{2}\,\frac{64a-15\left( 1+a\right) ^{2}}{96}.
\end{align*}
By expanding, we have $64a-15\left( 1+a\right) ^{2}=-15a^{2}+34a-15$. Its roots are $a=\frac{34\pm \sqrt{34^{2}-900}}{30}=\frac{34\pm 16}{30}$, that is, $a=\frac{3}{5}$ and $a=\frac{5}{3}$. Therefore,
\begin{equation*}
-15a^{2}+34a-15=\left( 5-3a\right) \left( 5a-3\right) ,
\end{equation*}
and this is nonnegative for $\frac{3}{5}\leq a\leq \frac{5}{3}$. In particular, it is nonnegative on $\left[ \frac{3}{5},1\right] $. So the last expression is $\geq \frac{\left( 1-a\right) ^{2}\left( 5-3a\right) \left( 5a-3\right) }{96}=B(a)\geq 0$. This proves \eqref{scalar1} in Case 3. It also gives the improvement by $B(a)$ when $a\geq \frac{3}{5}$. Indeed, we simply keep the term $B(a)$ instead of dropping it.

\textbf{Proof of (ii).}

\emph{Step 1: reduction.} We subtract $\frac{\left( 1-a\right) ^{2}}{3}$ from both sides. Then we use $Q\geq \frac{3}{2}t^{2}$ and $r^{2}=s-t^{2}$, and we obtain
\begin{align*}
& Q+\frac{\left( 1-a\right) ^{2}\left( 1+2a\right) }{3}-\frac{\left(
s-t^{2}\right) t+rt^{2}}{\sqrt{2}}-\frac{\left( 1-a\right) ^{2}}{3} \\
& \geq \frac{3}{2}t^{2}-\frac{st}{\sqrt{2}}+\frac{t^{3}}{\sqrt{2}}-\frac{
rt^{2}}{\sqrt{2}}+\frac{2a\left( 1-a\right) ^{2}}{3}\geq A_{5}t^{2}-\frac{s}{
\sqrt{2}}t+\frac{2a\left( 1-a\right) ^{2}}{3},\qquad A_{5}=\frac{3}{2}-\frac{r
}{\sqrt{2}},
\end{align*}
where we dropped the nonnegative term $\frac{t^{3}}{\sqrt{2}}$. It is enough to show that the last expression is nonnegative. Since $a\geq \frac{1}{\sqrt{5}}\approx 0.447>\frac{1}{3}$, only the ranges of Cases 2 and 3 in the proof of (i) occur.

\emph{Case 2: $\frac{1}{3}\leq a\leq \frac{3}{5}$.} By the positivity bound, $\frac{r}{\sqrt{2}}\leq a$. So $A_{5}\geq \frac{3}{2}-a=:A_{0}>0$. By completing the square with the coefficient $A_{0}$, we get
\begin{equation*}
A_{0}t^{2}-\frac{s}{\sqrt{2}}t=A_{0}\left( t-\frac{s}{2\sqrt{2}A_{0}}\right)
^{2}-\frac{s^{2}}{8A_{0}},
\end{equation*}
and therefore, with $s^{2}=\left( 1-a\right) ^{2}\left( 1+a\right) ^{2}$, we have
\begin{equation*}
A_{5}t^{2}-\frac{s}{\sqrt{2}}t+\frac{2a\left( 1-a\right) ^{2}}{3}\geq
A_{0}\left( t-\frac{s}{2\sqrt{2}A_{0}}\right) ^{2}+\left( 1-a\right)
^{2}\left( \frac{2a}{3}-\frac{\left( 1+a\right) ^{2}}{8A_{0}}\right) .
\end{equation*}
We multiply the last bracket by the positive number $24A_{0}=36-24a$, and we get
\begin{align*}
24A_{0}\left( \frac{2a}{3}-\frac{\left( 1+a\right) ^{2}}{8A_{0}}\right)
& =16a\left( \frac{3}{2}-a\right) -3\left( 1+a\right) ^{2} \\
& =24a-16a^{2}-3-6a-3a^{2}=-19a^{2}+18a-3.
\end{align*}
At $a=\frac{1}{3}$ this equals $-\frac{19}{9}+6-3=\frac{8}{9}$, and at $a=\frac{3}{5}$ it equals $-\frac{171}{25}+\frac{54}{5}-3=\frac{-171+270-75}{25}=\frac{24}{25}$. If a concave quadratic is positive at the two endpoints of an interval, then it is positive on the whole interval. Therefore, the bracket is positive on $\left[ \frac{1}{3},\frac{3}{5}\right] $, and the expression is nonnegative.

\emph{Case 3: $\frac{3}{5}\leq a<1$.} Now $\frac{r}{\sqrt{2}}\leq \sqrt{\frac{s}{2}}\leq \sqrt{\frac{8}{25}}=\frac{2\sqrt{2}}{5}\approx 0.566<\frac{3}{5}$, so $A_{5}\geq \frac{3}{2}-\frac{3}{5}=\frac{9}{10}>\frac{4}{5}$. The computation of Case 3 in the proof of (i) used only $A\geq \frac{4}{5}$ and $\frac{3}{5}\leq a<1$. Therefore, it works without change with $A_{5}$ in place of $A$, and it gives the same term $B(a)$. This proves \eqref{scalar2} and its improvement. This completes the proof.
\end{proof}

\subsection{Proof of \eqref{cubic4}}

\begin{proposition}
\label{Pcubic}Let $n\in \mathcal{J}$ and $u\in \mathcal{A}_{n}$ with mean $a=\langle u\rangle $. Then \eqref{cubic4} holds, and, if $a\geq \frac{3}{5}$,
\begin{equation}
\mathcal{F}_{n}(u)\geq \frac{1}{12}d_{n}(u)^{4}+B(a),  \label{cubicB}
\end{equation}
with $B(a)$ given by \eqref{Bdef}. In particular, $\mathcal{F}_{n}(u)=0$ if and only if $u=1$.
\end{proposition}

\begin{proof}
We use the decomposition \eqref{decomp} and the notation $r,t,s,Q=Q_{n}(z)$. When $n=\infty $, both $a$ and $v$ are smooth. So $z=u-a-v\in H^{1}(\mathbb{T})$, and we can apply Lemma \ref{Linteraction}. If $a=1$, then $u=1$ and there is nothing to prove. So we assume $a<1$. By \eqref{basic} and Lemma \ref{Lpositivity}, we have $r^{2}+t^{2}=1-a^{2}$ and $r^{2}\leq 2a^{2}$. These are the hypotheses of Theorem \ref{Tscalar}.

\emph{The case $6\leq n\leq \infty $.} By Lemma \ref{Lspectral}, $Q\geq q_{n}t^{2}\geq 2t^{2}$ for finite $n\geq 6$, and $Q\geq 3t^{2}\geq 2t^{2}$ on the circle. We start from the exact formula \eqref{Fexact}. We replace $-\frac{1}{3}\left\langle \left( v+z\right) ^{3}\right\rangle $ by the lower bound $-\frac{1}{3}\left\vert \left\langle \left( v+z\right) ^{3}\right\rangle \right\vert $ and use \eqref{interaction}. Then we apply Theorem \ref{Tscalar} (i). We obtain
\begin{align*}
\mathcal{F}_{n}(u)& =Q+\frac{\left( 1-a\right) ^{2}\left( 1+2a\right) }{3}-
\frac{1}{3}\left\langle \left( v+z\right) ^{3}\right\rangle \\
& \geq Q+\frac{\left( 1-a\right) ^{2}\left( 1+2a\right) }{3}-\frac{r^{2}t}{
\sqrt{2}}-\sqrt{2}rt^{2}-\frac{\sqrt{Q}\,t^{2}}{\sqrt{6}} \\
& \geq \frac{\left( 1-a\right) ^{2}}{3}=\frac{1}{3}\left( \frac{d_{n}(u)^{2}}{
2}\right) ^{2}=\frac{1}{12}d_{n}(u)^{4},
\end{align*}
where in the last step we used $d_{n}(u)^{2}=2\left( 1-a\right) $ from \eqref{basic}.

\emph{The case $n=5$.} In this case, the only high frequencies are $\pm 2$. So $Q=Q_{5}(z)=\left( \ell _{5}(2)-1\right) t^{2}=q_{5}t^{2}\geq \frac{3}{2}t^{2}$ by \eqref{qvalues}. Moreover, $a\geq \frac{1}{\sqrt{5}}$ by Lemma \ref{Lpositivity}. By \eqref{Fexact}, \eqref{interaction5} and Theorem \ref{Tscalar} (ii), we have
\begin{equation*}
\mathcal{F}_{5}(u)\geq Q+\frac{\left( 1-a\right) ^{2}\left( 1+2a\right) }{3}-
\frac{r^{2}t+rt^{2}}{\sqrt{2}}\geq \frac{\left( 1-a\right) ^{2}}{3}=\frac{1}{
12}d_{5}(u)^{4}.
\end{equation*}

\emph{Large means.} If $a\geq \frac{3}{5}$, we use the same two chains of inequalities with the improved right-hand side $\frac{\left( 1-a\right) ^{2}}{3}+B(a)$ of Theorem \ref{Tscalar}. This gives \eqref{cubicB}.

\emph{Zero deficit.} If $u\neq 1$, then $a<1$ and the right-hand side of \eqref{cubic4} is positive. So $\mathcal{F}_{n}(u)>0$. Conversely, $\mathcal{F}_{n}(1)=0$ follows directly from the definition.
\end{proof}

\subsection{The four-cycle: an exact formula}

On $C_{4}$ we use the orthonormal basis $1,e,f,\chi _{2}$ of $L^{2}(C_{4})$, where
\begin{equation}
e=\left( \sqrt{2},0,-\sqrt{2},0\right) ,\qquad f=\left( 0,\sqrt{2},0,-\sqrt{2}
\right) ,\qquad \chi _{2}=\left( 1,-1,1,-1\right) .  \label{basis4}
\end{equation}
These are $e(j)=\sqrt{2}\cos \frac{\pi j}{2}$, $f(j)=\sqrt{2}\sin \frac{\pi j}{2}$ and $\chi _{2}(j)=\cos (\pi j)=(-1)^{j}$ (the character of frequency $2$). They are orthogonal, and they satisfy $\left\langle e^{2}\right\rangle =\left\langle f^{2}\right\rangle =\left\langle \chi _{2}^{2}\right\rangle =1$. Moreover, by \eqref{eigen}, the eigenvalue is $\ell _{4}(1)=\ell _{4}(3)=1$ for $e$ and $f$. For $\chi _{2}$, it is $\ell _{4}(2)=\frac{1-\cos \pi }{1-\cos \frac{\pi }{2}}=2$. Therefore, we can write every $u\in \mathcal{A}_{4}$ in a unique way as $u=a+xe+yf+c\chi _{2}$ with $v=xe+yf$, $z=c\chi _{2}$, and we have
\begin{equation*}
r^{2}=x^{2}+y^{2},\qquad t^{2}=c^{2},\qquad Q_{4}(c\chi _{2})=\left( 2-1\right)
c^{2}=c^{2},\qquad a^{2}+r^{2}+c^{2}=1,\qquad a\geq \frac{1}{2},
\end{equation*}
where the last inequality comes from Lemma \ref{Lpositivity}.

\begin{proposition}
\label{PSOS}For every $u=a+xe+yf+c\chi _{2}\in \mathcal{A}_{4}$ we have the exact identity
\begin{equation}
\mathcal{F}_{4}(u)=\left( c-\frac{x^{2}-y^{2}}{2}\right) ^{2}+x^{2}y^{2}+
\frac{1}{2}\left( x^{2}+y^{2}\right) c^{2}+\frac{1}{4}c^{4}+\mathcal{R}(a).
\label{SOS}
\end{equation}
\end{proposition}

\begin{proof}
\emph{Step 1: pointwise products.} We multiply the vectors \eqref{basis4} entry by entry, and we get
\begin{equation*}
e^{2}=\left( 2,0,2,0\right) =1+\chi _{2},\quad f^{2}=\left( 0,2,0,2\right)
=1-\chi _{2},\quad ef=0,\quad \chi _{2}^{2}=\left( 1,1,1,1\right) =1.
\end{equation*}
Therefore,
\begin{equation*}
v^{2}=x^{2}e^{2}+2xyef+y^{2}f^{2}=x^{2}\left( 1+\chi _{2}\right) +y^{2}\left( 1-\chi _{2}\right) =r^{2}+\left( x^{2}-y^{2}\right) \chi _{2}.
\end{equation*}

\emph{Step 2: the cubic mean.} We expand $\left( v+c\chi _{2}\right) ^{3}=v^{3}+3cv^{2}\chi _{2}+3c^{2}v\chi _{2}^{2}+c^{3}\chi _{2}^{3}$ and take the mean of each term. First, $\left\langle v^{3}\right\rangle =0$, since the frequencies of $v^{3}$ are $\pm 1$ and $\pm 3$, and none of them is $\equiv 0$ modulo $4$. Next, $\left\langle v^{2}\chi _{2}\right\rangle =r^{2}\left\langle \chi _{2}\right\rangle +\left( x^{2}-y^{2}\right) \left\langle \chi _{2}^{2}\right\rangle =x^{2}-y^{2}$. In addition, $\left\langle v\chi _{2}^{2}\right\rangle =\left\langle v\right\rangle =0$ and $\left\langle \chi _{2}^{3}\right\rangle =\left\langle \chi _{2}\right\rangle =0$. Therefore,
\begin{equation*}
\frac{1}{3}\left\langle \left( v+c\chi _{2}\right) ^{3}\right\rangle =c\left(
x^{2}-y^{2}\right) .
\end{equation*}

\emph{Step 3: substitution.} We insert this into \eqref{Fexact2} with $Q_{4}(z)=c^{2}$ and $s=r^{2}+c^{2}$. We expand $\frac{\left( r^{2}+c^{2}\right) ^{2}}{4}=\frac{r^{4}}{4}+\frac{r^{2}c^{2}}{2}+\frac{c^{4}}{4}$, and we add and subtract $\frac{1}{4}\left( x^{2}-y^{2}\right) ^{2}$. Then
\begin{align*}
\mathcal{F}_{4}(u)-\mathcal{R}(a)& =c^{2}+\frac{\left( r^{2}+c^{2}\right) ^{2}}{
4}-c\left( x^{2}-y^{2}\right) \\
& =\left( c-\frac{x^{2}-y^{2}}{2}\right) ^{2}+\frac{r^{4}-\left(
x^{2}-y^{2}\right) ^{2}}{4}+\frac{r^{2}c^{2}}{2}+\frac{c^{4}}{4}.
\end{align*}
Since $r^{2}=x^{2}+y^{2}$, we have
\begin{equation*}
r^{4}-\left( x^{2}-y^{2}\right) ^{2}=\left( r^{2}-x^{2}+y^{2}\right) \left(
r^{2}+x^{2}-y^{2}\right) =\left( 2y^{2}\right) \left( 2x^{2}\right)
=4x^{2}y^{2}.
\end{equation*}
By collecting the terms, we obtain \eqref{SOS}. This completes the proof.
\end{proof}

\subsection{Proof of \eqref{cubic6} and of the zero-deficit statement}\label{SScubic6}

All the terms in \eqref{SOS} are nonnegative. We drop the first four terms. We use $a\geq \frac{1}{2}$ and $\left( 1-a\right) ^{3}=\left( \frac{d_{4}(u)^{2}}{2}\right) ^{3}=\frac{1}{8}d_{4}(u)^{6}$, and we get
\begin{equation*}
\mathcal{F}_{4}(u)\geq \mathcal{R}(a)=\frac{\left( 1-a\right) ^{3}\left(
1+3a\right) }{12}=\frac{1+3a}{96}d_{4}(u)^{6}\geq \frac{1+\frac{3}{2}}{96}
d_{4}(u)^{6}=\frac{5}{192}d_{4}(u)^{6},
\end{equation*}
which is \eqref{cubic6}. In every model, the right-hand side of the corresponding cubic bound is positive when $u\neq 1$. Moreover, $\mathcal{F}_{n}(1)=0$. Therefore, $u=1$ is the only zero of the cubic deficit in $\mathcal{A}_{n}$.

\subsection{Optimality of the powers}\label{SSpowers}

First, we study the cubic deficit along the simplest nonconstant functions. We use the same functions again in Section \ref{SSpower} for the logarithmic deficit. Let $e_{n}(j)=\sqrt{2}\cos \left( 2\pi j/n\right) =\frac{1}{\sqrt{2}}\left( \chi _{1}(j)+\chi _{-1}(j)\right) $ for finite $n$ and $e_{\infty }(x)=\sqrt{2}\cos \left( 2\pi x\right) $. Then
\begin{equation}
\left\langle e_{n}\right\rangle =0,\qquad \left\langle e_{n}^{2}\right\rangle
=1,\qquad \left\langle e_{n}^{3}\right\rangle =0,\qquad \left\Vert
e_{n}\right\Vert _{\infty }\leq \sqrt{2},\qquad \mathcal{E}_{n}(e_{n})=1.
\label{enprops}
\end{equation}
Indeed, $\widehat{e_{n}}(\pm 1)=\frac{1}{\sqrt{2}}$. So $\left\langle e_{n}^{2}\right\rangle =\left\vert \widehat{e_{n}}(1)\right\vert ^{2}+\left\vert \widehat{e_{n}}(-1)\right\vert ^{2}=\frac{1}{2}+\frac{1}{2}=1$ by Parseval, and $\mathcal{E}_{n}(e_{n})=\ell _{n}(1)\left\langle e_{n}^{2}\right\rangle =1$. Moreover, $\left\langle e_{n}^{3}\right\rangle =0$, because the frequencies of $e_{n}^{3}$ are $\pm 1,\pm 3$, and none of them is zero modulo $n\geq 4$. On the circle, these frequencies are not zero either. For $0<\varepsilon <\frac{1}{4}$, we define
\begin{equation}
u_{\varepsilon }=\frac{1+\varepsilon e_{n}}{\sqrt{1+\varepsilon ^{2}}}.
\label{testfamily}
\end{equation}
These functions are positive, since $1+\varepsilon e_{n}\geq 1-\sqrt{2}\varepsilon >0$. They have $L^{2}$-norm one, since $\left\langle \left( 1+\varepsilon e_{n}\right) ^{2}\right\rangle =1+2\varepsilon \left\langle e_{n}\right\rangle +\varepsilon ^{2}\left\langle e_{n}^{2}\right\rangle =1+\varepsilon ^{2}$. They are also nonconstant, and on the circle they are smooth, so $u_{\varepsilon }\in \mathcal{A}_{\infty }$. Their mean is $a_{\varepsilon }=\left( 1+\varepsilon ^{2}\right) ^{-1/2}$, their first harmonic is $\varepsilon a_{\varepsilon }e_{n}$, and their high part $z$ vanishes. Therefore, by \eqref{basic} and \eqref{Fexact} with $Q=0$ and $\left\langle v^{3}\right\rangle =0$, we have
\begin{equation*}
d_{n}(u_{\varepsilon })^{2}=2\left( 1-a_{\varepsilon }\right) ,\qquad
\mathcal{E}_{n}(u_{\varepsilon })=\varepsilon ^{2}a_{\varepsilon }^{2}=
\frac{\varepsilon ^{2}}{1+\varepsilon ^{2}},\qquad \mathcal{F}
_{n}(u_{\varepsilon })=\frac{\left( 1-a_{\varepsilon }\right)
^{2}\left( 1+2a_{\varepsilon }\right) }{3}.
\end{equation*}
By the binomial series $\left( 1+x\right) ^{-1/2}=1-\frac{1}{2}x+\frac{3}{8}x^{2}+O(x^{3})$ with $x=\varepsilon ^{2}$, we get
\begin{equation*}
a_{\varepsilon }=1-\frac{1}{2}\varepsilon ^{2}+\frac{3}{8}\varepsilon
^{4}+O(\varepsilon ^{6}),\qquad 1-a_{\varepsilon }=\frac{1}{2}
\varepsilon ^{2}+O(\varepsilon ^{4}),
\end{equation*}
and therefore
\begin{equation}
d_{n}(u_{\varepsilon })^{2}=\varepsilon ^{2}+O(\varepsilon ^{4}),\qquad
\mathcal{F}_{n}(u_{\varepsilon })=\frac{\left( \frac{1}{2}\varepsilon
^{2}\right) ^{2}\cdot 3}{3}+O(\varepsilon ^{6})=\frac{1}{4}\varepsilon
^{4}+O(\varepsilon ^{6}).  \label{testorders}
\end{equation}
As a consequence, for $0<p<4$, we have
\begin{equation*}
\frac{\mathcal{F}_{n}(u_{\varepsilon })}{d_{n}(u_{\varepsilon })^{p}}=\frac{
\frac{1}{4}\varepsilon ^{4}+O(\varepsilon ^{6})}{\varepsilon
^{p}(1+O(\varepsilon ^{2}))}=\frac{1}{4}\varepsilon ^{4-p}\left(
1+O(\varepsilon ^{2})\right) \rightarrow 0\qquad (\varepsilon \rightarrow 0),
\end{equation*}
which proves \eqref{cubicpower} for $n\in \mathcal{J}$.

On $C_{4}$, the family \eqref{testfamily} only gives $\mathcal{F}_{4}(u_{\varepsilon })\sim \frac{1}{4}\varepsilon ^{4}$. This is not small enough to show that the power six is optimal. We use Proposition \ref{PSOS} as a guide, and we make the first square vanish, that is, $y=0$, $c=\frac{x^{2}}{2}$. More precisely, we set
\begin{equation*}
U_{\varepsilon }=b_{\varepsilon }+\varepsilon e+\frac{\varepsilon ^{2}}{2}\chi _{2},\qquad
b_{\varepsilon }=\sqrt{1-\varepsilon ^{2}-\frac{\varepsilon ^{4}}{4}},\qquad
0<\varepsilon <\frac{1}{4}.
\end{equation*}
By orthogonality, $\left\Vert U_{\varepsilon }\right\Vert _{2}^{2}=b_{\varepsilon }^{2}+\varepsilon ^{2}+\frac{\varepsilon ^{4}}{4}=1$. Since $\varepsilon <\frac{1}{4}$, we have $b_{\varepsilon }^{2}\geq 1-\frac{1}{16}-\frac{1}{1024}>\frac{9}{16}$, and so $b_{\varepsilon }>\frac{3}{4}$. Moreover, every entry of $U_{\varepsilon }$ is at least
\begin{equation*}
b_{\varepsilon }-\left\Vert \varepsilon e\right\Vert _{\infty }-\left\Vert
\frac{\varepsilon ^{2}}{2}\chi _{2}\right\Vert _{\infty }=b_{\varepsilon }-\sqrt{2}
\varepsilon -\frac{\varepsilon ^{2}}{2}>\frac{3}{4}-\frac{\sqrt{2}}{4}-\frac{1
}{32}>\frac{3}{4}-\frac{3}{8}-\frac{1}{32}=\frac{11}{32}>0.
\end{equation*}
So $U_{\varepsilon }\in \mathcal{A}_{4}\setminus \{1\}$, and $d_{4}(U_{\varepsilon })^{2}=2\left( 1-b_{\varepsilon }\right) =\varepsilon ^{2}+O(\varepsilon ^{4})$. Indeed, $\sqrt{1-\eta }=1-\frac{\eta }{2}+O(\eta ^{2})$ with $\eta =\varepsilon ^{2}+\frac{\varepsilon ^{4}}{4}$ gives $1-b_{\varepsilon }=\frac{\varepsilon ^{2}}{2}+O(\varepsilon ^{4})$. In \eqref{SOS}, we have $x=\varepsilon $, $y=0$, $c=\frac{\varepsilon ^{2}}{2}$. So $c-\frac{x^{2}-y^{2}}{2}=0$ and $x^{2}y^{2}=0$, and we get
\begin{equation*}
\mathcal{F}_{4}(U_{\varepsilon })=\frac{1}{2}\varepsilon ^{2}\cdot \frac{
\varepsilon ^{4}}{4}+\frac{1}{4}\cdot \frac{\varepsilon ^{8}}{16}+\mathcal{R}
(b_{\varepsilon })=\frac{\varepsilon ^{6}}{8}+\frac{\varepsilon ^{8}}{64}+
\mathcal{R}(b_{\varepsilon }).
\end{equation*}
Moreover, $\mathcal{R}(b_{\varepsilon })=\frac{\left( 1-b_{\varepsilon }\right) ^{3}\left( 1+3b_{\varepsilon }\right) }{12}=\frac{\left( \frac{\varepsilon ^{2}}{2}\right) ^{3}\cdot 4}{12}+O(\varepsilon ^{8})=\frac{\varepsilon ^{6}}{24}+O(\varepsilon ^{8})$. Therefore,
\begin{equation*}
\mathcal{F}_{4}(U_{\varepsilon })=\frac{\varepsilon ^{6}}{8}+\frac{
\varepsilon ^{6}}{24}+O(\varepsilon ^{8})=\frac{3+1}{24}\varepsilon
^{6}+O(\varepsilon ^{8})=\frac{\varepsilon ^{6}}{6}+O(\varepsilon ^{8}),
\end{equation*}
and $\mathcal{F}_{4}(U_{\varepsilon })/d_{4}(U_{\varepsilon })^{p}\sim \frac{1}{6}\varepsilon ^{6-p}\rightarrow 0$ for $0<p<6$. This proves \eqref{cubicpower} for $n=4$. The quotients in \eqref{cubicpower} are nonnegative by \eqref{cubic4}--\eqref{cubic6}. Therefore, their infima are exactly zero.

\begin{proof}[Proof of Theorem \ref{Tcubic}]
The estimate \eqref{cubic4} is proved in Proposition \ref{Pcubic}. The estimate \eqref{cubic6} is proved in Section \ref{SScubic6}, by the exact formula of Proposition \ref{PSOS}. By Proposition \ref{Pcubic} and Section \ref{SScubic6}, the cubic deficit vanishes only for $u=1$. Finally, \eqref{cubicpower} is proved in Section \ref{SSpowers}. This completes the proof of Theorem \ref{Tcubic}.
\end{proof}

\section{The scalar remainder and the nearest-neighbor logarithmic stability -- \texorpdfstring{\mbox{Proof of Theorem~\ref{Tmajorantintro}}}{Proof of Theorem \ref{Tmajorantintro}}}

\subsection{An integral formula for $H$}

Let $\varphi (w)=w^{2}\log w^{2}=2w^{2}\log w$ for $w>0$. Its derivatives are
\begin{gather*}
\varphi ^{\prime }(w)=2w\log w^{2}+2w,\qquad \varphi ^{\prime \prime
}(w)=2\log w^{2}+6,\qquad \varphi ^{\prime \prime \prime }(w)=\frac{4}{w}, \\
\varphi ^{(4)}(w)=-\frac{4}{w^{2}},\qquad \varphi ^{(5)}(w)=\frac{8}{
w^{3}}.
\end{gather*}
At $w=1$, we have $\varphi (1)=0$, $\varphi ^{\prime }(1)=2$, $\varphi ^{\prime \prime }(1)=6$, $\varphi ^{\prime \prime \prime }(1)=4$. Therefore, the cubic Taylor polynomial of $\varphi $ at $w=1$ is
\begin{equation*}
P_{3}(w)=2\left( w-1\right) +3\left( w-1\right) ^{2}+\frac{2}{3}\left(
w-1\right) ^{3}.
\end{equation*}
We check that this is the cubic polynomial in \eqref{Hdef}. Indeed, $w^{2}-1=\left( w-1\right) \left( w+1\right) =2\left( w-1\right) +\left( w-1\right) ^{2}$, and $\frac{2}{3}\left( w-1\right) ^{2}\left( w+2\right) =\frac{2}{3}\left( w-1\right) ^{2}\left( 3+\left( w-1\right) \right) =2\left( w-1\right) ^{2}+\frac{2}{3}\left( w-1\right) ^{3}$. The sum of these two expressions is $P_{3}(w)$. Therefore, $H=P_{3}-\varphi $ is minus the cubic Taylor remainder of $\varphi $. By Taylor's formula with integral remainder, we have, for $w>0$,
\begin{equation*}
\varphi (w)-P_{3}(w)=\int_{1}^{w}\frac{\left( w-\sigma \right) ^{3}}{3!}
\varphi ^{(4)}(\sigma )d\sigma =-\frac{2}{3}\int_{1}^{w}\frac{\left( w-\sigma
\right) ^{3}}{\sigma ^{2}}d\sigma .
\end{equation*}
We substitute $\sigma =1+\tau \left( w-1\right) $, $0\leq \tau \leq 1$. Then $d\sigma =\left( w-1\right) d\tau $ and $w-\sigma =\left( 1-\tau \right) \left( w-1\right) $, and we obtain
\begin{equation}
\begin{split}
H(w)& =\frac{2}{3}\int_{0}^{1}\frac{\left( 1-\tau \right) ^{3}\left( w-1\right) ^{3}}{\left( 1+\tau \left( w-1\right) \right) ^{2}}\left( w-1\right) d\tau \\
& =\frac{2}{3}\left( w-1\right) ^{4}\int_{0}^{1}\frac{\left( 1-\tau \right) ^{3}}{\left( 1+\tau \left( w-1\right) \right) ^{2}}d\tau ,\qquad w\geq 0.
\end{split}
\label{Hintegral}
\end{equation}
Since $1+\tau \left( w-1\right) =\left( 1-\tau \right) +\tau w>0$ for $w\geq 0$ and $0\leq \tau <1$, the integrand is positive and $H\geq 0$. This is \cite[Lemma 1]{FI26}. At $w=0$, the integrand is $\left( 1-\tau \right) ^{3}/\left( 1-\tau \right) ^{2}=1-\tau $ and the integral equals $\frac{1}{2}$. So \eqref{Hintegral} gives $H(0)=\frac{2}{3}\cdot \frac{1}{2}=\frac{1}{3}$. This agrees with \eqref{Hdef}. Near $w=1$, the integral is close to $\int_{0}^{1}\left( 1-\tau \right) ^{3}d\tau =\frac{1}{4}$, and so $H(w)\approx \frac{1}{6}\left( w-1\right) ^{4}$. That is, the remainder is of fourth order, as expected for a cubic Taylor approximation.

For large $w$, the remainder is much smaller than $\left( w-1\right) ^{4}$. More precisely, for $w\rightarrow \infty $, we have $H(w)\sim \frac{2}{3}w^{3}$, so the remainder grows only cubically. Therefore, $H(w)/\left( w-1\right) ^{4}\rightarrow 0$. As a result, \emph{no} pointwise inequality $H(w)\geq c\left( w-1\right) ^{4}$ with $c>0$ holds on all of $\left[ 0,\infty \right) $. A pointwise bound that does hold is the bound of Xie and Zhang \cite[Lemma 12]{XZ26},
\begin{equation*}
H(w)\geq \frac{\left( w-1\right) ^{4}}{6\max \{1,w^{2}\}},\qquad w\geq 0,
\end{equation*}
which follows directly from \eqref{Hintegral}. Indeed, for $w\leq 1$ the denominator satisfies $1+\tau \left( w-1\right) \leq 1$, and for $w\geq 1$ it satisfies $1+\tau \left( w-1\right) \leq w$. Therefore, in both cases the integral in \eqref{Hintegral} is at least $\frac{1}{4\max \{1,w^{2}\}}$. As in \cite[Lemma 12]{XZ26}, we take the average and apply the Cauchy--Schwarz inequality in the form $\left\langle \frac{\left( u-1\right) ^{4}}{\max \{1,u^{2}\}}\right\rangle \left\langle \max \{1,u^{2}\}\right\rangle \geq \left\langle \left( u-1\right) ^{2}\right\rangle ^{2}$. Together with $\left\langle \max \{1,u^{2}\}\right\rangle \leq 1+\left\langle u^{2}\right\rangle =2$, we obtain $\left\langle H(u)\right\rangle \geq \frac{1}{12}\left\Vert u-1\right\Vert _{2}^{4}$. This is the origin of the constant $\frac{1}{12}$ in Theorem \ref{TB}. The factor $\frac{1}{2}$ is lost when we pass from the pointwise bound to the average. Now we prove Theorem \ref{Tmajorantintro}. The proof shows that if we average the exact remainder in \eqref{Hintegral} and use the normalization $\left\langle u^{2}\right\rangle =1$, then we obtain the constant $\frac{1}{6}$, which is twice as large, together with a sixth-order term.

\begin{proof}[Proof of Theorem \ref{Tmajorantintro}]
Let $(X,\mu )$ be the probability space in Theorem \ref{Tmajorantintro}, and let $d=\left\Vert u-1\right\Vert _{2}$. We prove the more precise estimate
\begin{equation}
\left\langle H(u)\right\rangle \geq d^{4}\Theta (d^{2})\geq \frac{1
}{6}d^{4}+\frac{1}{45}d^{6},\qquad \Theta (x)=\frac{2}{3}
\int_{0}^{1}\frac{\left( 1-\tau \right) ^{3}}{1-\tau \left( 1-\tau \right) x}
d\tau .  \label{majorant}
\end{equation}
Here the left-hand side is allowed to be $+\infty $. If $d=0$, then $u=1$ a.e.\ and both sides of \eqref{majorant} vanish. Let $d>0$ and $a=\langle u\rangle $. As in Section 2, we have $0<a<1$ and $d^{2}=\left\langle u^{2}\right\rangle -2a+1=2\left( 1-a\right) <2$.

\emph{Step 1: the auxiliary functions.} For $0<\tau <1$, we define
\begin{equation*}
\omega _{\tau }=\left( 1-\tau \right) +\tau u,\qquad q_{\tau }=\frac{\left(
u-1\right) ^{2}}{\omega _{\tau }},\qquad m_{\tau }=\left\langle \omega _{\tau
}^{2}\right\rangle .
\end{equation*}
Since $\omega _{\tau }\geq \min \{\tau ,1-\tau \}\left( 1+u\right) >0$, the function $q_{\tau }$ is well defined. Moreover, $0\leq q_{\tau }\leq \frac{\left( 1+u\right) ^{2}}{\min \{\tau ,1-\tau \}\left( 1+u\right) }=\frac{1+u}{\min \{\tau ,1-\tau \}}$, and so $q_{\tau }\in L^{2}$. We expand the square and use $\left\langle u\right\rangle =a$, $\left\langle u^{2}\right\rangle =1$. Then we get
\begin{multline*}
m_{\tau }=\left( 1-\tau \right) ^{2}+2\tau \left( 1-\tau \right) a+\tau
^{2}=\left[ \left( 1-\tau \right) ^{2}+\tau ^{2}\right] +2\tau \left( 1-\tau
\right) a \\
=1-2\tau \left( 1-\tau \right) +2\tau \left( 1-\tau \right) a,
\end{multline*}
because $\left( 1-\tau \right) ^{2}+\tau ^{2}=1-2\tau \left( 1-\tau \right) $. Therefore,
\begin{equation}
m_{\tau }=1-2\tau \left( 1-\tau \right) \left( 1-a\right) =1-\tau \left(
1-\tau \right) d^{2}\geq 1-\frac{1}{4}\cdot 2=\frac{1}{2},  \label{mtau}
\end{equation}
where we used $\tau \left( 1-\tau \right) \leq \frac{1}{4}$ and $d^{2}<2$. In addition, since $q_{\tau }\omega _{\tau }=\left( u-1\right) ^{2}$, we have
\begin{equation*}
\left\langle q_{\tau }\omega _{\tau }\right\rangle =\left\langle \left( u-1\right)
^{2}\right\rangle =d^{2}.
\end{equation*}

\emph{Step 2: Cauchy--Schwarz with its remainder.} Let $c_{\tau }=\frac{\left\langle q_{\tau }\omega _{\tau }\right\rangle }{\left\langle \omega _{\tau }^{2}\right\rangle }=\frac{d^{2}}{m_{\tau }}$. Then $c_{\tau }\omega _{\tau }$ is the orthogonal projection of $q_{\tau }$ onto the line spanned by $\omega _{\tau }$ in $L^{2}$. By expanding the square, we have
\begin{equation*}
\left\Vert q_{\tau }-c_{\tau }\omega _{\tau }\right\Vert _{2}^{2}=\left\langle
q_{\tau }^{2}\right\rangle -2c_{\tau }\left\langle q_{\tau }\omega _{\tau
}\right\rangle +c_{\tau }^{2}m_{\tau }=\left\langle q_{\tau
}^{2}\right\rangle -2\frac{d^{4}}{m_{\tau }}+\frac{d^{4}}{
m_{\tau }}=\left\langle q_{\tau }^{2}\right\rangle -\frac{d^{4}}{
m_{\tau }}.
\end{equation*}
Therefore, we have the exact identity
\begin{equation*}
\left\langle q_{\tau }^{2}\right\rangle =\frac{d^{4}}{m_{\tau }}
+\left\Vert q_{\tau }-c_{\tau }\omega _{\tau }\right\Vert
_{2}^{2}\geq \frac{d^{4}}{m_{\tau }},
\end{equation*}
which is exactly the Cauchy--Schwarz inequality $d^{4}=\left\langle q_{\tau }\omega _{\tau }\right\rangle ^{2}\leq \left\langle q_{\tau }^{2}\right\rangle \left\langle \omega _{\tau }^{2}\right\rangle $ written with its remainder.

\emph{Step 3: integration.} We apply \eqref{Hintegral} at every point of $X$, with $w$ equal to the value of $u$ at this point. Since $1+\tau \left( u-1\right) =\omega _{\tau }$, we get
\begin{equation*}
H(u)=\frac{2}{3}\int_{0}^{1}\left( 1-\tau \right) ^{3}\frac{\left(
u-1\right) ^{4}}{\omega _{\tau }^{2}}d\tau =\frac{2}{3}\int_{0}^{1}\left( 1-\tau
\right) ^{3}q_{\tau }^{2}d\tau .
\end{equation*}
We take the average over $X$ and exchange the order of integration. This is allowed by Tonelli's theorem, since all quantities are nonnegative. We obtain
\begin{multline*}
\left\langle H(u)\right\rangle =\frac{2}{3}\int_{0}^{1}\left( 1-\tau \right)
^{3}\left\langle q_{\tau }^{2}\right\rangle d\tau \geq \frac{2}{3}
\int_{0}^{1}\left( 1-\tau \right) ^{3}\frac{d^{4}}{m_{\tau }}d\tau \\
=\frac{2}{3}d^{4}\int_{0}^{1}\frac{\left( 1-\tau \right) ^{3}}{1-\tau
\left( 1-\tau \right) d^{2}}d\tau =d^{4}\Theta (d^{2}).
\end{multline*}
This proves the first inequality in \eqref{majorant}. It also holds when $\left\langle H(u)\right\rangle =+\infty $, since all the integrands are nonnegative.

\emph{Step 4: expansion of $\Theta $.} We use the algebraic identity $\frac{1}{1-\beta }=1+\beta +\frac{\beta ^{2}}{1-\beta }$, which can be checked by multiplying out. We take $\beta =\tau \left( 1-\tau \right) x\in \left[ 0,\frac{1}{2}\right] $ for $0\leq x\leq 2$, and we get
\begin{equation*}
\Theta (x)=\frac{2}{3}\int_{0}^{1}\left( 1-\tau \right) ^{3}d\tau +\frac{2}{3}
x\int_{0}^{1}\tau \left( 1-\tau \right) ^{4}d\tau +\frac{2}{3}
x^{2}\int_{0}^{1}\frac{\tau ^{2}\left( 1-\tau \right) ^{5}}{1-\tau \left(
1-\tau \right) x}d\tau .
\end{equation*}
The first integral is $\frac{1}{4}$. The second one is a Beta integral, and it is equal to $\frac{1!\,4!}{6!}=\frac{24}{720}=\frac{1}{30}$. Therefore,
\begin{equation*}
\Theta (x)=\frac{1}{6}+\frac{x}{45}+\frac{2x^{2}}{3}\int_{0}^{1}\frac{\tau
^{2}\left( 1-\tau \right) ^{5}}{1-\tau \left( 1-\tau \right) x}d\tau \geq
\frac{1}{6}+\frac{x}{45},\qquad 0\leq x\leq 2,
\end{equation*}
with strict inequality for $x>0$, since the last integrand is positive. Taking $x=d^{2}$, we obtain $d^{4}\Theta (d^{2})\geq \frac{1}{6}d^{4}+\frac{1}{45}d^{6}$. This is the second inequality in \eqref{majorant}.

\emph{Step 5: optimality of $\frac{1}{6}$.} We consider the two-point space $\{-1,1\}$ with the uniform measure and the function $g_{\varepsilon }(\sigma )=\left( 1+\varepsilon \sigma \right) /\sqrt{1+\varepsilon ^{2}}$, $0<\varepsilon <1$. It is positive and satisfies $\left\langle g_{\varepsilon }^{2}\right\rangle =1$. Then $g_{\varepsilon }-1=\varepsilon \sigma -\frac{1}{2}\varepsilon ^{2}+O(\varepsilon ^{3})$, and so $\left( g_{\varepsilon }-1\right) ^{4}=\varepsilon ^{4}+O(\varepsilon ^{5})$. Moreover, $H(1+\eta )=\frac{1}{6}\eta ^{4}+O(\eta ^{5})$ by the fourth-order Taylor expansion (see Step 1 in Section \ref{SSpower}). So we get $\left\langle H(g_{\varepsilon })\right\rangle =\frac{1}{6}\varepsilon ^{4}+O(\varepsilon ^{5})$. On the other hand, $\left\Vert g_{\varepsilon }-1\right\Vert _{2}^{2}=2\left( 1-\left( 1+\varepsilon ^{2}\right) ^{-1/2}\right) =\varepsilon ^{2}+O(\varepsilon ^{4})$, and so $\left\Vert g_{\varepsilon }-1\right\Vert _{2}^{4}=\varepsilon ^{4}+O(\varepsilon ^{6})$. Therefore, the quotient tends to $\frac{1}{6}$. As a consequence, no inequality $\left\langle H(u)\right\rangle \geq c\left\Vert u-1\right\Vert _{2}^{4}$ with $c>\frac{1}{6}$ can hold on all probability spaces. This completes the proof of Theorem \ref{Tmajorantintro}.
\end{proof}

\subsection{Proof of \eqref{log4}}\label{SSlog4}

Let $n\in \mathcal{J}$ and $u\in \mathcal{A}_{n}$. All terms in \eqref{splitting} are finite. On a finite cycle this is obvious, and on the circle $u$ is continuous by $H^{1}(\mathbb{T})\subset C^{0}(\mathbb{T})$. By \eqref{splitting}, Proposition \ref{Pcubic} and Theorem \ref{Tmajorantintro}, we have
\begin{equation*}
\Delta _{n}(u)=2\mathcal{F}_{n}(u)+\left\langle H(u)\right\rangle \geq 2\cdot
\frac{1}{12}d_{n}(u)^{4}+\frac{1}{6}d_{n}(u)^{4}+\frac{1}{45}d_{n}(u)^{6}=
\frac{1}{3}d_{n}(u)^{4}+\frac{1}{45}d_{n}(u)^{6},
\end{equation*}
which is \eqref{log4}.

\subsection{Optimality of the power four}\label{SSpower}

We show that for every fixed $n\in \mathcal{I}$ and $0<p<4$,
\begin{equation}
\inf_{u\in \mathcal{A}_{n}\setminus \{1\}}\frac{\Delta _{n}(u)}{d_{n}(u)^{p}}
=0,\qquad \text{and}\qquad \inf_{u\in \mathcal{A}_{n}\setminus \{1\}}\frac{
\Delta _{\mathrm{w},n}(u)}{d_{n}(u)^{p}}=0\quad (n<\infty ).  \label{logpower}
\end{equation}

\emph{Step 1: a Taylor bound for $H$.} We have $H(1)=H^{\prime }(1)=H^{\prime \prime }(1)=H^{\prime \prime \prime }(1)=0$, since $H=P_{3}-\varphi $ and $P_{3}$ is the cubic Taylor polynomial. Moreover, $H^{(4)}(w)=-\varphi ^{(4)}(w)=\frac{4}{w^{2}}$, so $H^{(4)}(1)=4$, and $H^{(5)}(w)=-\varphi ^{(5)}(w)=-\frac{8}{w^{3}}$. For $\left\vert \eta \right\vert \leq \frac{1}{2}$, Taylor's formula with Lagrange remainder of order five gives
\begin{align*}
\left\vert H(1+\eta )-\frac{H^{(4)}(1)}{4!}\eta ^{4}\right\vert & =\left\vert H(1+\eta )-\frac{1}{6}\eta ^{4}\right\vert \leq \frac{\max_{\left[ 1/2,3/2\right] }\left\vert H^{(5)}\right\vert }{5!}\left\vert \eta \right\vert ^{5} \\
& =\frac{8\cdot 8}{120}\left\vert \eta \right\vert ^{5}=\frac{8}{15}\left\vert \eta \right\vert ^{5},
\end{align*}
since $\left\vert H^{(5)}(w)\right\vert =\frac{8}{w^{3}}\leq 64$ on $\left[ \frac{1}{2},\frac{3}{2}\right] $.

\emph{Step 2: the family.} We consider $u_{\varepsilon }$ from \eqref{testfamily}. We have $\left\Vert e_{n}\right\Vert _{\infty }\leq \sqrt{2}$ and $a_{\varepsilon }\leq 1$. So the range of $u_{\varepsilon }$ is contained in $\left[ a_{\varepsilon }\left( 1-\sqrt{2}\varepsilon \right) ,1+\sqrt{2}\varepsilon \right] \subset \left[ \frac{1}{2},\frac{3}{2}\right] $ for $\varepsilon <\frac{1}{4}$. Here we use $a_{\varepsilon }\geq 1-\frac{\varepsilon ^{2}}{2}\geq \frac{31}{32}$, which we now prove. Indeed, $1-a_{\varepsilon }=1-\left( 1+\varepsilon ^{2}\right) ^{-1/2}\leq \frac{\varepsilon ^{2}}{2}$. This holds because $\left( 1+x\right) ^{-1/2}\geq 1-\frac{x}{2}$ for $x\geq 0$, by convexity. In addition,
\begin{equation*}
u_{\varepsilon }-1-\varepsilon e_{n}=\left( a_{\varepsilon }-1\right)
+\varepsilon \left( a_{\varepsilon }-1\right) e_{n},\qquad \left\Vert
u_{\varepsilon }-1-\varepsilon e_{n}\right\Vert _{\infty }\leq \frac{
\varepsilon ^{2}}{2}\left( 1+\sqrt{2}\varepsilon \right) \leq \varepsilon
^{2}.
\end{equation*}
Therefore, $\left( u_{\varepsilon }-1\right) ^{4}=\varepsilon ^{4}e_{n}^{4}+O(\varepsilon ^{5})$ uniformly in the point and in the model. By Step 1, we get
\begin{equation*}
\left\langle H(u_{\varepsilon })\right\rangle =\frac{\varepsilon ^{4}}{6}
\left\langle e_{n}^{4}\right\rangle +O(\varepsilon ^{5}),\qquad \left\langle
e_{n}^{4}\right\rangle \leq \left\Vert e_{n}\right\Vert _{\infty
}^{2}\left\Vert e_{n}\right\Vert _{2}^{2}\leq 2.
\end{equation*}
Combining this with the splitting \eqref{splitting} and \eqref{testorders}, we obtain
\begin{equation*}
\Delta _{n}(u_{\varepsilon })=2\mathcal{F}_{n}(u_{\varepsilon })+\left\langle
H(u_{\varepsilon })\right\rangle =\left( \frac{1}{2}+\frac{1}{6}\left\langle
e_{n}^{4}\right\rangle \right) \varepsilon ^{4}+O(\varepsilon
^{5})=O(\varepsilon ^{4}),\qquad d_{n}(u_{\varepsilon })\sim \varepsilon .
\end{equation*}
For finite $n$, the first harmonic also has word-length eigenvalue $\psi _{n}(1)=1$. Therefore, $\mathcal{E}_{\mathrm{w},n}(u_{\varepsilon })=\mathcal{E}_{n}(u_{\varepsilon })$ and $\Delta _{\mathrm{w},n}(u_{\varepsilon })=\Delta _{n}(u_{\varepsilon })=O(\varepsilon ^{4})$.

\emph{Step 3: conclusion.} For $0<p<4$, both quotients $\Delta _{n}(u_{\varepsilon })/d_{n}(u_{\varepsilon })^{p}$ and $\Delta _{\mathrm{w},n}(u_{\varepsilon })/d_{n}(u_{\varepsilon })^{p}$ are $O(\varepsilon ^{4-p})\rightarrow 0$. These quotients are nonnegative, since $\Delta _{n}\geq 0$ by Theorem \ref{TA} and $\Delta _{\mathrm{w},n}\geq 0$ by Theorem \ref{TB}. Therefore, \eqref{logpower} follows.

\subsection{Proof of the lower bound \eqref{sepnn}}\label{SSsepnn}

Let $n\in \mathcal{J}$, $u\in \mathcal{A}_{n}\setminus \{1\}$, $a=\langle u\rangle <1$ and $d=d_{n}(u)>0$. We consider two ranges of the mean separately.

\emph{Small and moderate means: $a\leq \frac{3}{5}$.} In this case, $d^{2}=2\left( 1-a\right) \geq \frac{4}{5}$. So the sixth-order term in \eqref{log4} is at least a fixed fraction of the fourth-order term, and we have
\begin{equation*}
\frac{\Delta _{n}(u)}{d^{4}}\geq \frac{1}{3}+\frac{d^{2}}{45}\geq \frac{1}{3}+
\frac{4/5}{45}=\frac{1}{3}+\frac{4}{225}=\frac{79}{225}.
\end{equation*}

\emph{Large means: $a\geq \frac{3}{5}$.} In this case, $d$ may be small and the $d^{6}$ term is not useful. Instead, we use the retained term $B(a)$. By \eqref{cubicB} and Theorem \ref{Tmajorantintro} in the splitting \eqref{splitting}, we have
\begin{equation*}
\Delta _{n}(u)\geq 2\left( \frac{1}{12}d^{4}+B(a)\right) +\frac{1}{6}d^{4}+
\frac{1}{45}d^{6}=\frac{1}{3}d^{4}+\frac{1}{45}d^{6}+2B(a).
\end{equation*}
Since $d^{4}=4\left( 1-a\right) ^{2}$, we have $\frac{2B(a)}{d^{4}}=\frac{2\left( 1-a\right) ^{2}\left( 5-3a\right) \left( 5a-3\right) }{96\cdot 4\left( 1-a\right) ^{2}}=\frac{\left( 5-3a\right) \left( 5a-3\right) }{192}$ and $\frac{d^{6}}{45d^{4}}=\frac{d^{2}}{45}=\frac{2\left( 1-a\right) }{45}$. Therefore,
\begin{equation*}
\frac{\Delta _{n}(u)}{d^{4}}\geq \frac{1}{3}+\Psi (a),\qquad \Psi (a)=\frac{2\left(
1-a\right) }{45}+\frac{\left( 5-3a\right) \left( 5a-3\right) }{192}.
\end{equation*}
The function $\Psi $ is a quadratic polynomial, and the coefficient of $a^{2}$ is $-\frac{15}{192}<0$. So it is concave. A concave function on an interval is at least the minimum of its values at the endpoints. We compute $\Psi \left( \frac{3}{5}\right) =\frac{2\cdot \frac{2}{5}}{45}+0=\frac{4}{225}$ and $\Psi (1)=0+\frac{2\cdot 2}{192}=\frac{1}{48}$, and we note that $\frac{1}{48}\approx 0.0208>\frac{4}{225}\approx 0.0178$. Therefore, $\Psi (a)\geq \frac{4}{225}$ on $\left[ \frac{3}{5},1\right] $.

The two ranges of $a$ cover the whole admissible class. This proves \eqref{sepnn}.

\section{Gross's two-point inequality and the four-cycle -- \texorpdfstring{\mbox{Proof of Theorem~\ref{Ttwopoint}}}{Proof of Theorem \ref{Ttwopoint}}}

\subsection{The two-point inequality}

\begin{proof}[Proof of Theorem \ref{Ttwopoint}]
\emph{Step 1: parametrization.} Let $g\geq 0$ on $\{-1,1\}$ with $\left\Vert g\right\Vert _{2}^{2}=\frac{1}{2}\left( g(1)^{2}+g(-1)^{2}\right) =1$. Exchanging the two points changes neither $\delta _{2}(g)$ nor $\left\Vert g-1\right\Vert _{2}$. So we may assume $g(1)\geq g(-1)$. Since $g\not\equiv 0$, we have $g(1)+g(-1)>0$, and we set
\begin{equation*}
\rho =\frac{g(1)-g(-1)}{g(1)+g(-1)}\in \left[ 0,1\right] ,\qquad c_{0}=\frac{
g(1)+g(-1)}{2}>0.
\end{equation*}
Then $g(\pm 1)=c_{0}\left( 1\pm \rho \right) $, and the normalization gives $c_{0}^{2}\frac{\left( 1+\rho \right) ^{2}+\left( 1-\rho \right) ^{2}}{2}=c_{0}^{2}\left( 1+\rho ^{2}\right) =1$. Therefore,
\begin{equation}
g(\pm 1)=\frac{1\pm \rho }{\sqrt{1+\rho ^{2}}},\qquad 0\leq \rho \leq 1.
\label{rhorep}
\end{equation}
The value $\rho =0$ corresponds to $g\equiv 1$, and $\rho =1$ corresponds to $g=(\sqrt{2},0)$, which vanishes at one point. We write $y=\sqrt{1+\rho ^{2}}\in \left[ 1,\sqrt{2}\right] $.

\emph{Step 2: the three quantities as functions of $\rho $.} The mean is $\langle g\rangle =\frac{1}{2}\cdot \frac{\left( 1+\rho \right) +\left( 1-\rho \right) }{y}=\frac{1}{y}$. So
\begin{equation}
D(\rho ):=\left\Vert g-1\right\Vert _{2}^{2}=\left\langle g^{2}\right\rangle
-2\langle g\rangle +1=2\left( 1-\frac{1}{y}\right) ,\qquad D^{\prime }(\rho )=
\frac{2y^{\prime }}{y^{2}}=\frac{2\rho }{y^{3}},  \label{Drho}
\end{equation}
where we used $y^{\prime }=\frac{\rho }{y}$. Twice the energy is $2\mathcal{E}_{2}(g)=\frac{1}{2}\left( g(1)-g(-1)\right) ^{2}=\frac{1}{2}\left( \frac{2\rho }{y}\right) ^{2}=\frac{2\rho ^{2}}{1+\rho ^{2}}$. For the entropy, since $\left\langle g^{2}\right\rangle =1$, we have $\Ent(g^{2})=\left\langle g^{2}\log g^{2}\right\rangle $. Moreover, $\log g(\pm 1)^{2}=2\log \left( 1\pm \rho \right) -\log \left( 1+\rho ^{2}\right) $. Therefore,
\begin{align*}
\Ent(g^{2})& =\frac{1}{2}\sum_{\pm }\frac{\left( 1\pm \rho \right) ^{2}}{
1+\rho ^{2}}\left( 2\log \left( 1\pm \rho \right) -\log \left( 1+\rho
^{2}\right) \right) \\
& =\frac{1}{1+\rho ^{2}}\sum_{\pm }\left( 1\pm \rho \right) ^{2}\log \left(
1\pm \rho \right) -\log \left( 1+\rho ^{2}\right) \cdot \frac{1}{2}\sum_{\pm }
\frac{\left( 1\pm \rho \right) ^{2}}{1+\rho ^{2}} \\
& =\frac{L(\rho )}{1+\rho ^{2}}-\log \left( 1+\rho ^{2}\right) ,
\end{align*}
where $L(\rho )=\left( 1+\rho \right) ^{2}\log \left( 1+\rho \right) +\left( 1-\rho \right) ^{2}\log \left( 1-\rho \right) $. At $\rho =1$, we define this expression by continuity, because $\left( 1-\rho \right) ^{2}\log \left( 1-\rho \right) \rightarrow 0$ there. We also used $\frac{1}{2}\sum_{\pm }\left( 1\pm \rho \right) ^{2}=1+\rho ^{2}$. Therefore,
\begin{equation}
\delta (\rho ):=\delta _{2}(g)=2\mathcal{E}_{2}(g)-\Ent(g^{2})=\frac{2\rho
^{2}-L(\rho )}{1+\rho ^{2}}+\log \left( 1+\rho ^{2}\right) .  \label{deltarho}
\end{equation}
As a check, at $\rho =1$ we have $L(1)=4\log 2$ and $\delta (1)=\frac{2-4\log 2}{2}+\log 2=1-\log 2>0$.

\emph{Step 3: the derivative of $\delta $.} We have
\begin{align*}
L^{\prime }(\rho )& =2\left( 1+\rho \right) \log \left( 1+\rho \right)
+\left( 1+\rho \right) -2\left( 1-\rho \right) \log \left( 1-\rho \right)
-\left( 1-\rho \right) \\
& =2\left( 1+\rho \right) \log \left( 1+\rho \right) -2\left( 1-\rho
\right) \log \left( 1-\rho \right) +2\rho .
\end{align*}
We write $\delta =\frac{N}{y^{2}}+\log y^{2}$ with $N=2\rho ^{2}-L$ and $y^{2}=1+\rho ^{2}$, so that $\left( y^{2}\right) ^{\prime }=2\rho $. By the quotient rule, we get
\begin{equation*}
\delta ^{\prime }=\frac{N^{\prime }y^{2}-2\rho N}{y^{4}}+\frac{2\rho }{y^{2}}
=\frac{N^{\prime }y^{2}-2\rho N+2\rho y^{2}}{y^{4}}.
\end{equation*}
We now compute the numerator. Let us write $\ell _{+}=\log \left( 1+\rho \right) $ and $\ell _{-}=\log \left( 1-\rho \right) $. We have $N^{\prime }=4\rho -L^{\prime }=2\rho -2\left( 1+\rho \right) \ell _{+}+2\left( 1-\rho \right) \ell _{-}$ and $N=2\rho ^{2}-\left( 1+\rho \right) ^{2}\ell _{+}-\left( 1-\rho \right) ^{2}\ell _{-}$. We collect the coefficients of $\ell _{+}$, $\ell _{-}$ and the remaining terms as follows:
\begin{align*}
\text{coefficient of }\ell _{+}:& \quad -2\left( 1+\rho \right) \left(
1+\rho ^{2}\right) +2\rho \left( 1+\rho \right) ^{2}=2\left( 1+\rho \right)
\left( \rho -1\right) =-2\left( 1-\rho ^{2}\right) , \\
\text{coefficient of }\ell _{-}:& \quad 2\left( 1-\rho \right) \left(
1+\rho ^{2}\right) +2\rho \left( 1-\rho \right) ^{2}=2\left( 1-\rho \right)
\left( 1+\rho \right) =2\left( 1-\rho ^{2}\right) , \\
\text{remaining terms}:& \quad 2\rho \left( 1+\rho ^{2}\right) -2\rho \cdot
2\rho ^{2}+2\rho \left( 1+\rho ^{2}\right) =4\rho .
\end{align*}
So the numerator equals $4\rho -2\left( 1-\rho ^{2}\right) \left( \ell _{+}-\ell _{-}\right) $. We recall that
\begin{equation*}
\operatorname{artanh}\rho =\frac{1}{2}\log \frac{1+\rho }{1-\rho }=\frac{1}{2}\left(
\ell _{+}-\ell _{-}\right) =\int_{0}^{\rho }\frac{d\sigma }{1-\sigma ^{2}}.
\end{equation*}
The integral formula follows from $\frac{d}{d\rho }\frac{1}{2}\log \frac{1+\rho }{1-\rho }=\frac{1}{2}\left( \frac{1}{1+\rho }+\frac{1}{1-\rho }\right) =\frac{1}{1-\rho ^{2}}$ and $\operatorname{artanh}0=0$. Then we obtain, for $0<\rho <1$,
\begin{equation}
\delta ^{\prime }(\rho )=\frac{4}{y^{4}}\left( \rho -\left( 1-\rho
^{2}\right) \operatorname{artanh}\rho \right) .  \label{deltaprime}
\end{equation}

\emph{Step 4: an integral decomposition.} Since $\frac{\rho ^{2}-\sigma ^{2}}{1-\sigma ^{2}}=\frac{\left( \rho ^{2}-1\right) +\left( 1-\sigma ^{2}\right) }{1-\sigma ^{2}}=1-\frac{1-\rho ^{2}}{1-\sigma ^{2}}$, we have
\begin{equation*}
\int_{0}^{\rho }\frac{\rho ^{2}-\sigma ^{2}}{1-\sigma ^{2}}d\sigma =\rho
-\left( 1-\rho ^{2}\right) \operatorname{artanh}\rho .
\end{equation*}
On the other hand, $\frac{\rho ^{2}-\sigma ^{2}}{1-\sigma ^{2}}=\left( \rho ^{2}-\sigma ^{2}\right) \left( 1+\frac{\sigma ^{2}}{1-\sigma ^{2}}\right) $ and $\int_{0}^{\rho }\left( \rho ^{2}-\sigma ^{2}\right) d\sigma =\rho ^{3}-\frac{\rho ^{3}}{3}=\frac{2}{3}\rho ^{3}$. Therefore,
\begin{equation}
\rho -\left( 1-\rho ^{2}\right) \operatorname{artanh}\rho =\frac{2}{3}\rho ^{3}+J(\rho
),\qquad J(\rho )=\int_{0}^{\rho }\frac{\sigma ^{2}\left( \rho ^{2}-\sigma
^{2}\right) }{1-\sigma ^{2}}d\sigma .  \label{Jdecomp}
\end{equation}
For $0\leq \sigma \leq \rho <1$, the integrand of $J$ is nonnegative. Moreover, since $\frac{1}{1-\sigma ^{2}}\geq 1$, we get
\begin{equation}
J(\rho )\geq \int_{0}^{\rho }\sigma ^{2}\left( \rho ^{2}-\sigma ^{2}\right)
d\sigma =\frac{\rho ^{5}}{3}-\frac{\rho ^{5}}{5}=\frac{2}{15}\rho ^{5}.
\label{Jbound}
\end{equation}

\emph{Step 5: comparison of derivatives.} Let $Z(\rho )=\delta (\rho )-\frac{2}{3}D(\rho )^{2}$. By \eqref{Drho}, we have $\left( D^{2}\right) ^{\prime }=2DD^{\prime }=2\cdot 2\left( 1-\frac{1}{y}\right) \cdot \frac{2\rho }{y^{3}}=\frac{8\rho \left( y-1\right) }{y^{4}}$. By \eqref{deltaprime} and \eqref{Jdecomp}, we obtain
\begin{equation*}
Z^{\prime }(\rho )=\frac{4}{y^{4}}\left( \frac{2}{3}\rho ^{3}+J(\rho )\right)
-\frac{2}{3}\cdot \frac{8\rho \left( y-1\right) }{y^{4}}=\frac{8\rho }{3y^{4}}
\left( \rho ^{2}-2\left( y-1\right) \right) +\frac{4}{y^{4}}J(\rho ).
\end{equation*}
Since $\rho ^{2}=y^{2}-1$, we have $\rho ^{2}-2\left( y-1\right) =y^{2}-2y+1=\left( y-1\right) ^{2}$. Using \eqref{Jbound}, we get
\begin{equation*}
Z^{\prime }(\rho )\geq \frac{8\rho \left( y-1\right) ^{2}}{3y^{4}}+\frac{
8\rho ^{5}}{15y^{4}}.
\end{equation*}
Now $\rho ^{4}=\left( y^{2}-1\right) ^{2}=\left( y-1\right) ^{2}\left( y+1\right) ^{2}$. Therefore, $\frac{8\rho ^{5}}{15y^{4}}=\frac{8\rho \left( y-1\right) ^{2}\left( y+1\right) ^{2}}{15y^{4}}$, and
\begin{equation*}
Z^{\prime }(\rho )\geq \frac{8\rho \left( y-1\right) ^{2}}{y^{4}}\left( \frac{1
}{3}+\frac{\left( y+1\right) ^{2}}{15}\right) =\frac{8\rho \left( y-1\right)
^{2}}{15y^{4}}\left( 5+\left( y+1\right) ^{2}\right) .
\end{equation*}
On the other hand, $\left( D^{3}\right) ^{\prime }=3D^{2}D^{\prime }=3\cdot 4\left( 1-\frac{1}{y}\right) ^{2}\cdot \frac{2\rho }{y^{3}}=\frac{24\rho \left( y-1\right) ^{2}}{y^{5}}$. Dividing, we obtain
\begin{equation*}
\frac{Z^{\prime }(\rho )}{\left( D^{3}\right) ^{\prime }(\rho )}\geq \frac{
8/(15y^{4})}{24/y^{5}}\left( 5+\left( y+1\right) ^{2}\right) =\frac{y\left(
5+\left( y+1\right) ^{2}\right) }{45}\geq \frac{1\cdot \left( 5+4\right) }{45}
=\frac{1}{5},
\end{equation*}
since $y\geq 1$. The inequality is strict for $\rho >0$, where $y>1$. So $Z^{\prime }(\rho )\geq \frac{1}{5}\left( D^{3}\right) ^{\prime }(\rho )$ on $(0,1)$, with strict inequality.

\emph{Step 6: integration and equality case.} We integrate from $0$ to $\rho $ and use $Z(0)=D(0)=0$. Then
\begin{equation*}
\delta (\rho )-\frac{2}{3}D(\rho )^{2}-\frac{1}{5}D(\rho )^{3}=\int_{0}^{\rho
}\left( Z^{\prime }-\frac{1}{5}\left( D^{3}\right) ^{\prime }\right) >0,\qquad
0<\rho <1,
\end{equation*}
which is \eqref{twopoint} with strict inequality for $0<\rho <1$. Both sides are continuous on $\left[ 0,1\right] $, and the function $Z-\frac{1}{5}D^{3}$ is strictly increasing on $(0,1)$. Therefore, its value at $\rho =1$ is also positive. Equality holds only at $\rho =0$, that is, for $g=1$.

\emph{Step 7: sharpness of the constants.} We use \eqref{Jdecomp} and \eqref{Jbound}. Since $\frac{1}{1-\sigma ^{2}}\leq \frac{1}{1-\rho ^{2}}$ on $\left[ 0,\rho \right] $, we have $\frac{2}{15}\rho ^{5}\leq J(\rho )\leq \frac{2\rho ^{5}}{15\left( 1-\rho ^{2}\right) }$. Therefore, $J(\rho )=O(\rho ^{5})$ as $\rho \rightarrow 0$. Then \eqref{deltaprime} gives $\delta ^{\prime }(\rho )=\frac{4}{y^{4}}\left( \frac{2}{3}\rho ^{3}+O(\rho ^{5})\right) =\frac{8}{3}\rho ^{3}+O(\rho ^{5})$, and so $\delta (\rho )=\frac{2}{3}\rho ^{4}+O(\rho ^{6})$. In addition, $D(\rho )=2\left( 1-\left( 1+\rho ^{2}\right) ^{-1/2}\right) =\rho ^{2}+O(\rho ^{4})$. Therefore, $\delta /D^{2}\rightarrow \frac{2}{3}$, and the constant $\frac{2}{3}$ cannot be increased. We now consider the constant $\frac{1}{5}$. The coefficient $\frac{y\left( 5+\left( y+1\right) ^{2}\right) }{45}$ tends to $\frac{1}{5}$ as $\rho \rightarrow 0$. Moreover, the difference $J(\rho )-\frac{2}{15}\rho ^{5}=\int_{0}^{\rho }\frac{\sigma ^{4}\left( \rho ^{2}-\sigma ^{2}\right) }{1-\sigma ^{2}}d\sigma =O(\rho ^{7})$. The only inequality that we used in Step 5 is \eqref{Jbound}. So we have the exact identity
\begin{equation*}
Z^{\prime }(\rho )=\frac{8\rho \left( y-1\right) ^{2}}{15y^{4}}\left( 5+\left( y+1\right) ^{2}\right) +\frac{4}{y^{4}}\Big( J(\rho )-\frac{2}{15}\rho ^{5}\Big) .
\end{equation*}
Moreover, $\left( D^{3}\right) ^{\prime }(\rho )=\frac{24\rho \left( y-1\right) ^{2}}{y^{5}}=6\rho ^{5}+O(\rho ^{7})$, since $y-1=\frac{\rho ^{2}}{2}+O(\rho ^{4})$. We divide and get $\frac{Z^{\prime }(\rho )}{\left( D^{3}\right) ^{\prime }(\rho )}=\frac{y\left( 5+\left( y+1\right) ^{2}\right) }{45}+O(\rho ^{2})=\frac{1}{5}+O(\rho ^{2})$. By integration, we get $\delta -\frac{2}{3}D^{2}=\frac{1}{5}D^{3}+O(\rho ^{8})$. So $\frac{1}{5}$ cannot be increased either.

\emph{Step 8: the homogeneous form.} Let $g\geq 0$ be arbitrary and nonzero. Let $M=\left\Vert g\right\Vert _{2}$ and $\widetilde{g}=g/M$, so that $\left\Vert \widetilde{g}\right\Vert _{2}=1$. The energy and the entropy are homogeneous of degree two, that is, $\mathcal{E}_{2}(Mg)=M^{2}\mathcal{E}_{2}(g)$ and $\Ent(M^{2}g^{2})=M^{2}\Ent(g^{2})$. The second identity holds because $\left\langle M^{2}g^{2}\log \left( M^{2}g^{2}\right) \right\rangle -\left\langle M^{2}g^{2}\right\rangle \log \left\langle M^{2}g^{2}\right\rangle =M^{2}\Ent(g^{2})+M^{2}\left\langle g^{2}\right\rangle \log M^{2}-M^{2}\left\langle g^{2}\right\rangle \log M^{2}$. Therefore, $\delta _{2}(g)=M^{2}\delta _{2}(\widetilde{g})$. Moreover,
\begin{equation*}
\left\Vert \widetilde{g}-1\right\Vert _{2}^{2}=2\left( 1-\left\langle
\widetilde{g}\right\rangle \right) =2\left( 1-\frac{\langle g\rangle }{M}
\right) =\frac{2\xi }{M},\qquad \xi =M-\langle g\rangle \geq 0,
\end{equation*}
where $\xi \geq 0$ by the Cauchy--Schwarz inequality $\langle g\rangle \leq \left\Vert g\right\Vert _{2}$. We apply \eqref{twopoint} to $\widetilde{g}$ and multiply by $M^{2}$. Then
\begin{equation*}
\delta _{2}(g)\geq M^{2}\left( \frac{2}{3}\cdot \frac{4\xi ^{2}}{M^{2}}+\frac{
1}{5}\cdot \frac{8\xi ^{3}}{M^{3}}\right) =\frac{8}{3}\xi ^{2}+\frac{8}{5}
\frac{\xi ^{3}}{M},
\end{equation*}
which is \eqref{twopointhom}. Equality holds if and only if $\widetilde{g}=1$, that is, if and only if $g$ is a positive constant. This completes the proof of Theorem \ref{Ttwopoint}.
\end{proof}

\begin{remark}
The inequality $\delta \geq \frac{2}{3}D^{2}$ can also be studied by power series. We expand both sides in powers of $x_{0}=\frac{1}{2}\left( g(1)^{2}-g(-1)^{2}\right) =\frac{2\rho }{1+\rho ^{2}}$ and compare the coefficients. The two series agree up to order $x_{0}^{4}$, and $\delta -\frac{2}{3}D^{2}=\frac{x_{0}^{6}}{320}+O(x_{0}^{8})$. Since $\frac{1}{5}D^{3}=\frac{x_{0}^{6}}{320}+O(x_{0}^{8})$, the sixth-order term in \eqref{twopoint} is exactly the first nonzero coefficient of this expansion. The advantage of the integral argument above is that it gives the sixth-order term globally.
\end{remark}

\subsection{The four-cycle as a product of two two-point spaces}

We identify the vertices $0,1,2,3$ of $C_{4}$ with the points
\begin{equation*}
0\mapsto (1,1),\qquad 1\mapsto (-1,1),\qquad 2\mapsto (-1,-1),\qquad
3\mapsto (1,-1)
\end{equation*}
of $\{-1,1\}^{2}$. Two consecutive vertices differ in exactly one coordinate. Indeed, $0\rightarrow 1$ changes the first coordinate, $1\rightarrow 2$ changes the second, $2\rightarrow 3$ changes the first, and $3\rightarrow 0$ changes the second. Therefore, the four edges of the cycle are the four edges of the square, with two edges in each coordinate direction. In the rest of this section, we write the two coordinates as $s,t\in \{-1,1\}$. A subscript indicates the coordinate over which the mean is taken. In the same way, a subscript on $\mathcal{E}_{2}$, $\Ent$ or $\left\Vert \cdot \right\Vert _{2}$ means that this quantity is computed in this coordinate. For example, $\left\langle F\right\rangle _{t}=\frac{1}{2}\left( F(1)+F(-1)\right) $ for a function $F$ of $t$.

Since $\mu _{4}=2\left( 1-\cos \frac{\pi }{2}\right) =2$, the energy \eqref{energy} is $\mathcal{E}_{4}(u)=\frac{1}{4\cdot 2}\sum_{\text{edges}}\left( \text{difference}\right) ^{2}=\frac{1}{8}\sum_{\text{edges}}\left( \text{difference}\right) ^{2}$. The two edges in the $s$-direction join $(1,t)$ and $(-1,t)$ for $t=\pm 1$, and we have
\begin{equation*}
\frac{1}{8}\sum_{t=\pm 1}\left( u(1,t)-u(-1,t)\right) ^{2}=\frac{1}{2}
\sum_{t=\pm 1}\frac{1}{4}\left( u(1,t)-u(-1,t)\right) ^{2}=\left\langle
\mathcal{E}_{2,s}\left( u(\cdot ,t)\right) \right\rangle _{t}.
\end{equation*}
The same holds for the $t$-direction. Therefore,
\begin{equation}
\mathcal{E}_{4}(u)=\left\langle \mathcal{E}_{2,s}\left( u(\cdot ,t)\right)
\right\rangle _{t}+\left\langle \mathcal{E}_{2,t}\left( u(s,\cdot )\right)
\right\rangle _{s}.  \label{E4product}
\end{equation}
This expresses $\mathcal{E}_{4}$ as the sum of two two-point energies, exactly as for a product space. We note that $C_{4}$ is the only cycle that is a Cartesian product of smaller graphs. Indeed, let $n\geq 3$ and let $C_{n}$ be the Cartesian product of two graphs, each with at least two vertices. Since $C_{n}$ is connected, both factors are connected. So every vertex has degree at least one in each factor. In the product, the degree of a vertex is the sum of its degrees in the two factors, and this sum is two. So every vertex has degree one in each factor. A connected graph in which every vertex has degree one has two vertices and one edge. Therefore, both factors are this graph, and their product is $C_{4}$. We recall that the word-length energy coincides with $\mathcal{E}_{4}$ on $C_{4}$, since both have eigenvalues $(0,1,2,1)$ for $k=0,1,2,3$. Indeed, $\ell _{4}=(0,1,2,1)$ by \eqref{eigen} and $\psi _{4}(k)=\min (k,4-k)=(0,1,2,1)$.

For $u\in \mathcal{A}_{4}$, we set
\begin{equation*}
g_{t}(s)=u(s,t),\qquad h(t)=\left\Vert g_{t}\right\Vert _{2,s}=\sqrt{\tfrac{1
}{2}\left( u(1,t)^{2}+u(-1,t)^{2}\right) },
\end{equation*}
\begin{equation*}
m(t)=\left\langle g_{t}\right\rangle _{s},\qquad a=\langle u\rangle ,\qquad
b=\left\langle h\right\rangle _{t}.
\end{equation*}
So $g_{t}$ is the ``slice'' of $u$ at height $t$, and it is a two-point function of $s$. The function $h$ of $t$ collects the norms of the slices. Then $\left\langle h^{2}\right\rangle _{t}=\left\langle u^{2}\right\rangle =1$, so $\left\Vert h\right\Vert _{2,t}=1$. Moreover, $0\leq m(t)\leq h(t)$ by the Cauchy--Schwarz inequality on each slice. Taking means in $t$, we get $a=\left\langle m\right\rangle _{t}\leq \left\langle h\right\rangle _{t}=b\leq \left\Vert h\right\Vert _{2,t}=1$. Therefore, $0<a\leq b\leq 1$.

\begin{lemma}
\label{Ldecomp4}For every $u\in \mathcal{A}_{4}$, we have the exact decomposition
\begin{equation}
\Delta _{4}(u)=\left\langle \delta _{2}(g_{t})\right\rangle _{t}+\delta
_{2}(h)+2\mathcal{R}_{\mathrm{slice}},\qquad \mathcal{R}_{\mathrm{slice}}=
\frac{1}{2}\left( h(1)h(-1)-\left\langle g_{1}g_{-1}\right\rangle
_{s}\right) \geq 0.  \label{Delta4decomp}
\end{equation}
\end{lemma}

\begin{proof}
\emph{The entropy.} Since $\left\langle u^{2}\right\rangle =1$, we have $\Ent(u^{2})=\left\langle u^{2}\log u^{2}\right\rangle _{s,t}=\left\langle \left\langle g_{t}^{2}\log g_{t}^{2}\right\rangle _{s}\right\rangle _{t}$. We add and subtract $\left\langle h^{2}\log h^{2}\right\rangle _{t}$ and obtain
\begin{equation}
\Ent(u^{2})=\left\langle \left\langle g_{t}^{2}\log g_{t}^{2}\right\rangle
_{s}-h(t)^{2}\log h(t)^{2}\right\rangle _{t}+\left\langle h^{2}\log
h^{2}\right\rangle _{t}=\left\langle \Ent_{s}(g_{t}^{2})\right\rangle
_{t}+\Ent_{t}(h^{2}).  \label{chainrule}
\end{equation}
Here $\Ent_{s}(g_{t}^{2})=\left\langle g_{t}^{2}\log g_{t}^{2}\right\rangle _{s}-\left\langle g_{t}^{2}\right\rangle _{s}\log \left\langle g_{t}^{2}\right\rangle _{s}$ is the homogeneous entropy of the slice. Note that $\left\langle g_{t}^{2}\right\rangle _{s}=h(t)^{2}$. In addition, $\Ent_{t}(h^{2})=\left\langle h^{2}\log h^{2}\right\rangle _{t}$ because $\left\langle h^{2}\right\rangle _{t}=1$. This is the usual chain rule for the entropy. Because of the convention $0\log 0=0$, it does not require positive entries. Indeed, a zero slice contributes zero to both sides.

\emph{The energy in the $t$-direction.} We have
\begin{align*}
\left\langle \mathcal{E}_{2,t}\left( u(s,\cdot )\right) \right\rangle _{s}& =
\left\langle \frac{1}{4}\left( u(s,1)-u(s,-1)\right) ^{2}\right\rangle _{s}=
\frac{1}{4}\left\Vert g_{1}-g_{-1}\right\Vert _{2,s}^{2} \\
& =\frac{1}{4}\left(
h(1)^{2}+h(-1)^{2}-2\left\langle g_{1}g_{-1}\right\rangle _{s}\right) \\
& =\frac{1}{4}\left( h(1)-h(-1)\right) ^{2}+\frac{1}{2}\left(
h(1)h(-1)-\left\langle g_{1}g_{-1}\right\rangle _{s}\right) \\
& =\mathcal{E}_{2,t}(h)+\mathcal{R}_{\mathrm{slice}},
\end{align*}
where we added and subtracted $\frac{1}{2}h(1)h(-1)$. By the Cauchy--Schwarz inequality, $\left\langle g_{1}g_{-1}\right\rangle _{s}\leq \left\Vert g_{1}\right\Vert _{2,s}\left\Vert g_{-1}\right\Vert _{2,s}=h(1)h(-1)$. Therefore, $\mathcal{R}_{\mathrm{slice}}\geq 0$. This also holds when a slice vanishes.

\emph{Assembling.} By \eqref{E4product} and the two identities above, we have
\begin{align*}
\Delta _{4}(u)& =2\left\langle \mathcal{E}_{2,s}(g_{t})\right\rangle _{t}+2
\mathcal{E}_{2,t}(h)+2\mathcal{R}_{\mathrm{slice}}-\left\langle
\Ent_{s}(g_{t}^{2})\right\rangle _{t}-\Ent_{t}(h^{2}) \\
& =\left\langle 2\mathcal{E}_{2,s}(g_{t})-\Ent_{s}(g_{t}^{2})\right\rangle
_{t}+\left( 2\mathcal{E}_{2,t}(h)-\Ent_{t}(h^{2})\right) +2\mathcal{R}_{
\mathrm{slice}},
\end{align*}
which is \eqref{Delta4decomp}.
\end{proof}

\subsection{Proof of \eqref{log4cycle}, strictness and nonattainment}\label{SSlog4cycle}

For a nonnegative function $g$ on $\{-1,1\}$, we define the remainder of the two-point inequality beyond its sharp quartic term by
\begin{equation*}
\mathcal{R}_{2}(g)=\delta _{2}(g)-\frac{8}{3}\left( \left\Vert g\right\Vert
_{2}-\langle g\rangle \right) ^{2}.
\end{equation*}
By Theorem \ref{Ttwopoint}, we have $\mathcal{R}_{2}(g)\geq 0$ for $g\neq 0$, with equality only for constant $g$. Moreover, $\mathcal{R}_{2}(0)=0$.

\begin{proposition}
\label{Pexact4}For every $u\in \mathcal{A}_{4}$, we have
\begin{equation}
\Delta _{4}(u)-\frac{1}{3}d_{4}(u)^{4}=\left\langle \mathcal{R}
_{2}(g_{t})\right\rangle _{t}+\mathcal{R}_{2}(h)+2\mathcal{R}_{\mathrm{slice}}+
\frac{8}{3}\left\langle \left( h-m-\left( b-a\right) \right)
^{2}\right\rangle _{t}+\frac{4}{3}\left( 2b-a-1\right) ^{2}.  \label{exact4}
\end{equation}
Every term on the right-hand side is nonnegative.
\end{proposition}

\begin{proof}
We apply the definition of $\mathcal{R}_{2}$ to $g_{t}$, for which $\left\Vert g_{t}\right\Vert _{2,s}=h(t)$ and $\left\langle g_{t}\right\rangle _{s}=m(t)$, and to $h$, for which $\left\Vert h\right\Vert _{2,t}=1$ and $\left\langle h\right\rangle _{t}=b$. Then
\begin{equation*}
\delta _{2}(g_{t})=\mathcal{R}_{2}(g_{t})+\frac{8}{3}\left( h(t)-m(t)\right)
^{2},\qquad \delta _{2}(h)=\mathcal{R}_{2}(h)+\frac{8}{3}\left( 1-b\right)
^{2}.
\end{equation*}
Since $\left\langle h-m\right\rangle _{t}=b-a$, the variance decomposition gives
\begin{equation*}
\left\langle \left( h-m\right) ^{2}\right\rangle _{t}=\left\langle \left( h-m-\left( b-a\right) \right) ^{2}\right\rangle _{t}+\left( b-a\right) ^{2}.
\end{equation*}
Inserting these identities into Lemma \ref{Ldecomp4}, we obtain
\begin{equation*}
\Delta _{4}(u)=\left\langle \mathcal{R}_{2}(g_{t})\right\rangle _{t}+\mathcal{R
}_{2}(h)+2\mathcal{R}_{\mathrm{slice}}+\frac{8}{3}\left\langle \left(
h-m-\left( b-a\right) \right) ^{2}\right\rangle _{t}+\frac{8}{3}\left(
\left( b-a\right) ^{2}+\left( 1-b\right) ^{2}\right) .
\end{equation*}
Finally, $d_{4}(u)^{4}=4\left( 1-a\right) ^{2}$. Let $X_{0}=b-a$ and $Y_{0}=1-b$, so that $X_{0}+Y_{0}=1-a$. Then the identity $2X_{0}^{2}+2Y_{0}^{2}-\left( X_{0}+Y_{0}\right) ^{2}=\left( X_{0}-Y_{0}\right) ^{2}$ gives
\begin{equation*}
\frac{8}{3}\left( \left( b-a\right) ^{2}+\left( 1-b\right) ^{2}\right) -
\frac{4}{3}\left( 1-a\right) ^{2}=\frac{4}{3}\left( 2X_{0}^{2}+2Y_{0}^{2}-\left(
X_{0}+Y_{0}\right) ^{2}\right) =\frac{4}{3}\left( X_{0}-Y_{0}\right) ^{2},
\end{equation*}
and $X_{0}-Y_{0}=2b-a-1$.
This proves \eqref{exact4}.
\end{proof}

\begin{proof}[Proof of \eqref{log4cycle}]
\emph{Step 1: the sixth-order bounds for the two-point remainders.} Let $X(t)=h(t)-m(t)\geq 0$. We apply \eqref{twopointhom} to $g_{t}$. Here $\left\Vert g_{t}\right\Vert _{2,s}=h(t)$, and $\xi =\left\Vert g_{t}\right\Vert _{2,s}-\left\langle g_{t}\right\rangle _{s}=X(t)$. Then
\begin{equation*}
\mathcal{R}_{2}(g_{t})=\delta _{2}(g_{t})-\frac{8}{3}X(t)^{2}\geq \frac{8}{5}
\frac{X(t)^{3}}{h(t)},
\end{equation*}
where the fraction is defined to be zero on a zero slice. Note that $X=h=0$ on such a slice. We also apply \eqref{twopointhom} to $h$. Here $\left\Vert h\right\Vert _{2,t}=1$ and $\xi =1-b$, and we get $\mathcal{R}_{2}(h)\geq \frac{8}{5}\left( 1-b\right) ^{3}$.

\emph{Step 2: H\"{o}lder's inequality.} We need a lower bound for $\left\langle X^{3}/h\right\rangle _{t}$ in terms of $\left\langle X\right\rangle _{t}=b-a$. On $\{h>0\}$, we write $X=\left( Xh^{-1/3}\right) \cdot h^{1/3}$. We apply H\"{o}lder's inequality with the exponents $3$ and $\frac{3}{2}$. Note that $\frac{1}{3}+\frac{2}{3}=1$. Then
\begin{equation*}
\left\langle X\right\rangle _{t}\leq \left\langle \left( Xh^{-1/3}\right)
^{3}\right\rangle _{t}^{1/3}\left\langle \left( h^{1/3}\right)
^{3/2}\right\rangle _{t}^{2/3}=\left\langle \frac{X^{3}}{h}\right\rangle
_{t}^{1/3}\left\langle \sqrt{h}\right\rangle _{t}^{2/3}.
\end{equation*}
Taking cubes, we get $\left( b-a\right) ^{3}\leq \left\langle X^{3}/h\right\rangle _{t}\left\langle \sqrt{h}\right\rangle _{t}^{2}$. By the Cauchy--Schwarz inequality, $\left\langle \sqrt{h}\right\rangle _{t}^{2}\leq \left\langle h\right\rangle _{t}=b$. Therefore,
\begin{equation*}
\left\langle \frac{X^{3}}{h}\right\rangle _{t}\geq \frac{\left( b-a\right)
^{3}}{b}.
\end{equation*}

\emph{Step 3: assembling.} We insert Steps 1 and 2 into \eqref{exact4} and drop the other three nonnegative terms. Using $b\leq 1$, we get
\begin{align}
\Delta _{4}(u)-\frac{1}{3}d_{4}(u)^{4}& \geq \frac{8}{5}\left( \frac{\left(
b-a\right) ^{3}}{b}+\left( 1-b\right) ^{3}\right) \geq \frac{8}{5}\left(
\left( b-a\right) ^{3}+\left( 1-b\right) ^{3}\right)  \notag \\
& \geq \frac{2}{5}\left( 1-a\right) ^{3}=\frac{1}{20}d_{4}(u)^{6},
\label{sextic}
\end{align}
where the last inequality is the power-mean inequality $X_{0}^{3}+Y_{0}^{3}\geq \frac{1}{4}\left( X_{0}+Y_{0}\right) ^{3}$ for $X_{0}=b-a\geq 0$, $Y_{0}=1-b\geq 0$, $X_{0}+Y_{0}=1-a$. It follows from the identity
\begin{align*}
4\left( X_{0}^{3}+Y_{0}^{3}\right) -\left( X_{0}+Y_{0}\right) ^{3}&
=3X_{0}^{3}+3Y_{0}^{3}-3X_{0}^{2}Y_{0}-3X_{0}Y_{0}^{2} \\
& =3\left( X_{0}-Y_{0}\right) \left( X_{0}^{2}-Y_{0}^{2}\right) =3\left(
X_{0}+Y_{0}\right) \left( X_{0}-Y_{0}\right) ^{2}\geq 0,
\end{align*}
and we also use $\left( 1-a\right) ^{3}=\frac{1}{8}d_{4}(u)^{6}$, so that $\frac{8}{5}\cdot \frac{1}{4}\cdot \frac{1}{8}=\frac{1}{20}$. This proves \eqref{log4cycle}.

\emph{Step 4: strictness.} If some slice $g_{t}$ is nonconstant, then its inequality \eqref{twopointhom} is strict by Theorem \ref{Ttwopoint}. So the first inequality in \eqref{sextic} is strict. Now assume that all slices are constant, but $u$ is not constant. Then $u(s,t)=h(t)$. So $h$ is nonconstant, and the inequality \eqref{twopointhom} for $h$ is strict. Therefore, \eqref{log4cycle} is strict for every nonconstant $u\in \mathcal{A}_{4}$.

\emph{Step 5: nonattainment.} Dividing \eqref{sextic} by $d_{4}(u)^{4}>0$, we get \eqref{minimizingintro}. So, if $u^{(k)}\in \mathcal{A}_{4}\setminus \{1\}$ and $\Delta _{4}(u^{(k)})/d_{4}(u^{(k)})^{4}\rightarrow \frac{1}{3}$, then $d_{4}(u^{(k)})\rightarrow 0$, that is, $u^{(k)}\rightarrow 1$ in $L^{2}(C_{4})$. Moreover, a nonconstant $u$ has $d_{4}(u)>0$, and so $\Delta _{4}(u)/d_{4}(u)^{4}>\frac{1}{3}$. Therefore, the value $\frac{1}{3}$ is not attained.
\end{proof}

\subsection{Optimality of the constant \texorpdfstring{$\frac{1}{3}$}{1/3}}\label{SSoptimal13}

Let $g_{\varepsilon }(s)=\frac{1+\varepsilon s}{\sqrt{1+\varepsilon ^{2}}}$ be the function of Step 5 of the proof of Theorem \ref{Tmajorantintro}, and let $G_{\varepsilon }(s,t)=g_{\varepsilon }(s)g_{\varepsilon }(t)$ with $0<\varepsilon <\frac{1}{4}$. These functions are positive, and $\left\langle G_{\varepsilon }^{2}\right\rangle =\left\langle g_{\varepsilon }^{2}\right\rangle _{s}\left\langle g_{\varepsilon }^{2}\right\rangle _{t}=1$. In the basis \eqref{basis4}, we have $s+t=\sqrt{2}e$ and $st=\chi _{2}$ on $C_{4}$. Therefore,
\begin{equation*}
G_{\varepsilon }=\frac{1+\varepsilon \left( s+t\right) +\varepsilon ^{2}st}{1+\varepsilon ^{2}}=\frac{1+\sqrt{2}\varepsilon e+\varepsilon ^{2}\chi _{2}}{1+\varepsilon ^{2}}.
\end{equation*}
So $G_{\varepsilon }$ is a constant plus a first harmonic of order $\varepsilon $ plus a second harmonic of order $\varepsilon ^{2}$. In the notation of Proposition \ref{PSOS}, the coefficients of $e$, $f$ and $\chi _{2}$ are $\frac{\sqrt{2}\varepsilon }{1+\varepsilon ^{2}}$, $0$ and $\frac{\varepsilon ^{2}}{1+\varepsilon ^{2}}$. So the first square in \eqref{SOS} is equal to $\Big( \frac{\varepsilon ^{2}}{1+\varepsilon ^{2}}-\frac{\varepsilon ^{2}}{\left( 1+\varepsilon ^{2}\right) ^{2}}\Big) ^{2}=\frac{\varepsilon ^{8}}{\left( 1+\varepsilon ^{2}\right) ^{4}}$, which is only of order $\varepsilon ^{8}$. For the family $U_{\varepsilon }$ in Section 3, this square vanishes.

\emph{Energy of a product.} By \eqref{E4product}, the $s$-energy is $\left\langle \mathcal{E}_{2,s}\left( g_{\varepsilon }(\cdot )g_{\varepsilon }(t)\right) \right\rangle _{t}=\left\langle g_{\varepsilon }(t)^{2}\right\rangle _{t}\mathcal{E}_{2}(g_{\varepsilon })=\mathcal{E}_{2}(g_{\varepsilon })$. This holds because $\mathcal{E}_{2}$ is homogeneous of degree two and $\left\Vert g_{\varepsilon }\right\Vert _{2}=1$. The same holds for the $t$-energy. Therefore, $\mathcal{E}_{4}(G_{\varepsilon })=2\mathcal{E}_{2}(g_{\varepsilon })$.

\emph{Entropy of a product.} We have $G_{\varepsilon }^{2}\log G_{\varepsilon }^{2}=g_{\varepsilon }(s)^{2}g_{\varepsilon }(t)^{2}\left( \log g_{\varepsilon }(s)^{2}+\log g_{\varepsilon }(t)^{2}\right) $. Taking the average and using $\left\langle g_{\varepsilon }^{2}\right\rangle =1$, we get $\Ent(G_{\varepsilon }^{2})=2\left\langle g_{\varepsilon }^{2}\log g_{\varepsilon }^{2}\right\rangle =2\Ent(g_{\varepsilon }^{2})$. Therefore,
\begin{equation*}
\Delta _{4}(G_{\varepsilon })=2\left( 2\mathcal{E}_{2}(g_{\varepsilon })-
\Ent(g_{\varepsilon }^{2})\right) =2\delta _{2}(g_{\varepsilon }).
\end{equation*}

\emph{The two-point deficit.} In the representation \eqref{rhorep}, we have exactly $\rho =\varepsilon $. So, by Step 7 of the proof of Theorem \ref{Ttwopoint}, we get $\delta _{2}(g_{\varepsilon })=\frac{2}{3}\varepsilon ^{4}+O(\varepsilon ^{6})$.

\emph{The distance.} We have $\left\langle G_{\varepsilon }\right\rangle =\left\langle g_{\varepsilon }\right\rangle ^{2}=\frac{1}{1+\varepsilon ^{2}}$. So $d_{4}(G_{\varepsilon })^{2}=2\left( 1-\frac{1}{1+\varepsilon ^{2}}\right) =\frac{2\varepsilon ^{2}}{1+\varepsilon ^{2}}$ and $d_{4}(G_{\varepsilon })^{4}=\frac{4\varepsilon ^{4}}{\left( 1+\varepsilon ^{2}\right) ^{2}}=4\varepsilon ^{4}+O(\varepsilon ^{6})$.

Therefore,
\begin{equation*}
\frac{\Delta _{4}(G_{\varepsilon })}{d_{4}(G_{\varepsilon })^{4}}=\frac{\frac{4
}{3}\varepsilon ^{4}+O(\varepsilon ^{6})}{4\varepsilon ^{4}+O(\varepsilon
^{6})}\rightarrow \frac{1}{3}\qquad (\varepsilon \rightarrow 0).
\end{equation*}
Since $G_{\varepsilon }\in \mathcal{A}_{4}$, this limit and \eqref{log4cycle} give $\mathcal{C}_{4}=\frac{1}{3}$. Since $\mathcal{E}_{\mathrm{w},4}=\mathcal{E}_{4}$, we also have $\mathcal{C}_{4}^{\mathrm{w}}=\frac{1}{3}$. Finally, \eqref{log4} gives $\mathcal{C}_{n}\geq \frac{1}{3}$ for $n\in \mathcal{J}$. Therefore, all the equalities in \eqref{sharp13} follow.

\begin{remark}
Along $G_{\varepsilon }$, the five terms on the right-hand side of \eqref{exact4} divided by $d_{4}(G_{\varepsilon })^{6}$ tend to $\frac{1}{40}$, $\frac{1}{40}$, $0$, $\frac{1}{12}$ and $0$. Indeed, let $y=\sqrt{1+\varepsilon ^{2}}$. The slices of $G_{\varepsilon }$ are $g_{t}=g_{\varepsilon }(t)g_{\varepsilon }$. So $h=g_{\varepsilon }$, $m=g_{\varepsilon }/y$, $b=\frac{1}{y}$ and $a=\frac{1}{y^{2}}$. Since $\mathcal{R}_{2}$ is homogeneous of degree two, both $\left\langle \mathcal{R}_{2}(g_{t})\right\rangle _{t}$ and $\mathcal{R}_{2}(h)$ are equal to $\mathcal{R}_{2}(g_{\varepsilon })=\delta _{2}(g_{\varepsilon })-\frac{2}{3}D(\varepsilon )^{2}=\frac{1}{5}D(\varepsilon )^{3}+O(\varepsilon ^{8})$, by Step 7 of the proof of Theorem \ref{Ttwopoint}. Moreover, $\mathcal{R}_{\mathrm{slice}}=0$, since the two slices are multiples of $g_{\varepsilon }$. The fourth term is equal to $\frac{8}{3}\left( 1-\frac{1}{y}\right) ^{2}\left( 1-\frac{1}{y^{2}}\right) $, and the fifth term is equal to $\frac{4}{3}\left( 1-\frac{1}{y}\right) ^{4}$. Finally, $D(\varepsilon )=2\left( 1-\frac{1}{y}\right) =\varepsilon ^{2}+O(\varepsilon ^{4})$ and $d_{4}(G_{\varepsilon })^{6}=8\left( 1-\frac{1}{y^{2}}\right) ^{3}=8\varepsilon ^{6}+O(\varepsilon ^{8})$. So
\begin{equation*}
\frac{\Delta _{4}(G_{\varepsilon })-\frac{1}{3}d_{4}(G_{\varepsilon })^{4}}{d_{4}(G_{\varepsilon })^{6}}\rightarrow \frac{1}{40}+\frac{1}{40}+\frac{1}{12}=\frac{2}{15}\qquad (\varepsilon \rightarrow 0).
\end{equation*} The constant $\frac{1}{20}$ in \eqref{log4cycle} comes from the two terms with $\mathcal{R}_{2}$, and the missing $\frac{1}{12}$ comes from the fourth term, which we dropped in \eqref{sextic}. Therefore, the best constant in the sixth-order term of \eqref{log4cycle} lies in $\left[ \frac{1}{20},\frac{2}{15}\right] $. We do not know its exact value.
\end{remark}

\subsection{The discrete cube}\label{SScube}

The argument of this section extends to the discrete cube $\Omega _{d}=\{-1,1\}^{d}$, $d\geq 1$, with the uniform probability measure and the energy
\begin{equation*}
\mathcal{E}^{(d)}(g)=\sum_{i=1}^{d}\left\langle \frac{1}{4}\left( g-g\circ
\sigma _{i}\right) ^{2}\right\rangle ,
\end{equation*}
where $\sigma _{i}$ changes the sign of the $i$-th coordinate. For $d=1$ this is $\mathcal{E}_{2}$, and for $d=2$ it is $\mathcal{E}_{4}$ by \eqref{E4product}. The Walsh functions $\chi _{S}(x)=\prod_{i\in S}x_{i}$ are eigenfunctions with eigenvalues $\left\vert S\right\vert $. So the spectral gap is one. Gross's inequality on $\Omega _{d}$, which follows from the two-point case by tensorization \cite{Gross75}, reads $2\mathcal{E}^{(d)}(g)\geq \Ent(g^{2})$.

\begin{proposition}
\label{Rcube}Let $d\geq 1$. For every $g\geq 0$ on $\Omega _{d}$ with $\left\Vert g\right\Vert _{2}=1$, we have
\begin{equation}
\Delta ^{(d)}(g):=2\mathcal{E}^{(d)}(g)-\Ent(g^{2})\geq \frac{2}{3d}\left\Vert
g-1\right\Vert _{2}^{4}.  \label{cube}
\end{equation}
The constant $\frac{2}{3d}$ is sharp.
\end{proposition}

\begin{proof}
We prove \eqref{cube} by induction on $d$. The case $d=1$ is Theorem \ref{Ttwopoint}. Let $d\geq 2$ and write the points of $\Omega _{d}$ as $(s,t)$ with $s\in \{-1,1\}$ and $t\in \Omega _{d-1}$. As above, let $g_{t}(s)=g(s,t)$, $h(t)=\left\Vert g_{t}\right\Vert _{2,s}$, $m(t)=\left\langle g_{t}\right\rangle _{s}$, $a=\langle g\rangle $ and $b=\left\langle h\right\rangle _{t}$, so that $h\geq 0$, $\left\Vert h\right\Vert _{2,t}=1$ and $0<a\leq b\leq 1$. The chain rule \eqref{chainrule} holds without any change. That is, we have $\Ent(g^{2})=\left\langle \Ent_{s}(g_{t}^{2})\right\rangle _{t}+\Ent_{t}(h^{2})$. The direction $i=1$ contributes $\left\langle \mathcal{E}_{2,s}(g_{t})\right\rangle _{t}$ to the energy. Each direction $i\geq 2$ acts on $t$. For such a direction, the reverse triangle inequality $\left\vert \left\Vert g_{t}\right\Vert _{2,s}-\left\Vert g_{\sigma _{i}t}\right\Vert _{2,s}\right\vert \leq \left\Vert g_{t}-g_{\sigma _{i}t}\right\Vert _{2,s}$ gives
\begin{equation*}
\left\langle \frac{1}{4}\left( g-g\circ \sigma _{i}\right) ^{2}\right\rangle
=\left\langle \frac{1}{4}\left\Vert g_{t}-g_{\sigma _{i}t}\right\Vert
_{2,s}^{2}\right\rangle _{t}\geq \left\langle \frac{1}{4}\left( h(t)-h(\sigma
_{i}t)\right) ^{2}\right\rangle _{t}.
\end{equation*}
Summing over $i\geq 2$, we obtain $\mathcal{E}^{(d)}(g)\geq \left\langle \mathcal{E}_{2,s}(g_{t})\right\rangle _{t}+\mathcal{E}^{(d-1)}(h)$. We now use \eqref{twopointhom} on the slices (when $g_{t}=0$, both $\delta _{2}(g_{t})$ and $h(t)-m(t)$ vanish), the induction hypothesis applied to $h$, the inequality $\left\langle \left( h-m\right) ^{2}\right\rangle _{t}\geq \left\langle h-m\right\rangle _{t}^{2}=\left( b-a\right) ^{2}$ and the identity $\left\Vert h-1\right\Vert _{2,t}^{2}=2\left( 1-b\right) $. Then
\begin{equation*}
\Delta ^{(d)}(g)\geq \left\langle \delta _{2}(g_{t})\right\rangle _{t}+\Delta
^{(d-1)}(h)\geq \frac{8}{3}\left( b-a\right) ^{2}+\frac{2}{3\left(
d-1\right) }\cdot 4\left( 1-b\right) ^{2}=\frac{8}{3}\left( X_{0}^{2}+\frac{Y_{0}^{2}
}{d-1}\right) ,
\end{equation*}
with $X_{0}=b-a\geq 0$, $Y_{0}=1-b\geq 0$ and $X_{0}+Y_{0}=1-a$. The Cauchy--Schwarz inequality $\left( X_{0}+Y_{0}\right) ^{2}\leq d\left( X_{0}^{2}+\frac{Y_{0}^{2}}{d-1}\right) $ gives $\Delta ^{(d)}(g)\geq \frac{8}{3d}\left( 1-a\right) ^{2}=\frac{2}{3d}\left\Vert g-1\right\Vert _{2}^{4}$, because $\left\Vert g-1\right\Vert _{2}^{2}=2\left( 1-a\right) $. This proves \eqref{cube}. For the sharpness, we consider the product $G_{\varepsilon }(x)=\prod_{i=1}^{d}g_{\varepsilon }(x_{i})$, with $g_{\varepsilon }$ as in Section \ref{SSoptimal13}. Exactly as in the case $d=2$, the energy and the entropy are additive on products of normalized factors. Therefore, $\Delta ^{(d)}(G_{\varepsilon })=d\,\delta _{2}(g_{\varepsilon })=\frac{2}{3}d\varepsilon ^{4}+O(\varepsilon ^{6})$. In addition, we have $\left\Vert G_{\varepsilon }-1\right\Vert _{2}^{2}=2\left( 1-\left( 1+\varepsilon ^{2}\right) ^{-d/2}\right) =d\varepsilon ^{2}+O(\varepsilon ^{4})$. Therefore, the quotient $\Delta ^{(d)}(G_{\varepsilon })/\left\Vert G_{\varepsilon }-1\right\Vert _{2}^{4}$ tends to $\frac{2}{3d}$. For $d=2$, we recover the constant $\frac{1}{3}$ of \eqref{log4cycle}.
\end{proof}

Finally, we compare \eqref{cube} with the product estimate of Xie and Zhang. Let $d$ be even. Then $\Omega _{d}$ is the product of $\frac{d}{2}$ copies of $C_{4}$, and $\mathcal{E}^{(d)}$ is the sum of the energies $\mathcal{E}_{4}$ of the factors by \eqref{E4product}. Moreover, $\mathcal{E}_{4}$ is also the word-length energy on $C_{4}$. So \cite[(1.8)]{XZ26} gives \eqref{cube} with the constant $\frac{1}{12}\cdot \frac{2}{d}=\frac{1}{6d}$. The sharp constant $\frac{2}{3d}$ is four times larger.
\section{The word-length energy -- \texorpdfstring{\mbox{Proof of Theorem~\ref{Tmomentintro}}}{Proof of Theorem \ref{Tmomentintro}}}

In this section, $n$ is finite, except in Section \ref{SSupper}, where we also treat the circle. For a real function $f$ on $C_{n}$ with $\langle f\rangle =0$, we set
\begin{equation*}
s=\left\Vert f\right\Vert _{2}^{2}=\sum_{k\neq 0}\left\vert \widehat{f}
(k)\right\vert ^{2},\qquad T=\mathcal{E}_{\mathrm{w},n}(f)-s=\sum_{k\neq
0}\left( \psi _{n}(k)-1\right) \left\vert \widehat{f}(k)\right\vert ^{2}\geq
0.
\end{equation*}
So $T$ is a weighted sum. In this sum, the frequencies $\pm 1$ have weight zero, and the frequencies $\pm k$ have weight $\psi _{n}(k)-1=k-1$ for $2\leq k\leq n/2$. For the word-length energy, the quantity $T$ plays the same role as $Q_{n}(z)$ for the nearest-neighbor energy. That is, it is the excess of the energy over the Poincar\'{e} level.

\subsection{A third-moment estimate}

Xie and Zhang \cite[Propositions 10--11]{XZ26} proved that $\left\vert \left\langle f^{3}\right\rangle \right\vert \leq 3s\sqrt{T}$ for all real mean-zero $f$ on $C_{n}$ with $n\geq 4$. In Theorem \ref{Tmomentintro}, the squared coefficient $9$ is replaced by $\frac{383}{64}$, which is strictly less than six, the critical value for our method. This theorem is the key result for the word-length case. With the notation above, Theorem \ref{Tmomentintro} says that
\begin{equation}
\left\vert \left\langle f^{3}\right\rangle \right\vert ^{2}\leq \frac{383}{64}s^{2}T  \label{moment}
\end{equation}
for every finite $n\geq 6$ and every real function $f$ on $C_{n}$ with $\langle f\rangle =0$.

Before the proof, we explain the general scheme. We write $f=\sum_{k\neq 0}\widehat{f}(k)\chi _{k}$. Then $f^{3}=\sum_{k_{1},k_{2},k_{3}}\widehat{f}(k_{1})\widehat{f}(k_{2})\widehat{f}(k_{3})\chi _{k_{1}+k_{2}+k_{3}}$. By the orthogonality of characters, we have
\begin{equation*}
\left\langle f^{3}\right\rangle =\sum_{k_{1}+k_{2}+k_{3}\equiv 0\ (\mathrm{mod}\ n)}\widehat{f}(k_{1})\widehat{f}(k_{2})\widehat{f}(k_{3}),
\end{equation*}
and this is a sum over the \emph{resonant} ordered triples. We take absolute values, and we bound $\left\vert \left\langle f^{3}\right\rangle \right\vert $ by a sum of the form $\sum_{k\geq 2}a_{k}X_{k}$. Here $a_{k}$ is the amplitude of the frequency $k$, and $X_{k}$ is a quadratic form in the amplitudes. Then the Cauchy--Schwarz inequality with the weights $k-1$ gives
\begin{equation*}
\Big( \sum_{k\geq 2}a_{k}X_{k}\Big) ^{2}=\Big( \sum_{k\geq 2}\sqrt{k-1}
\,a_{k}\cdot \frac{X_{k}}{\sqrt{k-1}}\Big) ^{2}\leq \Big( \sum_{k\geq
2}\left( k-1\right) a_{k}^{2}\Big) \Big( \sum_{k\geq 2}\frac{X_{k}^{2}}{k-1}
\Big) =TW,
\end{equation*}
and so it remains to show that $W\leq \frac{383}{64}s^{2}$. Here $W$ is a quartic form in the amplitudes. For $n=10$, in one case we argue directly with the bound for $\left\vert \left\langle f^{3}\right\rangle \right\vert $. Our improvement over \cite{XZ26} comes from a more careful assignment of the resonant terms to the factors $a_{k}$ before this step. Note that the frequency $k=1$ has weight zero in $T$. Therefore, no term may be attributed to $a_{1}$.

\begin{proof}[Proof of Theorem \ref{Tmomentintro}]
If $T=0$, then only the frequencies $\pm 1$ are present, and $\left\langle f^{3}\right\rangle =0$. Indeed, the frequencies of $f^{3}$ are $\pm 1,\pm 3$, and none of them is $\equiv 0$ modulo $n\geq 6$. So we assume $T>0$. Then $s>0$. We first treat the odd cycles, then the even cycles with $n\geq 12$, and finally the three cycles $n=6,8,10$.

\textbf{Odd cycles $n=2m+1\geq 7$.} The nonzero frequencies are $\pm 1,\ldots ,\pm m$. We write
\begin{equation*}
f=\sum_{j=1}^{m}\left( c_{j}\chi _{j}+\overline{c_{j}}\chi _{-j}\right)
,\qquad a_{j}=\sqrt{2}\left\vert c_{j}\right\vert ,
\end{equation*}
so that, by Parseval, $s=\sum_{j=1}^{m}2\left\vert c_{j}\right\vert ^{2}=\sum_{j=1}^{m}a_{j}^{2}$ and $T=\sum_{j=2}^{m}\left( j-1\right) a_{j}^{2}$.

\emph{Step 1: the resonant triples.} Consider an ordered triple of signed frequencies from $\{\pm 1,\ldots ,\pm m\}$ with sum $\equiv 0$ modulo $n$. We have three cases, according to the signs:
\begin{itemize}
\item Two positive and one negative, $(i,j,-k)$. Then $i+j-k\in \left[ 2-m,2m-1\right] $, and this interval lies strictly between $-n$ and $n$. So $i+j-k\equiv 0$ forces $i+j=k$. The coefficient is $c_{i}c_{j}\overline{c_{k}}$, and there are three positions for the negative index.
\item One positive and two negative. This is the complex conjugate of the previous case.
\item All positive, $(i,j,k)$. Then $3\leq i+j+k\leq 3m<2n$, so $i+j+k\equiv 0$ forces $i+j+k=n$. The coefficient is $c_{i}c_{j}c_{k}$. The all-negative case is the complex conjugate of this case.
\end{itemize}
Therefore, we have the exact identity
\begin{equation*}
\left\langle f^{3}\right\rangle =6\Re \sum_{\substack{ 1\leq i,j,k\leq m  \\
i+j=k}}c_{i}c_{j}\overline{c_{k}}+2\Re \sum_{\substack{ 1\leq i,j,k\leq m  \\
i+j+k=n}}c_{i}c_{j}c_{k},
\end{equation*}
where the sums are over ordered indices.

\emph{Step 2: absolute values.} For ordered pairs of indices in $\{1,\ldots ,m\}$, we define
\begin{equation*}
L_{k}=\sum_{i+j=k}a_{i}a_{j},\qquad R_{k}=\sum_{i+j=n-k}a_{i}a_{j},\qquad
1\leq k\leq m.
\end{equation*}
Then $L_{1}=0$, since $i+j=1$ has no solution. Also, $R_{1}=a_{m}^{2}$, since the only solution of $i+j=2m$ is $(m,m)$. For $2\leq k\leq m$, $L_{k}$ has the $k-1$ terms given by the pairs $(1,k-1),\ldots ,(k-1,1)$. Moreover, $R_{k}$ has the $k$ terms given by the pairs $(i,j)$ with $i+j=2m+1-k$ and $1\leq i,j\leq m$, that is, $i\in \left[ m+1-k,m\right] $. Since $\left\vert c_{i}c_{j}c_{k}\right\vert =\frac{a_{i}a_{j}a_{k}}{2\sqrt{2}}$, we take absolute values in Step 1 and we get
\begin{equation}
\left\vert \left\langle f^{3}\right\rangle \right\vert \leq \frac{6}{2\sqrt{2}
}\sum_{k}a_{k}\sum_{i+j=k}a_{i}a_{j}+\frac{2}{2\sqrt{2}}\sum_{k}a_{k}
\sum_{i+j=n-k}a_{i}a_{j}=\frac{1}{\sqrt{2}}\sum_{k=1}^{m}a_{k}\left(
3L_{k}+R_{k}\right) .  \label{oddbound}
\end{equation}

\emph{Step 3: removing the $k=1$ term.} The term with $k=1$ is $a_{1}R_{1}=a_{1}a_{m}^{2}=a_{m}\left( a_{1}a_{m}\right) $. It is a multiple of $a_{m}$, so we attribute it to the frequency $m$. More precisely, we set $\widetilde{R}_{k}=R_{k}$ for $k<m$ and $\widetilde{R}_{m}=R_{m}+a_{1}a_{m}$. Then
\begin{equation*}
\left\vert \left\langle f^{3}\right\rangle \right\vert \leq \frac{1}{\sqrt{2}}
\sum_{k=2}^{m}a_{k}\left( 3L_{k}+\widetilde{R}_{k}\right) ,
\end{equation*}
and every summand has the form $a_{k}\times (\cdots )$ with $k\geq 2$. So the weight $k-1$ is positive.

\emph{Step 4: weighted Cauchy--Schwarz.} By the general scheme, we obtain
\begin{equation*}
\left\vert \left\langle f^{3}\right\rangle \right\vert ^{2}\leq TW,\qquad
W=\sum_{k=2}^{m}\frac{\left( 3L_{k}+\widetilde{R}_{k}\right) ^{2}}{2\left(
k-1\right) }.
\end{equation*}

\emph{Step 5: bounding each square.} Let $U_{k}$ and $V_{k}$ denote the sums $L_{k}$ and $R_{k}$ with $a_{i}a_{j}$ replaced by $a_{i}^{2}a_{j}^{2}$. For $k<m$, $3L_{k}+R_{k}$ is a sum of $k-1$ terms with weight $3$ and $k$ terms with weight $1$. The Cauchy--Schwarz inequality $\left( \sum_{\ell }\omega _{\ell }x_{\ell }\right) ^{2}\leq \left( \sum_{\ell }\omega _{\ell }^{2}\right) \left( \sum_{\ell }x_{\ell }^{2}\right) $ with $x_{\ell }=a_{i}a_{j}$ gives
\begin{equation*}
\frac{\left( 3L_{k}+R_{k}\right) ^{2}}{2\left( k-1\right) }\leq \frac{9\left(
k-1\right) +k}{2\left( k-1\right) }\left( U_{k}+V_{k}\right) =\left( 5+\frac{1
}{2\left( k-1\right) }\right) \left( U_{k}+V_{k}\right) \leq \frac{11}{2}
\left( U_{k}+V_{k}\right) .
\end{equation*}
For $k=m$, the two ordered terms $a_{1}a_{m}$ and $a_{m}a_{1}$ of $R_{m}$ receive an extra $a_{1}a_{m}$ in total. So each of them now has weight $\frac{3}{2}$. The other $m-2$ terms of $R_{m}$ keep weight $1$, and the $m-1$ terms of $L_{m}$ have weight $3$. Therefore,
\begin{align*}
\frac{\left( 3L_{m}+\widetilde{R}_{m}\right) ^{2}}{2\left( m-1\right) }& \leq
\frac{9\left( m-1\right) +\left( m-2\right) +2\cdot \frac{9}{4}}{2\left(
m-1\right) }\left( U_{m}+V_{m}\right) =\frac{10\left( m-1\right) +\frac{7}{2}
}{2\left( m-1\right) }\left( U_{m}+V_{m}\right) \\
& =\left( 5+\frac{7}{4\left( m-1\right) }\right) \left( U_{m}+V_{m}\right)
\leq \frac{47}{8}\left( U_{m}+V_{m}\right) ,
\end{align*}
since $m\geq 3$.

\emph{Step 6: summing.} Consider the index sums $i+j$ of the pairs in $U_{k}$, $2\leq k\leq m$. They range over $2,\ldots ,m$. The index sums of the pairs in $V_{k}$ range over $n-k\in \{m+1,\ldots ,2m-1\}$. Every ordered pair $(i,j)\in \{1,\ldots ,m\}^{2}$ has $i+j\in \{2,\ldots ,2m\}$. So each pair appears in exactly one of the $U_{k},V_{k}$, except $(m,m)$ with sum $2m$. Therefore,
\begin{equation*}
\sum_{k=2}^{m}\left( U_{k}+V_{k}\right) \leq \sum_{i,j=1}^{m}a_{i}^{2}a_{j}^{2}=
\Big( \sum_{i=1}^{m}a_{i}^{2}\Big) ^{2}=s^{2}.
\end{equation*}
As a consequence, $W\leq \frac{47}{8}s^{2}$, and $\frac{47}{8}=5.875<\frac{383}{64}=5.984375$. This proves \eqref{moment} for odd $n\geq 7$.

\textbf{Even cycles $n=2m\geq 12$.} Now the frequency $m\equiv -m$ is its own conjugate. Also, $\chi _{m}(j)=(-1)^{j}$ is real, and $c_{m}=\widehat{f}(m)$ is real. We write
\begin{equation*}
f=\sum_{j=1}^{m-1}\left( c_{j}\chi _{j}+\overline{c_{j}}\chi _{-j}\right)
+c_{m}\chi _{m},\qquad a_{j}=\sqrt{2}\left\vert c_{j}\right\vert \ (j<m),\qquad
a_{m}=\left\vert c_{m}\right\vert ,
\end{equation*}
so that again $s=\sum_{j=1}^{m}a_{j}^{2}$. Here the middle frequency contributes $c_{m}^{2}$ without the factor $2$. Moreover, $T=\sum_{j=2}^{m}\left( j-1\right) a_{j}^{2}$, since the weight of $m$ is $\psi _{n}(m)-1=m-1$.

\emph{Step 1: the resonant triples.} We treat the triples without the frequency $m$ as in the odd case. For the all-positive ones, we have $i+j+k=2m$, since $3\leq i+j+k\leq 3\left( m-1\right) <4m$. Consider a triple containing $m$ exactly once, $(m,\pm i,\pm j)$. It resonates if and only if $m\pm i\pm j\equiv 0$ modulo $2m$. For $(m,i,j)$, this means $i+j=m$, since $2\leq i+j\leq 2m-2$. For $(m,-i,-j)$, it also means $i+j=m$. For $(m,i,-j)$, it would mean $i-j\equiv m$, which is impossible since $\left\vert i-j\right\vert \leq m-2$. Two occurrences of $m$ give $2m+k\equiv k\not\equiv 0$, and three occurrences give $3m\equiv m\not\equiv 0$. There are three positions for $m$. Therefore,
\begin{equation*}
\left\langle f^{3}\right\rangle =6\Re \sum_{\substack{ 1\leq i,j,k<m  \\
i+j=k}}c_{i}c_{j}\overline{c_{k}}+2\Re \sum_{\substack{ 1\leq i,j,k<m  \\
i+j+k=2m}}c_{i}c_{j}c_{k}+3c_{m}\sum_{\substack{ 1\leq i,j<m  \\ i+j=m}}\left(
c_{i}c_{j}+\overline{c_{i}}\,\overline{c_{j}}\right) .
\end{equation*}

\emph{Step 2: absolute values.} Let $L_{k}=\sum_{i+j=k}a_{i}a_{j}$ and $R_{k}=\sum_{i+j=2m-k}a_{i}a_{j}$, with indices in $\{1,\ldots ,m-1\}$, and let $L_{m}=\sum_{i+j=m}a_{i}a_{j}$. Since $\left\vert 3c_{m}\cdot 2\Re \left( c_{i}c_{j}\right) \right\vert \leq 6\left\vert c_{m}\right\vert \frac{a_{i}a_{j}}{2}=3a_{m}a_{i}a_{j}$, we have
\begin{equation*}
\left\vert \left\langle f^{3}\right\rangle \right\vert \leq \frac{1}{\sqrt{2}}
\sum_{k=2}^{m-1}a_{k}\left( 3L_{k}+R_{k}\right) +3a_{m}L_{m}.
\end{equation*}
Here $R_{1}=0$, because $i+j=2m-1$ has no solution with $i,j\leq m-1$. So no re-attribution is needed. For $2\leq k<m$, both $L_{k}$ and $R_{k}$ have $k-1$ terms. Indeed, for $R_{k}$, the conditions $i+j=2m-k$ and $i,j\leq m-1$ mean that $i\in \left[ m+1-k,m-1\right] $. Moreover, $L_{m}$ has $m-1$ terms.

\emph{Step 3: weighted Cauchy--Schwarz.} The first vector is $\left( \sqrt{k-1}\,a_{k}\right) _{k=2}^{m}$, and its squared norm is $T$. The second vector has the entries $\frac{3L_{k}+R_{k}}{\sqrt{2\left( k-1\right) }}$ for $k<m$ and the entry $\frac{3L_{m}}{\sqrt{m-1}}$ for $k=m$. So we get $\left\vert \left\langle f^{3}\right\rangle \right\vert ^{2}\leq TW$ with
\begin{align*}
W& =\sum_{k=2}^{m-1}\frac{\left( 3L_{k}+R_{k}\right) ^{2}}{2\left(
k-1\right) }+\frac{9L_{m}^{2}}{m-1}\leq \sum_{k=2}^{m-1}\frac{9\left(
k-1\right) +\left( k-1\right) }{2\left( k-1\right) }\left( U_{k}+V_{k}\right) +
\frac{9L_{m}^{2}}{m-1} \\
& =5\sum_{k=2}^{m-1}\left( U_{k}+V_{k}\right) +\frac{9L_{m}^{2}}{m-1}.
\end{align*}

\emph{Step 4: the missing pairs.} Let $S_{0}=\sum_{j=1}^{m-1}a_{j}^{2}$ and $U_{m}=\sum_{i+j=m}a_{i}^{2}a_{j}^{2}$. The index sums of $U_{k}$ range over $2,\ldots ,m-1$. The index sums of $V_{k}$ range over $2m-k\in \{m+1,\ldots ,2m-2\}$. Every ordered pair in $\{1,\ldots ,m-1\}^{2}$ has $i+j\in \{2,\ldots ,2m-2\}$. So each pair occurs exactly once, except the pairs with $i+j=m$. Therefore,
\begin{equation*}
\sum_{k=2}^{m-1}\left( U_{k}+V_{k}\right) =S_{0}^{2}-U_{m}.
\end{equation*}
We do not throw away this missing sum $U_{m}$. Since $L_{m}$ has $m-1$ terms, the Cauchy--Schwarz inequality gives $L_{m}^{2}\leq \left( m-1\right) U_{m}$, that is, $-U_{m}\leq -\frac{L_{m}^{2}}{m-1}$. Also, $L_{m}=\sum_{j=1}^{m-1}a_{j}a_{m-j}\leq \left( \sum_{j}a_{j}^{2}\right) ^{1/2}\left( \sum_{j}a_{m-j}^{2}\right) ^{1/2}=S_{0}$. Consequently,
\begin{align}
W& \leq 5S_{0}^{2}-5U_{m}+\frac{9L_{m}^{2}}{m-1}\leq 5S_{0}^{2}-\frac{5L_{m}^{2}}{m-1}+\frac{9L_{m}^{2}}{m-1}  \notag \\
& =5S_{0}^{2}+\frac{4L_{m}^{2}}{m-1}\leq \left( 5+\frac{4}{m-1}\right) S_{0}^{2}\leq \frac{29}{5}s^{2},  \label{evenW}
\end{align}
for $m\geq 6$, and $\frac{29}{5}=5.8<\frac{383}{64}$. This proves \eqref{moment} for even $n\geq 12$. For $m\leq 5$, the factor $5+\frac{4}{m-1}$ is at least six, and this bound is not enough. This is why the cycles with $n=6,8,10$ need a separate treatment.

\textbf{The six-cycle.} The frequencies are $\pm 1,\pm 2$ and the middle frequency $3$. Let $a_{1},a_{2}\geq 0$ be the amplitudes of the frequencies $\pm 1,\pm 2$, and let $b=\left\vert \widehat{f}(3)\right\vert $ be the amplitude of the middle frequency. So $b=a_{m}$ in the notation of the even case, and we use $b$ for the middle amplitude also for $n=8$ and $n=10$. Then $s=a_{1}^{2}+a_{2}^{2}+b^{2}$ and $T=(2-1)a_{2}^{2}+(3-1)b^{2}=a_{2}^{2}+2b^{2}$. The multisets of frequencies with sum $\equiv 0$ modulo $6$ are the following. First, $\{1,1,-2\}$ and its conjugate, with three orderings each. Second, $\{2,2,2\}$ and its conjugate, with one ordering each. Third, $\{1,2,3\}$ and its conjugate $\{-1,-2,3\}$, with six orderings each. One can check that no other combination works. For example, $1+1+1=3$ and $3+3+k\equiv k$. We take absolute values with $\left\vert c_{j}\right\vert =\frac{a_{j}}{\sqrt{2}}$, and we get
\begin{equation*}
\left\vert \left\langle f^{3}\right\rangle \right\vert \leq 6\cdot \frac{
a_{1}^{2}}{2}\cdot \frac{a_{2}}{\sqrt{2}}+2\cdot \frac{a_{2}^{3}}{2\sqrt{2}}
+12\cdot b\cdot \frac{a_{1}a_{2}}{2}=\frac{a_{2}\left(
3a_{1}^{2}+a_{2}^{2}\right) }{\sqrt{2}}+6a_{1}a_{2}b=:P.
\end{equation*}
The term $6a_{1}a_{2}b$ contains the factor $a_{2}$, with weight $1$ in $T$. It also contains the factor $b$, with weight $2$ in $T$. We split it as $a_{1}a_{2}b+5a_{1}a_{2}b$, and we attribute the first part to $a_{2}$ and the second part to $b$. This gives
\begin{equation*}
P=a_{2}\left( \frac{3a_{1}^{2}+a_{2}^{2}}{\sqrt{2}}+a_{1}b\right) +\sqrt{2}
b\cdot \frac{5a_{1}a_{2}}{\sqrt{2}}.
\end{equation*}
Since $a_{2}^{2}+\left( \sqrt{2}b\right) ^{2}=T$, the Cauchy--Schwarz inequality gives $P^{2}\leq TW$ with
\begin{equation*}
W=\left( \frac{3a_{1}^{2}+a_{2}^{2}}{\sqrt{2}}+a_{1}b\right) ^{2}+\frac{25}{2}
a_{1}^{2}a_{2}^{2}.
\end{equation*}
We set $a_{1}=\sqrt{2}x$, $a_{2}=\sqrt{2}y$, $b=z$, where $x,y,z\geq 0$ are real numbers. Then $s=2x^{2}+2y^{2}+z^{2}$, $\frac{3a_{1}^{2}+a_{2}^{2}}{\sqrt{2}}+a_{1}b=\frac{6x^{2}+2y^{2}}{\sqrt{2}}+\sqrt{2}xz=\sqrt{2}\left( 3x^{2}+y^{2}+xz\right) $ and $\frac{25}{2}a_{1}^{2}a_{2}^{2}=50x^{2}y^{2}$. So we have
\begin{align}
6s^{2}-W& =6\left( 2x^{2}+2y^{2}+z^{2}\right) ^{2}-2\left(
3x^{2}+y^{2}+xz\right) ^{2}-50x^{2}y^{2}  \notag \\
& =6\left( x^{2}-xz-\frac{7}{6}y^{2}\right) ^{2}+16\left( xz-\frac{9}{16}
y^{2}\right) ^{2}+\frac{421}{48}y^{4}+24y^{2}z^{2}+6z^{4},  \label{SOS6}
\end{align}
and this is an exact polynomial identity. It can be checked by expanding both sides. Both sides are equal to
\begin{equation*}
6x^{4}-12x^{3}z-14x^{2}y^{2}+22x^{2}z^{2}-4xy^{2}z+22y^{4}+24y^{2}z^{2}+6z^{4}.
\end{equation*} In particular, $W\leq 6s^{2}$. To get a coefficient strictly below six, we put $A=x^{2}-xz-\frac{7}{6}y^{2}$. Since $\frac{421}{48}>8$, \eqref{SOS6} gives $6s^{2}-W\geq 6A^{2}+8y^{4}+6z^{4}$. From $x^{2}=A+xz+\frac{7}{6}y^{2}$ and $xz\leq \frac{x^{2}+z^{2}}{2}$, we get $x^{2}\leq \left\vert A\right\vert +\frac{x^{2}+z^{2}}{2}+\frac{7}{6}y^{2}$, that is, $x^{2}\leq 2\left\vert A\right\vert +z^{2}+\frac{7}{3}y^{2}$. Therefore,
\begin{equation*}
s=2x^{2}+2y^{2}+z^{2}\leq 4\left\vert A\right\vert +\frac{20}{3}y^{2}+3z^{2}.
\end{equation*}
The weighted Cauchy--Schwarz inequality $\left( \sum \alpha _{i}\beta _{i}\right) ^{2}\leq \left( \sum \frac{\alpha _{i}^{2}}{\lambda _{i}}\right) \left( \sum \lambda _{i}\beta _{i}^{2}\right) $ with $\left( \alpha _{i}\right) =\left( 4,\frac{20}{3},3\right) $, $\left( \beta _{i}\right) =\left( \left\vert A\right\vert ,y^{2},z^{2}\right) $ and $\left( \lambda _{i}\right) =\left( 6,8,6\right) $ gives
\begin{equation*}
s^{2}\leq \left( \frac{16}{6}+\frac{\left( 20/3\right) ^{2}}{8}+\frac{9}{6}
\right) \left( 6A^{2}+8y^{4}+6z^{4}\right) =\frac{175}{18}\left(
6A^{2}+8y^{4}+6z^{4}\right) ,
\end{equation*}
since $\frac{16}{6}+\frac{50}{9}+\frac{3}{2}=\frac{48+100+27}{18}=\frac{175}{18}$.
Therefore, $6s^{2}-W\geq \frac{18}{175}s^{2}\geq \frac{1}{10}s^{2}$ and $W\leq \frac{59}{10}s^{2}<\frac{383}{64}s^{2}$.

\textbf{The eight-cycle.} The frequencies are $\pm 1,\pm 2,\pm 3$ and the middle frequency $4$. Let $a_{1},a_{2},a_{3}\geq 0$ be the amplitudes of the frequencies $\pm 1,\pm 2,\pm 3$, and let $b=\left\vert \widehat{f}(4)\right\vert $ be the amplitude of the middle frequency. Then $s=a_{1}^{2}+a_{2}^{2}+a_{3}^{2}+b^{2}$ and $T=a_{2}^{2}+2a_{3}^{2}+3b^{2}$. Up to conjugation, the multisets with sum $\equiv 0$ modulo $8$ are $\{1,1,-2\}$, $\{1,2,-3\}$, $\{2,3,3\}$ (note that $2+3+3=8$), $\{1,3,4\}$ and $\{2,2,4\}$. We take absolute values. A multiset with a repeated index has three orderings, and the other multisets have six orderings. The conjugate gives a factor $2$. Therefore,
\begin{equation*}
\left\vert \left\langle f^{3}\right\rangle \right\vert \leq \frac{
3a_{1}^{2}a_{2}+6a_{1}a_{2}a_{3}+3a_{2}a_{3}^{2}}{\sqrt{2}}
+6a_{1}a_{3}b+3a_{2}^{2}b=:P.
\end{equation*}
For example, $\{1,2,-3\}$ gives $6\cdot 2\cdot \frac{a_{1}a_{2}a_{3}}{2\sqrt{2}}=\frac{6a_{1}a_{2}a_{3}}{\sqrt{2}}$, and $\{1,3,4\}$ gives $6\cdot 2\cdot b\frac{a_{1}a_{3}}{2}=6a_{1}a_{3}b$. We attribute the terms to the factors $a_{2},a_{3},b$ as follows. We assign one copy of $a_{1}a_{3}b$ to $a_{3}$ and five copies to $b$:
\begin{gather*}
P=a_{2}A_{2}+a_{3}A_{3}+bB, \\
A_{2}=\frac{3a_{1}^{2}+a_{3}^{2}}{\sqrt{2}
},\qquad A_{3}=\frac{6a_{1}a_{2}+2a_{2}a_{3}}{\sqrt{2}}+a_{1}b,\qquad
B=5a_{1}a_{3}+3a_{2}^{2}.
\end{gather*}
To check this, note that the $a_{2}a_{3}^{2}$ terms give $\frac{1+2}{\sqrt{2}}=\frac{3}{\sqrt{2}}$ and the $a_{1}a_{3}b$ terms give $1+5=6$. We write $P=\left( a_{2},\sqrt{2}a_{3},\sqrt{3}b\right) \cdot \left( A_{2},\frac{A_{3}}{\sqrt{2}},\frac{B}{\sqrt{3}}\right) $. Here the first vector has squared norm $T$. Then we get $P^{2}\leq TW$ with $W=A_{2}^{2}+\frac{1}{2}A_{3}^{2}+\frac{1}{3}B^{2}$. We set $a_{1}=x$, $a_{2}=y$, $a_{3}=z$, $b=\sqrt{2}w$, where $x,y,z,w\geq 0$ are real numbers. Then $s=x^{2}+y^{2}+z^{2}+2w^{2}$, $A_{2}^{2}=\frac{\left( 3x^{2}+z^{2}\right) ^{2}}{2}$, $A_{3}=\sqrt{2}\left( 3xy+yz+xw\right) $ so that $\frac{A_{3}^{2}}{2}=\left( 3xy+yz+xw\right) ^{2}$, and $\frac{B^{2}}{3}=\frac{\left( 5xz+3y^{2}\right) ^{2}}{3}$. The exact identity
\begin{align}
6s^{2}-W& =6s^{2}-\frac{\left( 3x^{2}+z^{2}\right) ^{2}}{2}-\left(
3xy+yz+xw\right) ^{2}-\frac{\left( 5xz+3y^{2}\right) ^{2}}{3}  \notag \\
& =\frac{3}{2}\left( x^{2}-y^{2}\right) ^{2}+\frac{3}{2}\left(
y^{2}-z^{2}\right) ^{2}+4z^{4}+\frac{2}{3}x^{2}z^{2}+24w^{2}\left(
w^{2}+y^{2}+z^{2}\right)  \notag \\
& \quad +23\left( wx-\frac{3}{23}xy-\frac{1}{23}yz\right) ^{2}+\frac{129}{23}
\left( xy-\frac{187}{129}yz\right) ^{2}+\frac{280}{129}y^{2}z^{2}
\label{SOS8}
\end{align}
can again be verified by expansion. Both sides are equal to
\begin{multline*}
\frac{3}{2}x^{4}+3y^{4}+\frac{11}{2}z^{4}+24w^{4}+3x^{2}y^{2}+\frac{2}{3}x^{2}z^{2}+11y^{2}z^{2} \\
+23w^{2}x^{2}+24w^{2}y^{2}+24w^{2}z^{2}-6wx^{2}y-2wxyz-16xy^{2}z.
\end{multline*}
For instance, the last three terms expand to $23w^{2}x^{2}+6x^{2}y^{2}+14y^{2}z^{2}-6wx^{2}y-2wxyz-16xy^{2}z$. We put $Y_{1}=x^{2}-y^{2}$ and $Y_{2}=y^{2}-z^{2}$. Then $6s^{2}-W\geq \frac{3}{2}Y_{1}^{2}+\frac{3}{2}Y_{2}^{2}+4z^{4}+24w^{4}$, and $s=Y_{1}+2Y_{2}+3z^{2}+2w^{2}$. Indeed, $Y_{1}+2Y_{2}=x^{2}+y^{2}-2z^{2}$. The weighted Cauchy--Schwarz inequality with $\left( \alpha _{i}\right) =\left( 1,2,3,2\right) $, $\left( \beta _{i}\right) =\left( Y_{1},Y_{2},z^{2},w^{2}\right) $ and $\left( \lambda _{i}\right) =\left( \frac{3}{2},\frac{3}{2},4,24\right) $ gives
\begin{equation*}
s^{2}\leq \left( \frac{2}{3}+\frac{8}{3}+\frac{9}{4}+\frac{4}{24}\right)
\left( \frac{3}{2}Y_{1}^{2}+\frac{3}{2}Y_{2}^{2}+4z^{4}+24w^{4}\right) \leq
\frac{23}{4}\left( 6s^{2}-W\right) ,
\end{equation*}
since $\frac{2}{3}+\frac{8}{3}+\frac{9}{4}+\frac{1}{6}=\frac{8+32+27+2}{12}=\frac{23}{4}$.
Therefore, $W\leq \left( 6-\frac{4}{23}\right) s^{2}=\frac{134}{23}s^{2}\approx 5.826\,s^{2}<\frac{383}{64}s^{2}$.

\textbf{The ten-cycle.} The frequencies are $\pm 1,\ldots ,\pm 4$ and the middle frequency $5$. Let $a_{1},\ldots ,a_{4}\geq 0$ be the amplitudes of the frequencies $\pm 1,\ldots ,\pm 4$, and let $b=\left\vert \widehat{f}(5)\right\vert $ be the amplitude of the middle frequency. We set
\begin{equation*}
S_{0}=\sum_{j=1}^{4}a_{j}^{2},\qquad
T_{0}=a_{2}^{2}+2a_{3}^{2}+3a_{4}^{2},\qquad s=S_{0}+b^{2},\qquad
T=T_{0}+4b^{2},
\end{equation*}
\begin{equation*}
L=2\left( a_{1}a_{4}+a_{2}a_{3}\right) ,\qquad U=2\left(
a_{1}^{2}a_{4}^{2}+a_{2}^{2}a_{3}^{2}\right) .
\end{equation*}
If $S_{0}=0$, then the only frequency is the middle one, and $\left\langle f^{3}\right\rangle =\widehat{f}(5)^{3}\left\langle \chi _{5}^{3}\right\rangle =0$. So we assume $S_{0}>0$. The even-cycle decomposition with $m=5$ gives $\left\vert \left\langle f^{3}\right\rangle \right\vert \leq P=P_{0}+3bL$. Here $P_{0}=\frac{1}{\sqrt{2}}\sum_{k=2}^{4}a_{k}\left( 3L_{k}+R_{k}\right) $ and $L=L_{5}$. Explicitly, we have
\begin{equation*}
L_{2}=a_{1}^{2},\quad L_{3}=2a_{1}a_{2},\quad L_{4}=2a_{1}a_{3}+a_{2}^{2},\quad
R_{2}=a_{4}^{2},\quad R_{3}=2a_{3}a_{4},\quad R_{4}=2a_{2}a_{4}+a_{3}^{2}.
\end{equation*}
Here $R_{k}$ is given by $i+j=10-k$ with $i,j\leq 4$. For example, for $R_{4}$, the condition $i+j=6$ gives $(2,4),(4,2),(3,3)$. The pairs with $i+j=5$ are $(1,4),(4,1),(2,3),(3,2)$. They are exactly the missing ones, and the sum of the squares of the corresponding products $a_{i}a_{j}$ is $U$. The weighted Cauchy--Schwarz inequality over $k=2,3,4$, as in Step 3 of the even case, and the identity $\sum_{k=2}^{4}\left( U_{k}+V_{k}\right) =S_{0}^{2}-U$ from Step 4 with $m=5$ give $P_{0}^{2}\leq 5T_{0}\left( S_{0}^{2}-U\right) $. Next, we apply the Cauchy--Schwarz inequality to $P=\sqrt{T_{0}}\cdot \frac{P_{0}}{\sqrt{T_{0}}}+2b\cdot \frac{3L}{2}$. Here the vector $\left( \sqrt{T_{0}},2b\right) $ has squared norm $T$. If $T_{0}=0$, then $P_{0}=0$ and the estimate is trivial. Otherwise, we obtain
\begin{equation}
P^{2}\leq T\left( \frac{P_{0}^{2}}{T_{0}}+\frac{9L^{2}}{4}\right) \leq
TW,\qquad W=5\left( S_{0}^{2}-U\right) +\frac{9}{4}L^{2}\leq 5S_{0}^{2}+L^{2},
\label{W10}
\end{equation}
where the last inequality follows from $L^{2}=4\left( a_{1}a_{4}+a_{2}a_{3}\right) ^{2}\leq 8\left( a_{1}^{2}a_{4}^{2}+a_{2}^{2}a_{3}^{2}\right) =4U$. This follows from $\left( X+Y\right) ^{2}\leq 2X^{2}+2Y^{2}$. So $-5U\leq -\frac{5}{4}L^{2}$ and $-\frac{5}{4}L^{2}+\frac{9}{4}L^{2}=L^{2}$. We also have
\begin{equation}
L\leq S_{0},\qquad L^{2}\leq 4\left( a_{1}^{2}+a_{2}^{2}\right) \left(
a_{3}^{2}+a_{4}^{2}\right) \leq 2S_{0}T_{0}.  \label{L10}
\end{equation}
The first inequality follows from $2a_{1}a_{4}\leq a_{1}^{2}+a_{4}^{2}$ and $2a_{2}a_{3}\leq a_{2}^{2}+a_{3}^{2}$. The second one follows from the Cauchy--Schwarz inequality $\left( a_{1}a_{4}+a_{2}a_{3}\right) ^{2}\leq \left( a_{1}^{2}+a_{2}^{2}\right) \left( a_{4}^{2}+a_{3}^{2}\right) $, together with $a_{1}^{2}+a_{2}^{2}\leq S_{0}$ and $2\left( a_{3}^{2}+a_{4}^{2}\right) \leq 2a_{3}^{2}+3a_{4}^{2}\leq T_{0}$. We distinguish three cases.

\emph{Case 1: $T_{0}\leq \frac{2}{5}S_{0}$.} Then \eqref{W10}--\eqref{L10} give $W\leq 5S_{0}^{2}+2S_{0}T_{0}\leq 5S_{0}^{2}+\frac{4}{5}S_{0}^{2}=\frac{29}{5}S_{0}^{2}\leq \frac{29}{5}s^{2}$. Intuitively, if $T_{0}$ is small compared with $S_{0}$, then the frequencies $2,3,4$ are small, and the resonances $1+4=5$ and $2+3=5$ are weak.

\emph{Case 2: $T_{0}>\frac{2}{5}S_{0}$ and $b^{2}\geq \frac{S_{0}}{512}$.} Then $s=S_{0}+b^{2}\geq \frac{513}{512}S_{0}$, and \eqref{W10}--\eqref{L10} give
\begin{equation*}
W\leq 5S_{0}^{2}+L^{2}\leq 6S_{0}^{2}\leq 6\left( \frac{512}{513}\right)
^{2}s^{2}<\frac{383}{64}s^{2}.
\end{equation*}
The last inequality is equivalent to $383\cdot 513^{2}>384\cdot 512^{2}$. We write $513^{2}=512^{2}+2\cdot 512+1$. Then it becomes $383\left( 2\cdot 512+1\right) >512^{2}$, that is, $383\cdot 1025=392575>262144$, which is true. In this case, the middle frequency carries a non-negligible part of $s$. So the factor $\left( S_{0}/s\right) ^{2}<1$ makes up for the constant $6$.

\emph{Case 3: $T_{0}>\frac{2}{5}S_{0}$ and $b^{2}<\frac{S_{0}}{512}$.} Then $\frac{b^{2}}{T_{0}}<\frac{S_{0}/512}{2S_{0}/5}=\frac{5}{1024}$, so $\frac{b}{\sqrt{T_{0}}}<\frac{\sqrt{5}}{32}$. By $P_{0}\leq \sqrt{5T_{0}\left( S_{0}^{2}-U\right) }\leq \sqrt{5T_{0}}S_{0}$ and $L\leq S_{0}$, we have
\begin{equation*}
P\leq S_{0}\sqrt{T_{0}}\left( \sqrt{5}+\frac{3b}{\sqrt{T_{0}}}\right) \leq
S_{0}\sqrt{T_{0}}\left( \sqrt{5}+\frac{3\sqrt{5}}{32}\right) =\frac{35\sqrt{5}
}{32}S_{0}\sqrt{T_{0}}.
\end{equation*}
Therefore, $P^{2}\leq \frac{35^{2}\cdot 5}{1024}S_{0}^{2}T_{0}=\frac{6125}{1024}S_{0}^{2}T_{0}\leq \frac{6125}{1024}s^{2}T<\frac{383}{64}s^{2}T$, because $\frac{383}{64}=\frac{6128}{1024}$. The margin here is only $\frac{3}{1024}$. In our proof, this case determines the constant $\frac{383}{64}$.

The three cases cover all possibilities, and \eqref{moment} is proved for $n=10$. This completes the proof of Theorem \ref{Tmomentintro}.
\end{proof}

\begin{example}
\label{Exfive}The estimate \eqref{moment} is not true for the five-cycle. Even the weaker estimate with the constant $6$ in place of $\frac{383}{64}$ fails on $C_{5}$. Indeed, let $f(j)=\sqrt{2}\cos \left( 2\pi j/5\right) +\frac{1}{\sqrt{2}}\cos \left( 4\pi j/5\right) $. In the notation of \eqref{fiveexact}, we have $f=v+z$ with $v=\frac{1}{\sqrt{2}}\left( \chi +\chi ^{-1}\right) $ and $z=\frac{1}{2\sqrt{2}}\left( \chi ^{2}+\chi ^{-2}\right) $. So $c_{1}=\frac{1}{\sqrt{2}}$ and $c_{2}=\frac{1}{2\sqrt{2}}$, and both are real. Therefore,
\begin{equation*}
\left\Vert f\right\Vert _{2}^{2}=2c_{1}^{2}+2c_{2}^{2}=1+\frac{1}{4}=\frac{5}{4}
,\qquad \mathcal{E}_{\mathrm{w},5}(f)=1\cdot 1+2\cdot \frac{1}{4}=\frac{3}{2}
,\qquad T=\frac{3}{2}-\frac{5}{4}=\frac{1}{4},
\end{equation*}
and, by \eqref{fiveexact}, $\left\langle f^{3}\right\rangle =6\left( c_{1}^{2}c_{2}+c_{1}c_{2}^{2}\right) =6\left( \frac{1}{2}\cdot \frac{1}{2\sqrt{2}}+\frac{1}{\sqrt{2}}\cdot \frac{1}{8}\right) =\frac{6}{2\sqrt{2}}\left( \frac{1}{2}+\frac{1}{4}\right) =\frac{9}{4\sqrt{2}}$. Therefore,
\begin{equation*}
\frac{\left\vert \left\langle f^{3}\right\rangle \right\vert ^{2}}{s^{2}T}=
\frac{81/32}{\left( 25/16\right) \cdot \left( 1/4\right) }=\frac{81}{32}\cdot
\frac{64}{25}=\frac{162}{25}=6.48>6.
\end{equation*}
The reason is the resonance $1+2+2=5$. It produces the product $a_{1}a_{2}^{2}$. Like $a_{1}^{2}a_{2}$, this product can only be assigned to the frequency $2$. But the weight of this frequency in $T$ is only $1$. For this reason, we treat the five-cycle separately in Appendix A, see Remark \ref{Rfive}.
\end{example}

\subsection{From the third-moment estimate to logarithmic stability}

The next theorem is the abstract counterpart of Theorem \ref{Tscalar}. Here, the hypothesis $T\geq 1-3a^{2}$ comes from the nonnegativity of $u$. This hypothesis makes the third-moment estimate effective on the whole admissible class.

\begin{theorem}
\label{Tmeanlog}Let $(X,\mu )$ be a probability space. Let $u\geq 0$ with $\left\Vert u\right\Vert _{2}=1$ and $u\in L^{3}(X,\mu )$, and let $\mathcal{E}$ be a nonnegative quadratic form vanishing on the constants, with $\mathcal{E}(u)<\infty $. Set $a=\langle u\rangle $, $s=1-a^{2}$, $T=\mathcal{E}(u)-s$ and $K=\sqrt{383/64}$. Assume that
\begin{equation}
T\geq \max \{0,1-3a^{2}\},\qquad \left\vert \left\langle \left( u-a\right)
^{3}\right\rangle \right\vert \leq Ks\sqrt{T}.  \label{hyp}
\end{equation}
Then
\begin{equation}
2\mathcal{E}(u)-\Ent(u^{2})\geq \frac{385}{1152}\left\Vert u-1\right\Vert
_{2}^{4}.  \label{meanlog}
\end{equation}
\end{theorem}

\begin{proof}
\emph{Step 0: finiteness and trivial case.} For $w\geq 0$, we have $\left\vert w^{2}\log w^{2}\right\vert \leq e^{-1}+2w^{3}$. Indeed, for $w\leq 1$ this holds because $x\log x\geq -e^{-1}$ on $\left[ 0,1\right] $. For $w\geq 1$ it holds because $\log w^{2}=2\log w\leq 2w$. So all entropies are finite when $u\in L^{3}$. If $a=1$, then $u=1$ and both sides of \eqref{meanlog} vanish. So let $0<a<1$, $d=\left\Vert u-1\right\Vert _{2}=\sqrt{2\left( 1-a\right) }>0$ and $\eta =1-a\in (0,1)$.

\emph{Step 1: the cubic deficit for a general form.} We repeat the computation that led to \eqref{Fexact2}, with $Q_{n}(z)$ replaced by $T=\mathcal{E}(u)-s$. Both quantities are ``energy minus Poincar\'{e} level''. Moreover, the computation of $\left\langle \left( u-1\right) ^{2}\left( u+2\right) \right\rangle $ only used $\langle u\rangle =a$ and $\left\langle u^{2}\right\rangle =1$. Therefore, we have
\begin{equation*}
\mathcal{F}_{\mathcal{E}}(u):=\mathcal{E}(u)-\frac{1}{3}\left\langle \left(
u-1\right) ^{2}\left( u+2\right) \right\rangle =T+\frac{s^{2}}{4}+\mathcal{R}
(a)-\frac{1}{3}\left\langle \left( u-a\right) ^{3}\right\rangle .
\end{equation*}

\emph{Step 2: completing the square in $\sqrt{T}$.} By \eqref{hyp}, we have $-\frac{1}{3}\left\langle \left( u-a\right) ^{3}\right\rangle \geq -\frac{1}{3}Ks\sqrt{T}$. Since $T-\frac{Ks}{3}\sqrt{T}=\left( \sqrt{T}-\frac{Ks}{6}\right) ^{2}-\frac{K^{2}s^{2}}{36}$, we obtain
\begin{equation*}
\mathcal{F}_{\mathcal{E}}(u)\geq \left( \sqrt{T}-\frac{Ks}{6}\right)
^{2}+\left( \frac{1}{4}-\frac{K^{2}}{36}\right) s^{2}+\mathcal{R}(a).
\end{equation*}
With $K^{2}=6-\frac{1}{64}$, we have $\frac{K^{2}}{36}=\frac{1}{6}-\frac{1}{2304}$, and so $\frac{1}{4}-\frac{K^{2}}{36}=\frac{1}{12}+\frac{1}{2304}$.

\emph{Step 3: adding the scalar remainder.} The splitting identity \eqref{splitting} holds for any energy, because it only used $\left\langle u^{2}\right\rangle =1$. So $2\mathcal{E}(u)-\Ent(u^{2})=2\mathcal{F}_{\mathcal{E}}(u)+\left\langle H(u)\right\rangle $. By Theorem \ref{Tmajorantintro}, $\left\langle H(u)\right\rangle \geq \frac{d^{4}}{6}+\frac{d^{6}}{45}$. Dividing by $d^{4}=4\eta ^{2}$, we get
\begin{equation*}
\frac{2\mathcal{E}(u)-\Ent(u^{2})}{d^{4}}\geq \frac{2}{d^{4}}\left( \sqrt{T}-
\frac{Ks}{6}\right) ^{2}+\frac{2}{d^{4}}\left[ \left( \frac{1}{12}+\frac{1}{
2304}\right) s^{2}+\mathcal{R}(a)\right] +\frac{1}{6}+\frac{d^{2}}{45}.
\end{equation*}

\emph{Step 4: the mean terms.} Recall that $s=\left( 1-a\right) \left( 1+a\right) $, $\mathcal{R}(a)=\frac{\left( 1-a\right) ^{3}\left( 1+3a\right) }{12}$ and $d^{4}=4\left( 1-a\right) ^{2}$. Then
\begin{align*}
& \frac{2}{d^{4}}\left[ \left( \frac{1}{12}+\frac{1}{2304}\right) s^{2}+
\mathcal{R}(a)\right] \\
& \qquad =\frac{1}{2\left( 1-a\right) ^{2}}\left[ \left( \frac{1}{12}+\frac{1
}{2304}\right) \left( 1-a\right) ^{2}\left( 1+a\right) ^{2}+\frac{\left(
1-a\right) ^{3}\left( 1+3a\right) }{12}\right] \\
& \qquad =\frac{\left( 1+a\right) ^{2}}{24}+\frac{\left( 1+a\right) ^{2}}{4608}+
\frac{\left( 1-a\right) \left( 1+3a\right) }{24} \\
& \qquad =\frac{1+2a-a^{2}}{12}+\frac{\left( 1+a\right) ^{2}}{4608},
\end{align*}
since $\left( 1+a\right) ^{2}+\left( 1-a\right) \left( 1+3a\right) =1+2a+a^{2}+1+2a-3a^{2}=2+4a-2a^{2}$. Next, $\frac{1+2a-a^{2}}{12}+\frac{1}{6}-\frac{1}{3}=\frac{1+2a-a^{2}-2}{12}=-\frac{1-2a+a^{2}}{12}=-\frac{\left( 1-a\right) ^{2}}{12}=-\frac{\eta ^{2}}{12}$, and $\frac{d^{2}}{45}=\frac{2\eta }{45}$. Therefore,
\begin{equation}
\frac{2\mathcal{E}(u)-\Ent(u^{2})}{d^{4}}-\frac{1}{3}\geq \frac{2}{d^{4}}
\left( \sqrt{T}-\frac{Ks}{6}\right) ^{2}+\vartheta (\eta )+\frac{\left( 1+a\right)
^{2}}{4608},\qquad \vartheta (\eta )=\frac{2\eta }{45}-\frac{\eta ^{2}}{12}.
\label{meangap}
\end{equation}
Note that when $a\rightarrow 1$, that is, $\eta \rightarrow 0$, the sum of the last two terms on the right-hand side tends to $\frac{4}{4608}=\frac{1}{1152}$. This is where the final constant comes from, and it is exactly the contribution of $K^{2}<6$.

\emph{Step 5: the right-hand side of \eqref{meangap} is at least $\frac{1}{1152}$.} The function $\vartheta $ is a concave parabola. So on any interval, it is at least the minimum of its values at the endpoints. We use $K\leq \sqrt{6}$, $s\leq 1$ and the hypothesis $T\geq 1-3a^{2}$. We distinguish three ranges of the mean.

If $a\leq \frac{1}{3}$, then $\eta \in \left[ \frac{2}{3},1\right) $ and $T\geq 1-3a^{2}\geq 1-\frac{1}{3}=\frac{2}{3}$. So $\sqrt{T}\geq \sqrt{\frac{2}{3}}=\frac{2}{\sqrt{6}}$. Also, $\frac{Ks}{6}\leq \frac{\sqrt{6}}{6}=\frac{1}{\sqrt{6}}$. So $\sqrt{T}-\frac{Ks}{6}\geq \frac{1}{\sqrt{6}}$, and the square is at least $\frac{1}{6}$. Since $d^{4}\leq 4$, the first term is at least $\frac{2}{4}\cdot \frac{1}{6}=\frac{1}{12}$. On $\left[ 0,1\right] $, we have $\vartheta (\eta )\geq \min \{\vartheta (0),\vartheta (1)\}=\min \{0,\frac{2}{45}-\frac{1}{12}\}=-\frac{7}{180}$. We drop the last term, and we see that the right-hand side is at least $\frac{1}{12}-\frac{7}{180}=\frac{15-7}{180}=\frac{2}{45}>\frac{1}{1152}$.

If $\frac{1}{3}\leq a\leq \frac{1}{2}$, then $\eta \in \left[ \frac{1}{2},\frac{2}{3}\right] $ and $T\geq 1-3\cdot \frac{1}{4}=\frac{1}{4}$. So $\sqrt{T}\geq \frac{1}{2}$. Moreover, $s\leq 1-\frac{1}{9}=\frac{8}{9}$, so $\frac{Ks}{6}\leq \frac{\sqrt{6}}{6}\cdot \frac{8}{9}=\frac{8}{9\sqrt{6}}\approx 0.363<\frac{3}{8}$. So the square is at least $\left( \frac{1}{2}-\frac{3}{8}\right) ^{2}=\frac{1}{64}$. Since $d^{4}=4\eta ^{2}\leq 4\cdot \frac{4}{9}=\frac{16}{9}$, the first term is at least $\frac{2\cdot 9}{16}\cdot \frac{1}{64}=\frac{9}{512}$. On $\left[ \frac{1}{2},\frac{2}{3}\right] $, we have $\vartheta (\eta )\geq \min \{\vartheta (\frac{1}{2}),\vartheta (\frac{2}{3})\}$ with $\vartheta (\frac{1}{2})=\frac{1}{45}-\frac{1}{48}=\frac{1}{720}$ and $\vartheta (\frac{2}{3})=\frac{4}{135}-\frac{1}{27}=-\frac{1}{135}$. So $\vartheta \geq -\frac{1}{135}$. The right-hand side is at least $\frac{9}{512}-\frac{1}{135}=\frac{9\cdot 135-512}{69120}=\frac{703}{69120}>\frac{1}{1152}=\frac{60}{69120}$.

If $\frac{1}{2}\leq a<1$, then $\eta \leq \frac{1}{2}$. In this range, $T$ may be very small. So we drop the square and consider
\begin{equation*}
\Phi(a)=\eta \left( \frac{2}{45}-\frac{\eta }{12}\right) +\frac{\left(
1+a\right) ^{2}}{4608},\qquad \eta =1-a.
\end{equation*}
The coefficient of $a^{2}$ in this function is $-\frac{1}{12}+\frac{1}{4608}=-\frac{383}{4608}<0$. So $\Phi$ is concave on $\left[ \frac{1}{2},1\right] $ and $\Phi(a)\geq \min \{\Phi(\frac{1}{2}),\Phi(1)\}$. We compute $\Phi(1)=\frac{4}{4608}=\frac{1}{1152}$ and
\begin{multline*}
\Phi \left( \tfrac{1}{2}\right) =\frac{1}{2}\left( \frac{2}{45}-\frac{1}{24}
\right) +\frac{9/4}{4608}=\frac{1}{2}\cdot \frac{16-15}{360}+\frac{9}{18432} \\
=\frac{1}{720}+\frac{1}{2048}=\frac{128+45}{92160}=\frac{173}{92160}>\frac{1}{
1152}.
\end{multline*}
Therefore, $\Phi(a)\geq \frac{1}{1152}$ on $\left[ \frac{1}{2},1\right] $.

In all cases, the right-hand side of \eqref{meangap} is at least $\frac{1}{1152}$, and \eqref{meanlog} follows from the identity $\frac{1}{3}+\frac{1}{1152}=\frac{385}{1152}$.
\end{proof}

\subsection{Proof of \eqref{logword}, \eqref{wordtonn} and \eqref{sepw}}\label{SSword}

\emph{Verification of the hypotheses on $C_{n}$, $n\geq 6$.} Let $u\in \mathcal{A}_{n}$ with the decomposition $u=a+v+z$ of \eqref{decomp}. We take $\mathcal{E}=\mathcal{E}_{\mathrm{w},n}$ in Theorem \ref{Tmeanlog}. It is a nonnegative quadratic form vanishing on the constants, since $\psi _{n}(0)=0$. Then $T=\mathcal{E}_{\mathrm{w},n}(u)-s=\sum_{k\neq 0}\left( \psi _{n}(k)-1\right) \left\vert \widehat{u}(k)\right\vert ^{2}$. The terms with $k=\pm 1$ vanish. For $2\leq k\leq n-2$, the weight $\psi _{n}(k)-1$ is at least $1$. Therefore,
\begin{equation*}
T\geq \sum_{2\leq k\leq n-2}\left\vert \widehat{u}(k)\right\vert
^{2}=\left\Vert z\right\Vert _{2}^{2}=1-a^{2}-\left\Vert v\right\Vert
_{2}^{2}\geq 1-a^{2}-2a^{2}=1-3a^{2},
\end{equation*}
by Lemma \ref{Lpositivity}. So $T\geq \max \{0,1-3a^{2}\}$. Next, $f=u-a$ is real with $\langle f\rangle =0$ and $\left\Vert f\right\Vert _{2}^{2}=s$. Moreover, $\mathcal{E}_{\mathrm{w},n}(f)=\mathcal{E}_{\mathrm{w},n}(u)$, since the constants are in the kernel. So $\mathcal{E}_{\mathrm{w},n}(f)-\left\Vert f\right\Vert _{2}^{2}=T$, and Theorem \ref{Tmomentintro} gives $\left\vert \left\langle \left( u-a\right) ^{3}\right\rangle \right\vert \leq Ks\sqrt{T}$ with $K^{2}=\frac{383}{64}$. Integrability is automatic on a finite cycle. So we can apply Theorem \ref{Tmeanlog}, and we get
\begin{equation}
\Delta _{\mathrm{w},n}(u)\geq \frac{385}{1152}d_{n}(u)^{4},\qquad n\geq 6.
\label{wordgap}
\end{equation}
This proves \eqref{sepw} for $n\geq 6$. In particular, it proves \eqref{logword} for $n\geq 6$, with strict inequality for $u\neq 1$. For $n=4$, the two energies coincide on $C_{4}$, and \eqref{logword} follows from \eqref{log4cycle}, with strict inequality for $u\neq 1$. For $n=5$, \eqref{logword} and \eqref{sepw} follow from Proposition \ref{Pfive} in Appendix A, since $\frac{503}{1500}>\frac{385}{1152}$.

\emph{Proof of \eqref{wordtonn}.} We first show that $\ell _{n}(k)\geq \psi _{n}(k)$ for $0\leq k<n$. This is \cite[(1.4)]{XZ26}. By the symmetries $\ell _{n}(n-k)=\ell _{n}(k)$ and $\psi _{n}(n-k)=\psi _{n}(k)$, we may assume $0\leq k\leq n/2$, so that $\psi _{n}(k)=k$. With $\theta =\frac{\pi }{n}$ and $1-\cos 2x=2\sin ^{2}x$, we have
\begin{equation*}
\ell _{n}(k)=\frac{\sin ^{2}(k\theta )}{\sin ^{2}\theta },\qquad 0\leq
k\theta \leq \frac{\pi }{2}.
\end{equation*}
First, let $k\geq 3$. The function $\sin $ is concave on $\left[ 0,\frac{\pi }{2}\right] $, with $\sin 0=0$ and $\sin \frac{\pi }{2}=1$. So it lies above the chord. That is, $\sin x\geq \frac{2}{\pi }x$ on $\left[ 0,\frac{\pi }{2}\right] $. Also, $\sin \theta \leq \theta $. Therefore, $\ell _{n}(k)\geq \frac{\left( 2k\theta /\pi \right) ^{2}}{\theta ^{2}}=\frac{4k^{2}}{\pi ^{2}}$, and this is $\geq k$ as soon as $k\geq \frac{\pi ^{2}}{4}\approx 2.47$. For $k=2$, we have $\ell _{n}(2)=4\cos ^{2}\frac{\pi }{n}\geq 4\cos ^{2}\frac{\pi }{4}=2=\psi _{n}(2)$ for $n\geq 4$. For $k=0,1$, both sides agree. So $\ell _{n}(k)\geq \psi _{n}(k)$ for all $k$. By Parseval, we have
\begin{equation*}
\mathcal{E}_{n}(u)-\mathcal{E}_{\mathrm{w},n}(u)=\sum_{k=0}^{n-1}\left( \ell
_{n}(k)-\psi _{n}(k)\right) \left\vert \widehat{u}(k)\right\vert ^{2}\geq 0.
\end{equation*}
Since both deficits have the same entropy, we have $\Delta _{n}(u)=\Delta _{\mathrm{w},n}(u)+2\left( \mathcal{E}_{n}(u)-\mathcal{E}_{\mathrm{w},n}(u)\right) $. Therefore, \eqref{wordtonn} follows from \eqref{logword}.

\subsection{Proof of the upper bounds \eqref{uppnn} and \eqref{uppw}}\label{SSupper}

We add a second harmonic to the first-harmonic family \eqref{testfamily}. Let $n\in \mathcal{J}$, and let $e_{n}$ be as in Section 3. For finite $n$, let $h_{n}(j)=\sqrt{2}\cos \left( 4\pi j/n\right) $. On the circle, let $h_{\infty }(x)=\sqrt{2}\cos \left( 4\pi x\right) $.

\emph{Step 1: properties of $e_{n}$ and $h_{n}$.} Since $2\cos ^{2}\phi =1+\cos 2\phi $, we have
\begin{equation}
e_{n}^{2}=1+\frac{h_{n}}{\sqrt{2}},\qquad h_{n}^{2}=1+\cos \left( 8\pi j/n\right) .
\label{ehsquare}
\end{equation}
The mean of $\cos \left( 8\pi j/n\right) $ is zero, since $4\not\equiv 0$ modulo $n$ for $n\geq 5$. On the circle, we replace $j/n$ by $x$, and the mean of $\cos \left( 8\pi x\right) $ is also zero. So $\left\langle h_{n}\right\rangle =0$ and $\left\langle h_{n}^{2}\right\rangle =1$. Moreover, $\left\langle e_{n}h_{n}\right\rangle =0$, since the frequencies $\pm 1\pm 2$ are not zero modulo $n$. By \eqref{ehsquare}, we get
\begin{equation*}
\left\langle e_{n}^{2}h_{n}\right\rangle =\left\langle h_{n}\right\rangle +\frac{\left\langle h_{n}^{2}\right\rangle }{\sqrt{2}}=\frac{1}{\sqrt{2}},\qquad \left\langle e_{n}^{4}\right\rangle =\left\langle \Big( 1+\frac{h_{n}}{\sqrt{2}}\Big) ^{2}\right\rangle =1+\frac{1}{2}=\frac{3}{2}.
\end{equation*}
We also have $\left\langle e_{n}^{3}\right\rangle =0$ by \eqref{enprops}, and $\left\vert e_{n}\right\vert ,\left\vert h_{n}\right\vert \leq \sqrt{2}$. The function $h_{n}$ has only the frequencies $\pm 2$. By \eqref{eigen} and the formula $\ell _{n}(2)=4\cos ^{2}\frac{\pi }{n}$ from the proof of Lemma \ref{Lspectral}, we get $\mathcal{E}_{n}(h_{n})=\ell _{n}(2)=1+q_{n}$. This also holds on the circle, since $\ell _{\infty }(2)=4$. For finite $n$, we have $\mathcal{E}_{\mathrm{w},n}(h_{n})=\psi _{n}(2)=2$.

Let $\mathcal{E}$ be either $\mathcal{E}_{n}$ or, for finite $n$, $\mathcal{E}_{\mathrm{w},n}$. Let $p=q_{n}$ in the first case and $p=1$ in the second case. Then $\mathcal{E}(e_{n})=1$ and $\mathcal{E}(h_{n})=1+p$. Moreover, $p\geq 1$ by \eqref{qn} and \eqref{qvalues}.

\emph{Step 2: the family.} For $0<\varepsilon <\frac{1}{8}$, we set
\begin{equation*}
\delta =\frac{\varepsilon ^{2}}{2\sqrt{2}p},\qquad \widetilde{u}_{\varepsilon }=a\left( 1+\varepsilon e_{n}+\delta h_{n}\right) ,\qquad a=\left( 1+\varepsilon ^{2}+\delta ^{2}\right) ^{-1/2}.
\end{equation*}
Since $1$, $e_{n}$ and $h_{n}$ are orthonormal, we have $\left\Vert \widetilde{u}_{\varepsilon }\right\Vert _{2}=1$, and the mean of $\widetilde{u}_{\varepsilon }$ is $a$. Since $p\geq 1$, we have $\delta \leq \frac{\varepsilon ^{2}}{2\sqrt{2}}$, and so $1+\varepsilon e_{n}+\delta h_{n}\geq 1-\sqrt{2}\varepsilon -\frac{\varepsilon ^{2}}{2}>0$. On the circle, $\widetilde{u}_{\varepsilon }$ is smooth. So $\widetilde{u}_{\varepsilon }\in \mathcal{A}_{n}\setminus \{1\}$. In the decomposition \eqref{decomp}, we have $v=a\varepsilon e_{n}$ and $z=a\delta h_{n}$. So $r=a\varepsilon $, $t=a\delta $ and $s=a^{2}\left( \varepsilon ^{2}+\delta ^{2}\right) =1-a^{2}$. Since $\left( 1+x\right) ^{-1/2}\geq 1-\frac{x}{2}$ for $x\geq 0$, we have $0<1-a\leq \frac{\varepsilon ^{2}+\delta ^{2}}{2}\leq \varepsilon ^{2}$. Moreover, $1-a=\frac{\varepsilon ^{2}}{2}+O(\varepsilon ^{4})$. Below, the constants in $O(\cdot )$ depend only on $p$.

\emph{Step 3: the cubic deficit.} Since $\mathcal{E}$ is diagonal in the Fourier basis, we have $\mathcal{E}(\widetilde{u}_{\varepsilon })-s=a^{2}\delta ^{2}\left( \mathcal{E}(h_{n})-1\right) =pa^{2}\delta ^{2}$. So the computation in Step 1 of the proof of Theorem \ref{Tmeanlog}, which is the computation that led to \eqref{Fexact2}, gives
\begin{equation*}
\mathcal{F}_{\mathcal{E}}(\widetilde{u}_{\varepsilon })=\mathcal{E}(\widetilde{u}_{\varepsilon })-\frac{1}{3}\left\langle \left( \widetilde{u}_{\varepsilon }-1\right) ^{2}\left( \widetilde{u}_{\varepsilon }+2\right) \right\rangle =pa^{2}\delta ^{2}+\frac{s^{2}}{4}+\mathcal{R}(a)-\frac{1}{3}\left\langle \left( v+z\right) ^{3}\right\rangle .
\end{equation*}
We expand the cube and use Step 1. Then
\begin{equation*}
\left\langle \left( v+z\right) ^{3}\right\rangle =a^{3}\left( \varepsilon ^{3}\left\langle e_{n}^{3}\right\rangle +3\varepsilon ^{2}\delta \left\langle e_{n}^{2}h_{n}\right\rangle +3\varepsilon \delta ^{2}\left\langle e_{n}h_{n}^{2}\right\rangle +\delta ^{3}\left\langle h_{n}^{3}\right\rangle \right) =\frac{3\varepsilon ^{2}\delta }{\sqrt{2}}+O(\varepsilon ^{5}).
\end{equation*}
Here we used $a^{3}=1+O(\varepsilon ^{2})$, $\delta =O(\varepsilon ^{2})$ and $\left\vert e_{n}\right\vert ,\left\vert h_{n}\right\vert \leq \sqrt{2}$. Moreover, $pa^{2}\delta ^{2}=p\delta ^{2}+O(\varepsilon ^{6})$, $\frac{s^{2}}{4}=\frac{\varepsilon ^{4}}{4}+O(\varepsilon ^{6})$ and $\mathcal{R}(a)=O\left( (1-a)^{3}\right) =O(\varepsilon ^{6})$. Since $p\delta ^{2}=\frac{\varepsilon ^{4}}{8p}$ and $\frac{\varepsilon ^{2}\delta }{\sqrt{2}}=\frac{\varepsilon ^{4}}{4p}$, we obtain
\begin{equation*}
\mathcal{F}_{\mathcal{E}}(\widetilde{u}_{\varepsilon })=\frac{\varepsilon ^{4}}{4}+p\delta ^{2}-\frac{\varepsilon ^{2}\delta }{\sqrt{2}}+O(\varepsilon ^{5})=\Big( \frac{1}{4}-\frac{1}{8p}\Big) \varepsilon ^{4}+O(\varepsilon ^{5}).
\end{equation*}
Note that our choice of $\delta $ minimizes $p\delta ^{2}-\frac{\varepsilon ^{2}\delta }{\sqrt{2}}$. So the second harmonic lowers the cubic deficit by $\frac{\varepsilon ^{4}}{8p}$.

\emph{Step 4: the scalar term.} We have $\widetilde{u}_{\varepsilon }-1-\varepsilon e_{n}=\left( a-1\right) \left( 1+\varepsilon e_{n}\right) +a\delta h_{n}$. Since $\left\vert a-1\right\vert \leq \varepsilon ^{2}$, $\left\vert 1+\varepsilon e_{n}\right\vert \leq 2$ and $a\delta \left\vert h_{n}\right\vert \leq \sqrt{2}\delta \leq \frac{\varepsilon ^{2}}{2}$, we get $\left\vert \widetilde{u}_{\varepsilon }-1-\varepsilon e_{n}\right\vert \leq 3\varepsilon ^{2}$. So $\left\vert \widetilde{u}_{\varepsilon }-1\right\vert \leq \sqrt{2}\varepsilon +3\varepsilon ^{2}<\frac{1}{2}$, and $\left( \widetilde{u}_{\varepsilon }-1\right) ^{4}=\varepsilon ^{4}e_{n}^{4}+O(\varepsilon ^{5})$ at every point. By Step 1 in Section \ref{SSpower}, we get
\begin{equation*}
\left\langle H(\widetilde{u}_{\varepsilon })\right\rangle =\frac{1}{6}\left\langle \left( \widetilde{u}_{\varepsilon }-1\right) ^{4}\right\rangle +O(\varepsilon ^{5})=\frac{\varepsilon ^{4}}{6}\left\langle e_{n}^{4}\right\rangle +O(\varepsilon ^{5})=\frac{\varepsilon ^{4}}{4}+O(\varepsilon ^{5}).
\end{equation*}

\emph{Step 5: conclusion.} The splitting \eqref{splitting} holds for $\mathcal{E}$, see Step 3 of the proof of Theorem \ref{Tmeanlog}. By Steps 3 and 4, we get
\begin{equation*}
2\mathcal{E}(\widetilde{u}_{\varepsilon })-\Ent(\widetilde{u}_{\varepsilon }^{2})=2\mathcal{F}_{\mathcal{E}}(\widetilde{u}_{\varepsilon })+\left\langle H(\widetilde{u}_{\varepsilon })\right\rangle =\Big( \frac{3}{4}-\frac{1}{4p}\Big) \varepsilon ^{4}+O(\varepsilon ^{5}).
\end{equation*}
On the other hand, $d_{n}(\widetilde{u}_{\varepsilon })^{2}=2\left( 1-a\right) =\varepsilon ^{2}+O(\varepsilon ^{4})$, and so $d_{n}(\widetilde{u}_{\varepsilon })^{4}=\varepsilon ^{4}+O(\varepsilon ^{6})$. Since $\widetilde{u}_{\varepsilon }\in \mathcal{A}_{n}\setminus \{1\}$, the definition \eqref{Cdef} gives
\begin{equation*}
\mathcal{C}\leq \frac{2\mathcal{E}(\widetilde{u}_{\varepsilon })-\Ent(\widetilde{u}_{\varepsilon }^{2})}{d_{n}(\widetilde{u}_{\varepsilon })^{4}}=\frac{3}{4}-\frac{1}{4p}+O(\varepsilon ),
\end{equation*}
where $\mathcal{C}=\mathcal{C}_{n}$ or $\mathcal{C}=\mathcal{C}_{n}^{\mathrm{w}}$. We let $\varepsilon \rightarrow 0$. For $p=q_{n}$, this gives \eqref{uppnn}. For $p=1$, it gives \eqref{uppw}.

For $\delta =0$, the functions $\widetilde{u}_{\varepsilon }$ are the functions \eqref{testfamily}, and the same computation gives only the value $\frac{3}{4}$. On $C_{4}$, the computation is different. Indeed, $4\equiv 0$ modulo $4$. So $h_{4}=\sqrt{2}\chi _{2}$, and we get $\left\langle h_{4}^{2}\right\rangle =2$, $\left\langle e_{4}^{2}h_{4}\right\rangle =\sqrt{2}$ and $\left\langle e_{4}^{4}\right\rangle =2$ instead of the values $1$, $\frac{1}{\sqrt{2}}$ and $\frac{3}{2}$ of Step 1. In this case, the second harmonic can cancel the whole fourth-order term of the cubic deficit, see Proposition \ref{PSOS}, and we get the value $\frac{1}{3}$ of Section \ref{SSoptimal13}.

\section{Proofs of Theorems \ref{Tlog} and \ref{Cseparation}}\label{Sproofs}

In this section, we put together the results of Sections 3--6 and of Appendix A.

\begin{proof}[Proof of Theorem \ref{Tlog}]
\emph{The estimate \eqref{log4}.} Let $n\in \mathcal{J}$ and $u\in \mathcal{A}_{n}$. By the splitting \eqref{splitting}, Proposition \ref{Pcubic} and Theorem \ref{Tmajorantintro}, we have $\Delta _{n}(u)\geq \frac{1}{3}d_{n}(u)^{4}+\frac{1}{45}d_{n}(u)^{6}$, see Section \ref{SSlog4}.

\emph{The estimate \eqref{log4cycle} and its strictness.} They are proved in Section \ref{SSlog4cycle}. The proof uses the exact decomposition of Proposition \ref{Pexact4} and the two-point estimate of Theorem \ref{Ttwopoint}.

\emph{The estimate \eqref{logword}.} For $n\geq 6$, it follows from \eqref{wordgap}, since $\frac{385}{1152}>\frac{1}{3}$. For $n=4$, it is \eqref{log4cycle}, since $\mathcal{E}_{\mathrm{w},4}=\mathcal{E}_{4}$. For $n=5$, it follows from Proposition \ref{Pfive}, since $\frac{503}{1500}>\frac{1}{3}$. See Section \ref{SSword} and Appendix A.

\emph{The sharpness of $\frac{1}{3}$.} By \eqref{log4} and \eqref{log4cycle}, we have $\mathcal{C}_{n}\geq \frac{1}{3}$ for every $n\in \mathcal{I}$. In Section \ref{SSoptimal13}, we show that $\Delta _{4}(G_{\varepsilon })/d_{4}(G_{\varepsilon })^{4}\rightarrow \frac{1}{3}$ as $\varepsilon \rightarrow 0$. So $\mathcal{C}_{4}=\frac{1}{3}$, and \eqref{sharp13} follows.

\emph{Minimizing sequences on $C_{4}$.} The bound \eqref{minimizingintro} is proved in Step 5 of the proof of \eqref{log4cycle} in Section \ref{SSlog4cycle}. It shows that every minimizing sequence for $\mathcal{C}_{4}$ converges to $1$ in $L^{2}(C_{4})$, and that the infimum in the definition of $\mathcal{C}_{4}$ is not attained.

\emph{The optimality of the power four.} This is \eqref{logpower}, which we prove in Section \ref{SSpower}.

\emph{The estimate \eqref{wordtonn}.} It is proved in Section \ref{SSword}. It follows from \eqref{logword} and the inequality $\ell _{n}(k)\geq \psi _{n}(k)$.

This completes the proof of Theorem \ref{Tlog}.
\end{proof}

\begin{proof}[Proof of Theorem \ref{Cseparation}]
\emph{The four-cycle.} We have $\mathcal{C}_{4}=\frac{1}{3}$ by \eqref{sharp13}. Since $\mathcal{E}_{\mathrm{w},4}=\mathcal{E}_{4}$, we also have $\mathcal{C}_{4}^{\mathrm{w}}=\frac{1}{3}$, see Section \ref{SSoptimal13}.

\emph{The lower bounds.} The lower bound \eqref{sepnn} is proved in Section \ref{SSsepnn}. For $n\geq 6$, the lower bound \eqref{sepw} is \eqref{wordgap}. For $n=5$, Proposition \ref{Pfive} gives $\mathcal{C}_{5}^{\mathrm{w}}\geq \frac{503}{1500}$. So \eqref{sepw} also holds for $n=5$, since $\frac{503}{1500}>\frac{385}{1152}$.

\emph{The upper bounds.} The upper bounds \eqref{uppnn} and \eqref{uppw} are proved in Section \ref{SSupper}.

This completes the proof of Theorem \ref{Cseparation}.
\end{proof}

\appendix
\section{The word-length energy on the five-cycle}

In this appendix, we prove the following result. It completes the proof of \eqref{logword} and of \eqref{sepw} in the case $n=5$.

\begin{proposition}
\label{Pfive}For every $u\in \mathcal{A}_{5}$, we have
\begin{equation}
\Delta _{\mathrm{w},5}(u)\geq \frac{503}{1500}d_{5}(u)^{4}=\left( \frac{1}{3}+
\frac{1}{500}\right) d_{5}(u)^{4}.  \label{fivegap}
\end{equation}
The inequality is strict for every $u\neq 1$.
\end{proposition}

As explained in Example \ref{Exfive}, the third-moment estimate of Theorem \ref{Tmomentintro} does not hold on $C_{5}$, even with the constant $6$. Therefore, the moment hypothesis in \eqref{hyp} is not available on $C_{5}$. Indeed, the quotient in Example \ref{Exfive} does not change when $f$ is multiplied by a positive constant. So, for small $\varepsilon >0$, the function $u=\left( 1+\varepsilon f\right) /\left\Vert 1+\varepsilon f\right\Vert _{2}\in \mathcal{A}_{5}$ does not satisfy this hypothesis, since $u-\langle u\rangle $ is a positive multiple of $f$. Instead, we change the center of the Taylor expansion of $w^{2}\log w^{2}$. That is, we expand at a point that depends on the mean of $u$. After this, we bound the deficit from below by an explicit function of the mean and of the two Fourier amplitudes. Then we eliminate the two amplitudes. In this way, we reduce \eqref{fivegap} to the positivity of an explicit polynomial of one variable on an interval.

\begin{proof}[Proof of Proposition \ref{Pfive}]
\emph{Step 1: a shifted entropy estimate.} For $\gamma >0$ and $w\geq 0$, we set
\begin{equation*}
T_{\gamma }(w)=\frac{2w^{3}}{3\gamma }+\left( \log \gamma ^{2}+1\right)
w^{2}-2\gamma w+\frac{\gamma ^{2}}{3}.
\end{equation*}
A direct computation from \eqref{Hdef} gives $T_{\gamma }(w)-w^{2}\log w^{2}=\gamma ^{2}H(w/\gamma )$. Indeed, we have $\left( w-\gamma \right) ^{2}\left( w+2\gamma \right) =w^{3}-3\gamma ^{2}w+2\gamma ^{3}$, and so
\begin{equation*}
\gamma ^{2}H\left( \frac{w}{\gamma }\right) =\frac{2w^{3}}{3\gamma }-2\gamma
w+\frac{4\gamma ^{2}}{3}+w^{2}-\gamma ^{2}+w^{2}\log \gamma ^{2}-w^{2}\log
w^{2}=T_{\gamma }(w)-w^{2}\log w^{2}.
\end{equation*}
Therefore, by \eqref{Hintegral},
\begin{equation*}
T_{\gamma }(w)-w^{2}\log w^{2}=\frac{2}{3}\left( w-\gamma \right)
^{4}\int_{0}^{1}\frac{\left( 1-\tau \right) ^{3}}{\left( \gamma +\tau \left(
w-\gamma \right) \right) ^{2}}d\tau ,\qquad w\geq 0.
\end{equation*}
The function $4\left( 1-\tau \right) ^{3}$ is a probability density on $\left[ 0,1\right] $. Its mean is $\frac{1}{5}$. Since $x\mapsto x^{-2}$ is convex on $(0,\infty )$, Jensen's inequality for the function $\tau \mapsto \gamma +\tau \left( w-\gamma \right) $, which is positive for $0\leq \tau <1$, gives
\begin{equation*}
\int_{0}^{1}\frac{\left( 1-\tau \right) ^{3}}{\left( \gamma +\tau \left(
w-\gamma \right) \right) ^{2}}d\tau \geq \frac{1}{4}\left( \gamma +\frac{
w-\gamma }{5}\right) ^{-2}=\frac{25}{4\left( w+4\gamma \right) ^{2}}.
\end{equation*}
Now let $u\geq 0$ be a function on a finite probability space with $\left\langle u^{2}\right\rangle =1$. Let $a=\langle u\rangle $ and $M_{3}=\left\langle u^{3}\right\rangle $. By the Cauchy--Schwarz inequality,
\begin{equation*}
\left\langle \frac{\left( u-\gamma \right) ^{4}}{\left( u+4\gamma \right)
^{2}}\right\rangle \geq \frac{\left\langle \left( u-\gamma \right)
^{2}\right\rangle ^{2}}{\left\langle \left( u+4\gamma \right)
^{2}\right\rangle }=\frac{\left( 1-2a\gamma +\gamma ^{2}\right) ^{2}}{
1+8a\gamma +16\gamma ^{2}}.
\end{equation*}
Since $\left\langle u^{2}\right\rangle =1$, we have $\Ent(u^{2})=\left\langle u^{2}\log u^{2}\right\rangle $. We average the identity for $T_{\gamma }(u)-u^{2}\log u^{2}$ and use the two lower bounds above. Then we obtain
\begin{equation}
\Ent(u^{2})\leq \frac{2M_{3}}{3\gamma }+\log \gamma ^{2}+1-2a\gamma +\frac{
\gamma ^{2}}{3}-\frac{25\left( 1-2a\gamma +\gamma ^{2}\right) ^{2}}{6\left(
1+8a\gamma +16\gamma ^{2}\right) }.  \label{shiftedent}
\end{equation}
For $\gamma =1$, this is a slightly better form of the estimate $\left\langle H(u)\right\rangle \geq 0$. The main point is that we can choose $\gamma $ depending on $a$.

\emph{Step 2: reduction to the mean and two amplitudes.} If $u=1$, both sides of \eqref{fivegap} vanish. So let $u\in \mathcal{A}_{5}$ with $u\neq 1$. We use the decomposition $u=a+v+z$ of \eqref{decomp}. We also use the notation $r=\left\Vert v\right\Vert _{2}$, $t=\left\Vert z\right\Vert _{2}$ and $s=r^{2}+t^{2}=1-a^{2}$. Then $a\geq \frac{1}{\sqrt{5}}$ by Lemma \ref{Lpositivity}. Moreover, $a<1$, since $a\leq \left\Vert u\right\Vert _{2}=1$ with equality only for $u=1$. So $s>0$. We choose
\begin{equation*}
\gamma =\frac{2+a}{3}.
\end{equation*}
On $C_{5}$, the only high frequencies are $\pm 2$. The word-length eigenvalue of these frequencies is $\psi _{5}(2)=2$. Therefore, $\mathcal{E}_{\mathrm{w},5}(u)=r^{2}+2t^{2}=s+t^{2}$. Moreover, as in Section 2, we have $\left\langle u^{3}\right\rangle =a^{3}+3as+\left\langle \left( v+z\right) ^{3}\right\rangle $. By \eqref{interaction5}, we get
\begin{equation*}
M_{3}\leq a^{3}+3as+\frac{3}{\sqrt{2}}\left( r^{2}t+rt^{2}\right) .
\end{equation*}
The coefficient of $M_{3}$ in \eqref{shiftedent} is positive. Therefore, we can insert this upper bound into \eqref{shiftedent}. Then we subtract the result from $2\mathcal{E}_{\mathrm{w},5}(u)=2\left( s+t^{2}\right) $, and we obtain
\begin{equation}
\Delta _{\mathrm{w},5}(u)\geq C(a,\gamma )+2t^{2}-\frac{\sqrt{2}}{\gamma }
\left( r^{2}t+rt^{2}\right) ,  \label{fivefirst}
\end{equation}
where
\begin{equation}
C(a,\gamma )=2s-\frac{2\left( a^{3}+3as\right) }{3\gamma }-\log \gamma
^{2}-1+2a\gamma -\frac{\gamma ^{2}}{3}+\frac{25\left( 1-2a\gamma +\gamma
^{2}\right) ^{2}}{6\left( 1+8a\gamma +16\gamma ^{2}\right) }.  \label{Canchor}
\end{equation}

\emph{Step 3: elimination of $r$ and $t$.} We put $c=\frac{99}{70\gamma }$. Since $99^{2}=9801>9800=2\cdot 70^{2}$, we have $c>\frac{\sqrt{2}}{\gamma }$. Since $r^{2}t+rt^{2}\geq 0$, we can replace $\frac{\sqrt{2}}{\gamma }$ by $c$ in \eqref{fivefirst}. By the concavity of the square root, for $0\leq t^{2}\leq s$ we have
\begin{equation*}
r=\sqrt{s-t^{2}}\leq \sqrt{s}-\frac{t^{2}}{2\sqrt{s}}.
\end{equation*}
Using also $r^{2}=s-t^{2}$, we get
\begin{equation*}
2t^{2}-c\left( r^{2}t+rt^{2}\right) \geq \kappa t^{2}-cst+c\,p(t),\qquad
\kappa =2-c\sqrt{s},\qquad p(t)=t^{3}+\frac{t^{4}}{2\sqrt{s}}.
\end{equation*}
The coefficient $\kappa $ is positive. Indeed, $a\geq \frac{1}{\sqrt{5}}>\frac{2}{5}$ implies $\gamma >\frac{4}{5}$ and $c<\frac{99}{56}$. Since $\sqrt{s}\leq 1$, we get $\kappa >2-\frac{99}{56}=\frac{13}{56}>0$. The function $p$ is convex on $\left[ 0,\infty \right) $, since $p^{\prime \prime }(t)=6t+\frac{6t^{2}}{\sqrt{s}}\geq 0$. Therefore, $p$ lies above its tangent line at $t_{0}=\frac{s}{2}$. That is,
\begin{equation*}
p(t)\geq \left( 3t_{0}^{2}+\frac{2t_{0}^{3}}{\sqrt{s}}\right) t-2t_{0}^{3}-
\frac{3t_{0}^{4}}{2\sqrt{s}},\qquad t\geq 0.
\end{equation*}
We set
\begin{equation*}
\Lambda =s-3t_{0}^{2}-\frac{2t_{0}^{3}}{\sqrt{s}},\qquad \Gamma
=2t_{0}^{3}+\frac{3t_{0}^{4}}{2\sqrt{s}}.
\end{equation*}
Since $c>0$, these bounds and \eqref{fivefirst}, with $\frac{\sqrt{2}}{\gamma }$ replaced by $c$, give
\begin{equation*}
\Delta _{\mathrm{w},5}(u)\geq C(a,\gamma )+\kappa t^{2}-c\Lambda t-c\Gamma .
\end{equation*}
The right-hand side is a quadratic polynomial in $t$. We complete the square and get
\begin{align*}
C(a,\gamma )+\kappa t^{2}-c\Lambda t-c\Gamma & =C(a,\gamma )+\kappa \left( t-
\frac{c\Lambda }{2\kappa }\right) ^{2}-c\Gamma -\frac{c^{2}\Lambda ^{2}}{
4\kappa } \\
& \geq C(a,\gamma )-c\Gamma -\frac{c^{2}\Lambda ^{2}}{4\kappa }.
\end{align*}
We only dropped a nonnegative square. The right-hand side depends only on $a$. So we have eliminated $r$ and $t$.

\emph{Step 4: removal of the logarithm.} We set $\zeta =1-\gamma =\frac{1-a}{3}$. Then $0<\zeta <1$, and
\begin{equation}
-\log \gamma ^{2}=-2\log \left( 1-\zeta \right) \geq 2\zeta +\zeta ^{2}+\frac{2
}{3}\zeta ^{3}+\frac{1}{2}\zeta ^{4}.  \label{fivelog}
\end{equation}
Indeed, the difference of the two sides is equal to $2\int_{0}^{\zeta }\frac{x^{4}}{1-x}dx\geq 0$. Let $\underline{C}(a,\gamma )$ denote the right-hand side of \eqref{Canchor} with $-\log \gamma ^{2}$ replaced by the polynomial on the right-hand side of \eqref{fivelog}. We define
\begin{equation}
\mathcal{G}(a)=\underline{C}(a,\gamma )-c\Gamma -\frac{c^{2}\Lambda ^{2}}{
4\kappa }-\frac{4}{3}\left( 1-a\right) ^{2}.  \label{Gfive}
\end{equation}
By Steps 2--4, we have proved
\begin{equation*}
\Delta _{\mathrm{w},5}(u)-\frac{4}{3}\left( 1-a\right) ^{2}\geq \mathcal{G}
(a).
\end{equation*}
Since $\frac{4}{3}\left( 1-a\right) ^{2}=\frac{1}{3}d_{5}(u)^{4}$, it remains to show that $\mathcal{G}(a)$ is larger than $\frac{1}{500}d_{5}(u)^{4}$.

\emph{Step 5: a rational parametrization.} We use the parameter
\begin{equation*}
\nu =\sqrt{\frac{1-a}{1+a}},\qquad a=\frac{1-\nu ^{2}}{1+\nu ^{2}},\qquad
\sqrt{s}=\frac{2\nu }{1+\nu ^{2}},\qquad \gamma =\frac{3+\nu ^{2}}{3\left(
1+\nu ^{2}\right) }.
\end{equation*}
Since $\frac{1}{\sqrt{5}}\leq a<1$, the allowed range is $0<\nu \leq \frac{\sqrt{5}-1}{2}<\frac{5}{8}$. In Steps 5 and 6, we write
\begin{equation*}
h=1+\nu ^{2},\qquad k=3+\nu ^{2},\qquad g=70\nu ^{2}-297\nu +210,\qquad
m=\nu ^{4}+66\nu ^{2}+225.
\end{equation*}
A direct computation gives
\begin{equation*}
t_{0}=\frac{2\nu ^{2}}{h^{2}},\qquad c=\frac{297h}{70k},\qquad \zeta =\frac{
2\nu ^{2}}{3h},\qquad \kappa =\frac{g}{35k}.
\end{equation*}
Note that $g>0$ on $\left[ 0,\frac{5}{8}\right] $. Indeed, its derivative $140\nu -297$ is negative there, and its value at $\nu =\frac{5}{8}$ is $\frac{1655}{32}>0$. Moreover, the two tangent coefficients are
\begin{align*}
\Lambda & =\frac{4\nu ^{2}}{h^{2}}-\frac{12\nu ^{4}}{h^{4}}-\frac{8\nu ^{5}}{
h^{5}}=\frac{4\nu ^{2}\left( 1-\nu ^{3}\right) ^{2}}{h^{5}}, \\
\Gamma & =\frac{16\nu ^{6}}{h^{6}}+\frac{12\nu ^{7}}{h^{7}}=\frac{4\nu
^{6}\left( 4h+3\nu \right) }{h^{7}}.
\end{align*}
For the rational term in \eqref{Canchor}, we have
\begin{equation*}
1-2a\gamma +\gamma ^{2}=\frac{4\nu ^{2}\left( 4\nu ^{2}+9\right) }{9h^{2}}
,\qquad 1+8a\gamma +16\gamma ^{2}=\frac{m}{9h^{2}}.
\end{equation*}
We insert these expressions into the definition of $\underline{C}$. The last term of $\underline{C}(a,\gamma )$, which is the rational term of \eqref{Canchor}, is equal to $\frac{200\nu ^{4}\left( 4\nu ^{2}+9\right) ^{2}}{27h^{2}m}$. The other terms of $\underline{C}(a,\gamma )$, together with $-\frac{4}{3}\left( 1-a\right) ^{2}$, add up to $\frac{8\nu ^{4}\left( 81+135\nu ^{2}+26\nu ^{4}-30\nu ^{6}\right) }{81h^{4}k}$. Over the common denominator $81h^{4}km$, we obtain
\begin{equation*}
\underline{C}(a,\gamma )-\frac{4}{3}\left( 1-a\right) ^{2}=\frac{16\nu
^{4}N(\nu )}{81h^{4}km},
\end{equation*}
where
\begin{equation*}
N(\nu )=585\nu ^{10}+4723\nu ^{8}+18288\nu ^{6}+43308\nu ^{4}+47223\nu
^{2}+18225.
\end{equation*}
As a result, \eqref{Gfive} becomes
\begin{equation*}
\frac{\mathcal{G}(a)}{\nu ^{4}}=\frac{16N(\nu )}{81h^{4}km}-\frac{594\nu
^{2}\left( 4h+3\nu \right) }{35kh^{6}}-\frac{88209\left( 1-\nu ^{3}\right)
^{4}}{35kgh^{8}}.
\end{equation*}
We multiply by the positive common denominator $2835h^{8}kgm$. This gives
\begin{equation}
\mathcal{G}\left( \frac{1-\nu ^{2}}{1+\nu ^{2}}\right) =\frac{\nu ^{4}P(\nu )
}{2835\left( 1+\nu ^{2}\right) ^{8}\left( \nu ^{2}+3\right) \left( 70\nu
^{2}-297\nu +210\right) \left( \nu ^{4}+66\nu ^{2}+225\right) },
\label{Grational}
\end{equation}
where
\begin{equation}
P(\nu )=560h^{4}gN(\nu )-48114\nu ^{2}\left( 4h+3\nu \right)
h^{2}gm-7144929\left( 1-\nu ^{3}\right) ^{4}m.  \label{Pfactored}
\end{equation}
Expanding \eqref{Pfactored}, we get $P(\nu )=\sum_{i=0}^{20}p_{i}\nu ^{i}$ with the integer coefficients $p_{i}$ listed in Table \ref{tabfive}.

\begin{table}[!t]
\centering
\small
\caption{The coefficients $p_{i}$ of the polynomial $P$ in \eqref{Pfactored}, the coefficients $\rho _{i}$ of the polynomial $R_{5}$ in \eqref{R5}, and the integers $L_{i}^{(0)}$, $L_{i}^{(1)}$ in \eqref{Bernsteinlower}.}
\label{tabfive}
\begin{tabular}{rrrrr}
\hline
$i$ & $p_i$ & $\rho_i$ & $L_i^{(0)}$ & $L_i^{(1)}$\\
\hline
0 & 535650975 & 261395707500 & 261395 & 2481\\
1 & -3031182000 & -1506497454000 & 237856 & 665\\
2 & 5275773486 & 2593135474200 & 215650 & 291\\
3 & -7507708560 & -3693594381840 & 194678 & 1435\\
4 & 21534350661 & 10631692721700 & 174861 & 4182\\
5 & -30479481252 & -15068216140560 & 156141 & 8620\\
6 & 44974079874 & 22253624634600 & 138473 & 14847\\
7 & -57182451876 & -28318284838800 & 121826 & 22966\\
8 & 69071641368 & 34284541706400 & 106182 & 33088\\
9 & -65515848948 & -32493524572080 & 91533 & 45328\\
10 & 56822280098 & 28236516749800 & 77882 & 59807\\
11 & -53042568876 & -26362450501200 & 65240 & 76648\\
12 & 28475383109 & 14159600685700 & 53626 & 95972\\
13 & -24396356952 & -12140680321440 & 43068 & 117897\\
14 & 10067916366 & 5012290618200 & 33599 & 142526\\
15 & -6767560800 & -3372302324160 & 25262 & 169942\\
16 & 2418511871 & 1205836245100 & 18104 & 200192\\
17 & -1174718160 & -586348686000 & 12182 & 233269\\
18 & 345665600 & 172585134400 & 7557 & 269086\\
19 & -97297200 & -48635128080 & 4298 & 307443\\
20 & 22932000 & 11462824800 & 2481 & 347980\\
\hline
\end{tabular}
\end{table}

\emph{Step 6: the quantitative bound.} Since $d_{5}(u)^{2}=2\left( 1-a\right) =\frac{4\nu ^{2}}{1+\nu ^{2}}$, we have $d_{5}(u)^{4}=\frac{16\nu ^{4}}{h^{2}}$. We divide the inequality at the end of Step 4 by $d_{5}(u)^{4}$ and use \eqref{Grational}. Then we get
\begin{equation}
\frac{\Delta _{\mathrm{w},5}(u)}{d_{5}(u)^{4}}-\frac{1}{3}\geq \frac{P(\nu )}{
D_{5}(\nu )},  \label{fivequot}
\end{equation}
where
\begin{equation*}
D_{5}(\nu )=16\cdot 2835\left( 1+\nu ^{2}\right) ^{6}\left( \nu
^{2}+3\right) \left( 70\nu ^{2}-297\nu +210\right) \left( \nu ^{4}+66\nu
^{2}+225\right) .
\end{equation*}
Every factor of $D_{5}$ is positive on $\left[ 0,\frac{5}{8}\right] $. So it is enough to prove that
\begin{equation}
R_{5}(\nu ):=500P(\nu )-D_{5}(\nu )>0,\qquad 0\leq \nu \leq \frac{5}{8}.
\label{R5}
\end{equation}
The polynomial $R_{5}$ has integer coefficients and degree twenty. We write $R_{5}(\nu )=\sum_{j=0}^{20}\rho _{j}\nu ^{j}$. Its coefficients are obtained from the numbers $p_{i}$ and the explicit product $D_{5}$, and they are listed in Table \ref{tabfive}.

\emph{Step 7: positivity of $R_{5}$.} We divide $\left[ 0,\frac{5}{8}\right] $ into the two intervals $\left[ 0,\frac{5}{16}\right] $ and $\left[ \frac{5}{16},\frac{5}{8}\right] $. For $\ell \in \{0,\frac{5}{16}\}$, we write $\nu =\ell +\frac{5}{16}x$ with $0\leq x\leq 1$, and we define
\begin{equation}
b_{i}(\ell )=\sum_{k=0}^{i}\frac{\binom{i}{k}}{\binom{20}{k}}\left( \frac{5}{16
}\right) ^{k}\sum_{j=k}^{20}\rho _{j}\binom{j}{k}\ell ^{j-k},\qquad 0\leq
i\leq 20,  \label{Bernsteincoef}
\end{equation}
with the convention $0^{0}=1$. We expand $R_{5}(\ell +\frac{5}{16}x)$ in powers of $x$, and we use the identity
\begin{equation*}
x^{k}=\sum_{i=k}^{20}\frac{\binom{i}{k}}{\binom{20}{k}}\binom{20}{i}
x^{i}\left( 1-x\right) ^{20-i},\qquad 0\leq k\leq 20.
\end{equation*}
This identity follows from the binomial expansion of $x^{k}\left( x+\left( 1-x\right) \right) ^{20-k}$. We obtain the exact identity
\begin{equation}
R_{5}\left( \ell +\tfrac{5}{16}x\right) =\sum_{i=0}^{20}b_{i}(\ell )\binom{20
}{i}x^{i}\left( 1-x\right) ^{20-i},  \label{Bernsteinid}
\end{equation}
that is, the numbers $b_{i}(\ell )$ are the coefficients of $R_{5}$ in the Bernstein basis of degree twenty on the corresponding interval. The numbers $b_{i}(\ell )$ are rational. We can compute them exactly from \eqref{Bernsteincoef}. More precisely, we set $D_{\ast }=20!\,16^{20}$. For $q\in \{0,1\}$ and $0\leq i\leq 20$, we let
\begin{equation*}
N_{i}^{(q)}=\sum_{j=0}^{20}\rho _{j}5^{j}16^{20-j}\sum_{k=0}^{\min \{i,j\}}
\binom{i}{k}\binom{j}{k}k!\left( 20-k\right) !\,q^{j-k},
\end{equation*}
with the convention $0^{0}=1$. We substitute $\ell =\frac{5q}{16}$ in \eqref{Bernsteincoef} and use $\frac{20!}{\binom{20}{k}}=k!\left( 20-k\right) !$. Then we get $b_{i}\left( \frac{5q}{16}\right) =\frac{N_{i}^{(q)}}{D_{\ast }}$. Table \ref{tabfive} lists positive integers $L_{i}^{(0)}$ and $L_{i}^{(1)}$ such that
\begin{equation}
b_{i}(0)\geq 10^{6}L_{i}^{(0)},\qquad b_{i}\left( \tfrac{5}{16}\right) \geq
10^{6}L_{i}^{(1)},\qquad 0\leq i\leq 20.  \label{Bernsteinlower}
\end{equation}
Each of these forty-two inequalities is the comparison $N_{i}^{(q)}\geq 10^{6}L_{i}^{(q)}D_{\ast }$ between two integers. For example,
\begin{equation*}
b_{0}(0)=261395707500>10^{6}\cdot 261395,
\end{equation*}
and the smallest entry on the second interval is
\begin{equation*}
b_{2}\left( \tfrac{5}{16}\right) =\frac{52280692536923224348362021447375}{
179449926349046518120448}>10^{6}\cdot 291.
\end{equation*}

The Bernstein polynomials $\binom{20}{i}x^{i}\left( 1-x\right) ^{20-i}$ are nonnegative on $\left[ 0,1\right] $. Their sum is equal to one. The smallest entries in the two columns of $L_{i}^{(q)}$ in Table \ref{tabfive} are $2481$ and $291$. Therefore, \eqref{Bernsteinid} and \eqref{Bernsteinlower} give
\begin{equation*}
R_{5}(\nu )\geq 2481\cdot 10^{6}\quad \text{for }0\leq \nu \leq \frac{5}{16},\qquad R_{5}(\nu )\geq 291\cdot 10^{6}\quad \text{for }\frac{5}{16}\leq
\nu \leq \frac{5}{8}.
\end{equation*}
In particular, \eqref{R5} holds on the whole interval $\left[ 0,\frac{5}{8}\right] $. Since $D_{5}>0$ there, we get
\begin{equation*}
\frac{P(\nu )}{D_{5}(\nu )}=\frac{1}{500}+\frac{R_{5}(\nu )}{500D_{5}(\nu )}>
\frac{1}{500}.
\end{equation*}
Together with \eqref{fivequot}, this proves \eqref{fivegap}, with strict inequality for $u\neq 1$. This completes the proof.
\end{proof}

\bigskip

\noindent\textbf{Acknowledgements.} The first author was partly supported by National Natural Science Foundation of China (No. 12622108). The second author was partly supported by a NSERC discovery grant. The third author was partly supported by a grant from the Simons foundation. The authors acknowledge the use of AI tools  for valuable assistance  during the exploratory stage of this project. All mathematical arguments and proofs presented in the final manuscript were independently checked, verified, and written by the authors.

\end{document}